\documentclass{amsart}
\usepackage[margin=1in]{geometry}

\usepackage{booktabs}
\usepackage{microtype}
\usepackage{url}
\usepackage{eepic}
\usepackage[colorlinks]{hyperref}
\hypersetup{citecolor=blue}
\usepackage[all,cmtip]{xy}
\usepackage{etoolbox}
\patchcmd{\section}{\scshape}{\upshape\bfseries}{}{}
\makeatletter
\renewcommand{\@secnumfont}{\bfseries}
\makeatother

\usepackage{amsmath, amstext, amssymb, amsthm, amscd}

\newtheorem{theorem}{Theorem}

\newtheorem{lemma}[theorem]{Lemma}

\numberwithin{equation}{section}

\theoremstyle{definition}

\numberwithin{equation}{section}

\def\arXiv#1{\href{http://arXiv.org/abs/#1}{arXiv:#1}}

\newcommand{\R}{{\mathbb R}}
\newcommand{\C}{{\mathbb C}}
\newcommand{\Z}{{\mathbb Z}}
\newcommand{\Q}{{\mathbb Q}}
\newcommand{\F}{{\mathbb F}}
\newcommand{\Proj}{{\mathbb P}}
\newcommand{\sO}{{\mathcal O}}

\newcommand{\Km}{{\mathrm{Km}}}
\newcommand{\Kum}{{\mathrm{Km}(J(C))}}
\newcommand{\NS}{\mathrm{NS}}
\newcommand{\Aut}{\mathrm{Aut}}

\DeclareMathOperator{\GL}{GL}
\DeclareMathOperator{\PGL}{PGL}
\DeclareMathOperator{\Sp}{Sp}

\title{Elliptic fibrations on a generic Jacobian Kummer surface}

\author{Abhinav Kumar}
\address{Department of Mathematics\\
Massachusetts Institute of Technology\\
Cambridge, MA 02139}
\email{abhinav@math.mit.edu}

\subjclass[2010]{Primary 14J27; Secondary 14J28, 14H45}

\thanks{The author was supported in part by NSF grants DMS-0757765 and
  DMS-0952486, and by a grant from the Solomon Buchsbaum Research
  Fund. He thanks Princeton University for its hospitality during the
  period when this project was started.}

\date{September 23, 2014}

\begin{document}

\begin{abstract}
We describe all the elliptic fibrations with section on the Kummer
surface $X$ of the Jacobian of a very general curve $C$ of genus $2$
over an algebraically closed field of characteristic $0$, modulo the
automorphism group of $X$ and the symmetric group on the Weierstrass
points of $C$. In particular, we compute elliptic parameters and
Weierstrass equations for the $25$ different fibrations and analyze
the reducible fibers and Mordell-Weil lattices. This answers
completely a question posed by Kuwata and Shioda in $2008$.
\end{abstract}

\maketitle

\section{Introduction}

In this paper we will analyze elliptic fibrations on a specific class
of K3 surfaces. In particular, for us a K3 surface $X$ will be a
projective algebraic non-singular surface over a field $k$ of
characteristic $0$, for which $H^1(X, \sO_X) = 0$ and the canonical
bundle $K_X \cong \sO_X$. These are one of the types of surfaces of
Kodaira dimension $0$. Another important class which appears in the
classification of surfaces is that of elliptic surfaces: those
surfaces for which there is a fibration $X \rightarrow \Proj^1$ of
curves of arithmetic genus $1$.

Not all K3 surfaces are elliptic: K3 surfaces equipped with a
polarization of a fixed degree $2n$ are described by a
$19$-dimensional moduli space, and those admitting an elliptic
fibration correspond to points of a countable union of subvarieties of
codimension $1$ in this moduli space.

On the other hand, one may associate a discrete parameter $n =
\chi(\sO_E) = \frac{1}{12}\chi_{top}(E)$ to an elliptic surface
$E$. The case $n = 1$ corresponds to rational elliptic surfaces, $n=2$
to elliptic K3 surfaces, and $n \geq 3$ to ``honestly'' elliptic
surfaces, i.e. those of Kodaira dimension $1$.

It is relatively easy, in principle, to describe elliptic fibrations
on a K3 surface $X$, as follows. The notions of linear, algebraic and
numerical equivalence for divisors are equivalent for a K3 surface,
and the N\'eron-Severi group $\NS(X)$ which describes the algebraic
divisors modulo this equivalence relation is torsion-free and a
lattice under the intersection pairing. It is an even lattice, and has
signature $(1,r - 1)$ for some $1 \leq r \leq 20$, by the Hodge index
theorem. A classical theorem of Piatetski-Shapiro and Shafarevich
\cite[pp. 559--560]{PSS} establishes a correspondence between elliptic
fibrations on $X$ and primitive nef elements of $\NS(X)$ of
self-intersection $0$. We choose the term elliptic divisor (class) to
describe an effective divisor (class) of self-intersection zero.

Another geometrical property of a K3 surface over $\C$ which is
explicitly computable, in principle, is its automorphism group. Here
the Torelli theorem of Piatetski-Shapiro and Shafarevich \cite{PSS}
for algebraic K3 surfaces stipulates that an element of $\Aut(X)$ is
completely specified by its action on $H^2(X,\Z)$. Conversely, an
automorphism of the lattice $H^2(X,\Z)$ which respects the Hodge
decomposition and preserves the ample or K\"ahler cone (alternatively,
preserves the classes of effective divisors) gives rise to an
automorphism of $X$.

A theorem of Sterk \cite{St} guarantees that for a K3 surface $X$ over
$\C$, and for any even integer $d \geq -2$, there are only finitely
many divisor classes of self-intersection $d$ modulo $\Aut(X)$. In
particular, there are only finitely many elliptic fibrations modulo
the automorphism group. In practice, however, it may be quite
difficult to compute all the elliptic fibrations for any given K3
surface (for example, as a double cover of $\Proj^2$, given by $y^2 =
f(x_0, x_1, x_2)$, or as a quartic surface in $\Proj^3$, given by
$f(x_0, x_1, x_2, x_3) = 0$). This is because it might not be easy to
compute the N\'eron-Severi group, or to understand $NS(X) \subset
H^2(X,\Z)$ well enough to compute the automorphism group, or even
knowing these two, to compute all the distinct orbits of $\Aut(X)$ on
the elliptic divisors.

In this paper, we carry out this computation for a case of geometric
and arithmetic interest: that of a (desingularized) Kummer surface of
a Jacobian of a genus $2$ curve $C$ whose Jacobian $J(C)$ has no extra
endomorphisms. In this paper, we will reserve the term ``generic
Jacobian Kummer surface'' and the label $\Kum$ for such a Kummer
surface. We also focus on elliptic fibrations with a section (also
called Jacobian elliptic fibrations, though we will not use this term
since ``Jacobian'' is being used in a different context), although the
calculations in the first part of the paper in principle will also
allow us to describe the genus $1$ fibrations without a section. Our
main theorem states that there are essentially twenty-five different
elliptic fibrations with section on a generic Jacobian Kummer
surface. Furthermore, we compute elliptic parameters and Weierstrass
equations for each of these elliptic fibrations, describe the
reducible fibers, and give a basis of the sections generating the
Mordell-Weil lattice in each case. Apart from solving a classification
problem, we expect that these results will facilitate further
calculations of Kummer surfaces and associated K3 surfaces and their
moduli spaces. Another application is in the construction of elliptic
surfaces and curves of high rank. For instance, the very last elliptic
fibration described in this paper gives a family of elliptic K3
surfaces of Mordell-Weil rank $6$ over the moduli space of genus $2$
curves with full $2$-level structure. In fact, this family descends to
the moduli space $\mathcal{M}_2$ of genus $2$ curves without $2$-level
structure, giving a family of elliptic K3 surfaces of geometric
Mordell-Weil rank $6$. Finally, many of the equations described in
this paper give new models of Kummer surfaces as double plane
sextics. Note that though these calculations are carried out for a
very general curve of genus $2$, in most cases the same or analogous
calculation may be carried out for special genus $2$ curves, resulting
in the same Weierstrass equations. The presence of extra endomorphisms
will manifest itself in an altered configuration of reducible fibers,
or by a jump in the Mordell-Weil rank. We also note that although we
work in characteristic zero, the classification result should also
hold (with minor modification to the formulas given here) in odd
positive characteristic: since the K3 surfaces considered are
non-supersingular (by the genericity assumption), they lift to
characteristic zero by Deligne's lifting result \cite{D}.

We now mention some related results. Recall that the two types of
principally polarized abelian surfaces are Jacobians of genus $2$, and
products of two elliptic curves. A natural question is to compute all
the elliptic fibrations on the Kummer surfaces of such abelian
surfaces, subject to genericity assumptions. Oguiso \cite{O}
classified the different elliptic fibrations (modulo the automorphism
group) on $\Km(E_1 \times E_2)$, when $E_1$ and $E_2$ are
non-isogenous elliptic curves. Nishiyama \cite{N} obtained the result
by a more algebraic method, and also extended it to the Kummer surface
of the product of an elliptic curve with itself. Keum and Kondo
\cite{KK} computed the automorphism group of such a Kummer
surface. Kuwata and Shioda \cite{KS} have more recently computed
elliptic parameters and Weierstrass equations for $\Km(E_1 \times
E_2)$. They asked the question which motivated this paper.

Keum \cite{Ke1} found an elliptic fibration on $\Kum$ with $D_6 A_3
A_1^6$ reducible fibers, and Mordell-Weil group $(\Z/2\Z)^2$. Shioda
\cite{Sh1} found an elliptic fibration on $\Kum$ with $D_4^2 A_1^6$
reducible fibers, and Mordell-Weil group $\Z \oplus (\Z/2\Z)^2$. In
previous work \cite{Ku}, the author found a further elliptic fibration
with $D_9A_1^6$ reducible fibers and Mordell-Weil group $\Z/2\Z$. The
automorphism group of $\Kum$ was computed by Kondo \cite{Ko}, building
on work of Keum \cite{Ke2}.

\section{Kummer surface}\label{kumsurf}

Let $C$ be a curve of genus $2$ over an algebraically closed field $k$
of characteristic different from $2$. We review some of the classical
geometry of the surface $\Kum$ (see \cite{Hud} for more background).

The linear system divisor $2 \Theta$, where $\Theta$ is an embedding
of $C$ in $J(C)$ defining the principal polarization, maps the
Jacobian of $C$ to $\Proj^3$ by a map of degree $2$. The image is a
quartic surface with $16$ singularities, the singular model of the
Kummer surface of $J(C)$. The $16$ singular points are ordinary double
points on the quartic surface, and are called {\bf nodes}. The node
$n_0$ comes from the identity or $0$ point of the Jacobian, whereas
the node $n_{ij}$ comes from the $2$-torsion point which is the
difference of divisors $[(\theta_i,0)] - [(\theta_j, 0)]$
corresponding to two distinct Weierstrass points on $C$. The nodes
$n_{0i}$ are usually abbreviated $n_i$.

There are also sixteen hyperplanes in $\Proj^3$ which are tangent to
the Kummer quartic. These are called {\bf tropes}. Each trope
intersects the quartic in a conic with multiplicity $2$ (this conic is
also called a trope, by abuse of notation), and contains $6$
nodes. Conversely, each node is contained in exactly $6$ tropes. This
beautiful configuration is called the $(16,6)$ Kummer configuration.

The N\'eron-Severi lattice of the non-singular Kummer contains classes
of rational curves $N_i$ and $N_{ij}$ which are $-2$-curves obtained
by blowing up the nodes of the singular Kummer, and also the classes
of the $T_i$ and $T_{ij}$ coming from the tropes. Henceforth we will
use the terminology $N_\mu$ to denote the class of a node, and
similarly $T_\mu$ for the class of a trope. We will denote the lattice
generated by these as $\Lambda_{(16,6)}$ (it is abstractly isomorphic
to $\langle 4 \rangle \oplus D_8(-1)^2$ \cite{Ke2}). It has signature
$(1,16)$ and discriminant $2^6$ and is always contained primitively in
the N\'eron-Severi lattice of a Kummer surface of a curve of genus
$2$. For a very general point of the moduli space of genus $2$ curves
or of principally polarized abelian surface (i.e. away from a
countable union of hypersurfaces), the N\'eron-Severi lattice of the
Kummer surface is actually isomorphic to $\Lambda_{(16,6)}$.

Let $H$ be the class of a non-singular hyperplane section of the
singular Kummer surface in $\Proj^3$. We have the following relations
in the N\'eron-Severi lattice of the Kummer surface $X$:
\begin{align*}
T_0 &= \frac{1}{2}\left(H - \sum\limits_{i=0}^5N_i\right) , \qquad T_i = \frac{1}{2}\left(H - N_i - \sum\limits_{j \neq i} N_{ij}\right) \\
T_{ij} &= \frac{1}{2}\left(H - N_i - N_j - N_{ij} - N_{lm} - N_{mn} - N_{ln}\right)
\end{align*}
where $\{l,m,n\}$ is the complementary set to $\{0,i,j\}$ in
$\{0,1,2,3,4,5\}$.

From these and the inner products
\begin{align*}
H^2 &= 4 & N_\mu^2 &= -2 \\
H\cdot N_\mu &= 0 & N_\mu \cdot N_\nu &= 0 \text{ for } \mu \neq \nu
\end{align*}
we can easily compute all the inner products between the classes of
the nodes and the tropes.

\section{Elliptic fibrations on a generic Kummer surface}

\subsection{Automorphism group}

We recall here some facts about the automorphism group of the generic
Jacobian Kummer surface. A few of these are classical \cite{Hud, Hut,
  Kl}, though the setup we will need is much more recent, coming from
work of Keum \cite{Ke2} and Kondo \cite{Ko}.

The classical automorphisms are the following:
\begin{enumerate}
\item Sixteen translations: The translations by $2$-torsion points on
  the abelian surface give rise to involutions on the Kummer
  surface. These are linear i.e. induced by elements of $\PGL_4(k)$ on
  the singular Kummer surface in $\Proj^3$.
\item Sixteen projections: Projecting the singular quartic Kummer
  surface $Y$ from any of the nodes $n_\alpha$ gives rise to a double
  cover $Y \rightarrow \Proj^2$. The interchange of sheets gives rise
  to an involution of the non-singular Kummer surface $X$, which we
  call ``projection'' $p_\alpha$ by abuse of notation.
\item Switch: The dual surface to $X$ is also a quartic surface which
  is projectively equivalent to $X$. Hence there is an involution of
  $X$ which switches the classes of the nodes and the tropes. We let
  $\sigma$ be the switch, which takes $N_\alpha$ to $T_\alpha$ and
  vice versa.
\item Sixteen correlations: These are obtained from projections of the
  dual quartic surface from its nodes. They have the form $\sigma
  \circ p_\alpha \circ \sigma$.
\item Sixty Cremona transformations called G\"opel tetrads. These
  correspond to sets of four nodes such that no three of them lie on a
  trope. There are sixty G\"opel tetrads. For each of them, there is a
  corresponding equation of the Kummer surface of the form
\begin{align*}
& A(x^2t^2 + y^2z^2) + B(y^2t^2 + x^2z^2) + C(z^2t^2 + x^2y^2) + Dxyzt \\
&+ F(yt+zx)(zt+xy) + G(zt+xy)(xt+yz) + H(xt+yz)(yt+xz) = 0,
\end{align*}
and the Cremona transformation is given by $(x,y,z,t) \mapsto
(yzt,ztx,txy,xyz)$.
\end{enumerate}
Finally, Keum \cite{Ke2} defined $192$ new automorphisms in 1997, and
shortly afterward Kondo \cite{Ko} showed that these together with the
classical automorphisms generate the automorphism group of $X$. Note
that Kondo's definition of Keum's automorphisms is slightly different
from the automorphisms defined by Keum (the two are related to each
other by multiplication by a classical automorphism). Here, we will
use Kondo's notation, as it seems more convenient. However, we will
not explicitly define Keum's automorphisms here.

Let $M$ be any even hyperbolic lattice. That is, $\langle x, x \rangle
\in 2\Z$ for $x \in M$ (which implies $\langle x, y \rangle \in \Z$
for all $x, y \in M$) and the form $\langle \, , \rangle$ on $M$ has
signature $(1,r)$ for some positive integer $r$. We say $\delta \in M$
is a {\bf root} if $\langle x, x \rangle = -2$. For each root
$\delta$, there is an isometry $s_\delta$ of $M$ defined by reflection
in the hyperplane $H_\delta$ orthogonal to $\delta$:
$$
s_\delta(x) = x + \langle x, \delta \rangle \,\delta.
$$
We now recall the general structure of the orthogonal group
$O(M)$. Let $P(M)$ be a connected component of the set
$$
\{ x \in M \otimes \R \, | \, \langle x, x \rangle > 0 \},
$$
and $W(M)$, the Weyl group of $M$, be the subgroup of $O(M)$ generated
by all the reflections in the roots of $M$. Then the hyperplanes
$H_\delta$ divide $P(M)$ into regions called chambers: these are the
connected components of $P(M) \backslash \left(\bigcup H_\delta
\right)$. These are permuted simply transitively by the Weyl group. If
we fix a chamber $D(M)$ and any interior point $\rho \in D(M)$, we
obtain a partition of the roots $\Delta = \Delta^+ \bigsqcup \Delta^-$
into positive and negative roots with respect to $\rho$ or $D(M)$. Let
$G(M)$ be the group of (affine) symmetries of $D(M)$. Then $O(M)$ is a
split extension of $\{\pm 1\} \cdot W(M)$ by $G(M)$ (see \cite{V} for
details).

The lattice $S = NS(X)$ can be primitively embedded into the even
unimodular lattice $L$ of signature $(1,25)$, which is the direct sum
$\Lambda \oplus U$ of the Leech lattice and a hyperbolic plane, such
that the orthogonal complement $T$ of $S$ in $L$ has its root
sublattice isomorphic to $A_3(-1) \oplus A_1(-1)^6$. There is an
explicit description of a convex fundamental domain $D$ in $L \otimes
\R$ for the reflection group of $L$ due to Conway \cite{Conw}.
Following Borcherds \cite{B}, let $D' = D \cap P(S)$ where $P(S)$ is a
connected component of the set $\{x \in S \otimes \R \,| \, \langle x,
x \rangle > 0\}$. In fact, we will choose $D'$ so that it contains the
Weyl vector $(\underline{0},0,1) \in \Lambda_{24} \oplus U$.

The region $D'$ is defined by finitely many inequalities. More
precisely, to each of the nodes and tropes, the projections and
correlations, Cremona transformations and Keum's automorphisms, Lemma
5.1 of \cite{Ko} associates a Leech root $r \in L$. Let $r'$ be its
orthogonal projection to $S \otimes \Q$.  Then
$$
D' = \{ x \in S \otimes \R \,|\, \langle x, r' \rangle \geq 0 \}
$$ 
So $D'$ is bounded by $16 + 16 + 16 + 16 + 60 + 192 = 316$
hyperplanes, which can be written down explicitly.

Now recall that
$$
D(S) = \{x \in P(S) \,| \, \langle x, \delta \rangle > 0 \textrm{ for any nodal class } \delta\}
$$
is the ample or K\"ahler cone of $X$ and a fundamental domain with
respect to the Weyl group $W(S)$. Let $O(S)$ be the orthogonal group
of the lattice $S$ and
$$
\Gamma(S) = \{g \in O(S) \,|\, g^* _{|_{A_{S}}} = \pm 1\}
$$
where $A_{S} = S^*/S \cong (\Z/2\Z)^4 \oplus \Z/4\Z$ is the discriminant group of $S$. Then Theorem 4.1 of Keum \cite{Ke2} shows that
$$
\Aut(X) \cong \{g \in \Gamma(S) \, | \, g \textrm{ preserves the K\"ahler cone} \} \cong \Gamma(S)/W(S).
$$

We note that the automorphism group $O(q_S)$ of the discriminant form
$q_S$ on the discriminant group $A_S$ is isomorphic to an extension of
$\Z/2\Z$ by $\Sp_4(\F_2) \cong S_6$.

We have that $D' \subset D(S)$, and in fact $D(S)$ contains an
infinite number of translates of $D'$. Let $N$ be the subgroup of
$\Aut(X)$ generated by the projections, correlations, Cremona
transformations and Keum's automorphisms. We will also think of $N$ as
a subgroup of $O(S)$. Also, let $K$ be the finite group of $O(S)$
generated by the translations, switch, and the permutations of the
Weierstrass points. Then $K \cong \Aut(D') \cong \Z/2\Z \times
(\Z/2\Z)^4 \rtimes \Sp_4(\mathbb{F}_2) \cong (\Z/2\Z)^5 \rtimes S_6$, a
group of order $6! \cdot 2^5$. It is shown in Kondo's paper (this
follows directly from Lemma 7.3) that $N$ and $K$ together generate
the following group $O(S)^+$, which will play an important role.
$$
O(S)^{+} = \{ g \in O(S) \, | \, g \textrm{ preserves the K\"ahler cone}\}.
$$
The above groups fit into the following commutative diagram, with
exact rows and columns.

$$
\begin{CD}
    @.          @.     1        @.      1      @.  \\
 @.   @.         @VVV           @VVV       @. \\
 1  @>>> W(S)  @>>> \Gamma(S) @ >>> Aut(X)  @>>> 1 \\
 @.   @|         @VVV           @VVV       @. \\
 1  @>>> W(S)  @>>> O(S) @ >>> O(S)^+  @>>> 1 \\
 @.   @.         @VVV           @VVV       @. \\
   @.   @. O(q_S)/\{\pm 1\} @= O(q_S)/\{\pm 1\} @.\\
 @.   @.         @VVV           @VVV       @. \\
    @.          @.     1        @.      1      @.  \\
\end{CD}
$$

with $O(q_S)/\{\pm 1\} \cong \Sp_4(\F_2) \cong S_6$.

It follows that $D'$ is like a fundamental domain for the action of
$N$ on $D(S)$, in a sense made precise by the following lemma.

\begin{lemma}\label{fund}
Let $v \in S$. Then there exists $g \in N$ such that $g(v) \in \overline{D'}$.
\end{lemma}
\begin{proof}
The proof is analogous to that of Lemma 7.3 of \cite{Ko}.  Let $w'$ be
the projection of the Weyl vector $w$, so $w' = 2H - \sum
N_\mu/2$. Choose $g \in N$ such that $\langle g(v), w' \rangle$ is
mimimum possible. For of the chosen generators of $N$ (i.e. a
projection, correlation, Cremona transformation or Keum's
automorphism), let $r$ be the associated Leech root as above and let
$\iota_r$ be the isometry of $S$ induced by this automorphism. Let
$r'$ be the projection to $S \otimes \Q$ of $r$.

We have
$$
\langle g(v), w' \rangle  \leq \langle \iota_rg(v), w' \rangle = \langle g(v), \iota_r^{-1}w' \rangle,
$$
where the inequality follows from the choice of $g$. Now, if $\iota_r$
corresponds to a projection or correlation, then $\iota_r$ is merely
reflection in $r'$. Therefore,
$$
\iota_r^{-1}(w') = \iota_r(w') = w' + 2\langle w', r' \rangle r'.
$$
Therefore
$$
\langle g(v), w' \rangle \leq \langle g(v), w' \rangle + 2 \langle g(v), r' \rangle \langle w', r' \rangle.
$$
Since $\langle w', r' \rangle > 0$, we obtain $\langle g(v), r' \rangle \geq 0$.

Next, if $\iota_r$ corresponds to a Cremona transformation
$c_{\alpha,\beta,\gamma,\delta}$ corresponding to a G\"opel tetrad,
then $\iota_r = r_{\alpha,\beta,\gamma,\delta} \circ
t_{\alpha,\beta,\gamma,\delta}$, where
$r_{\alpha,\beta,\gamma,\delta}$ is a reflection with respect to the
corresponding $r'$ and $t_{\alpha,\beta,\gamma,\delta}$ fixes
$H$,$N_\alpha$, $N_\beta$, $N_\gamma$, $N_\delta$ and permutes the
remaining nodes in a specific manner that we shall not need to
describe here.

It follows that
$$
\iota_r^{-1}(w') = \iota_r(w') = r_{\alpha,\beta,\gamma,\delta} \circ t_{\alpha,\beta,\gamma,\delta} (w') = r_{\alpha,\beta,\gamma,\delta} (w') = w' + 2\langle w', r' \rangle r'
$$
where the middle equality follows because $w' = 2H - \sum N_\mu/2$ is
symmetric in the nodes.

As before we obtain $\langle g(v), r' \rangle \geq 0$.

Finally, if $\iota_r$ corresponds to a Keum's automorphism, then by an
explicit calculation in \cite{Ko}, we have
$$
\iota_r^{-1}(w') = w' + 8 \sigma(r').
$$
Hence, we obtain
$$
\langle g(v), w' \rangle \leq \langle g(v), w' \rangle + 8 \langle g(v), \sigma(r') \rangle
$$
or $\langle g(v), \sigma(r') \rangle \geq 0$. As $r'$ ranges over the
$192$ vectors corresponding to Keum's automorphisms, so does
$\sigma(r')$. Therefore, we must have $\langle g(v), r' \rangle \geq
0$ for all $r'$, and therefore $g(v) \in \overline{D'}$ by definition.
\end{proof}

The calculations of \cite{Ko} show that $N$ is a normal subgroup of
$O(S)^+$, that $K \cap N = \{1\}$, and hence $O(S)^{+}$ is a split
extension of $K$ by $N$. Therefore we have an exact sequence
$$
1 \rightarrow N \rightarrow O(S)^+ \rightarrow \Aut(D') \rightarrow 1.
$$
Now we have $\Aut(X) \cong \{g \in O(S)^+ | g^* _{|_{A_S}} = \pm
1\}$. We can check that every element of $N$ acts trivially on the
discriminant group $A_S$. Therefore $\Aut(X)$ is the inverse image of
the subgroup of $K \cong \Aut(D')$ which acts trivially on the
discriminant group. This is easily seen to be the $2$-elementary
abelian group generated by the translations and the switch. Hence we
see that the automorphism group of $X$ is a split extension of this
elementary $2$-abelian group by $N$.

In what follows, we will actually work with the entire group $O(S)^+$,
since we would like to compute the elliptic fibrations modulo the
action of $S_6$ as well. We will see that including this extra
symmetry reduces the number of elliptic fibrations to a more
manageable number, namely $25$, and does not cause any loss of
information, as we may find all the elliptic fibrations by just
computing the orbits of these $25$ under $S_6$.

\subsection{Elliptic divisors and a sixteen-dimensional polytope}\label{polytope}

Following Keum, we will construct a polytope in $\R^{16}$ associated
to the domain $D'$ in Lorentzian space, whose vertices will give us
information about the elliptic divisors. Writing elements of $S
\otimes \R$ as $\beta H - \alpha_0 N_0 -\dots - \alpha_{45} N_{45}$,
we see that the quadratic form is $4 \beta^2 - 2\sum \alpha_\mu^2$.

If $F$ is the class of a fiber for a genus $1$ fibration, then $F^2 =
0$, and $F$ is nef. Conversely, if $F^2 = 0$ for a divisor (class) on
a K3 surface, then an easy application of the Riemann-Roch theorem
shows that $F$ or $-F$ must be an effective divisor class. We may move
$F$ to the nef cone by an element of $\{\pm 1\}\cdot W(S)$, and then a
result from \cite{PSS} establishes that the linear system defined by
$F$ gives rise to a genus $1$ fibration $X \rightarrow \Proj^1$. Of
course, if $F \in \overline{D(S)}$, then we don't need to move $F$ by
an element of the Weyl group at all.

The condition $F^2 = 0 $ translates to $\sum (\alpha_\mu / \beta)^2 =
2$.  If furthermore we want the fibration to be an elliptic fibration
(i.e. to have a section), there must exist an effective divisor $O \in
S$ such that $O \cdot F = 1$. If two divisors $F$ and $F'$ are related
by an element $g \in \Aut(X)$, then it is clear that they define the
``same'' elliptic fibration, namely there is a commutative diagram
$$
\begin{CD}
X @>g>> X \\
@VV{\pi'}V @VV{\pi}V \\
\Proj^1 @>h>> \Proj^1
\end{CD}
$$
where $\pi, \pi'$ are elliptic fibrations defined by $F,F'$
respectively, and $h$ is an automorphism of $\Proj^1$. So now we may
move $F \in \overline{D(S)}$ to $\overline{D'}$, by applying an
element of $N$ and using Lemma \ref{fund} above.

We map $D'$ to a polytope $\mathcal{P}$ in $\R^{16}$ by taking $\beta
H - \sum \alpha_\mu N_\mu$ to $(x_\mu)_\mu =
\big(\frac{\alpha_\mu}{\beta}\big)_\mu$. This map just normalizes the
coefficient of $H$ to $1$ in the expression for a divisor. Therefore
the image $\mathcal{P}$ of $\overline{D'}$ is the intersection of a
set of $32 + 32 + 60 + 192 = 316$ half-spaces of the form $\sum c_\mu
x_\mu \leq d$. An argument of Keum \cite{Ke2} shows that the polytope
$\mathcal{P}$ must be contained in the ball of radius $\sqrt{2}$, for
otherwise there would be infinitely many rational points of norm $2$
in $\mathcal{P}$, giving rise to infinitely many inequivalent elliptic
fibrations and contradicting Sterk's theorem.

In any case, we can enumerate all the vertices of $\mathcal{P}$ using
the program SymPol \cite{RS}, or linear programming software such as
\verb,glpk,. Here symmetry plays an important role, as we can restrict
ourselves to computing the vertices modulo the symmetries of $D'$,
which reduces the number by a factor of approximately $32 \cdot 720 =
23040$. The details of these calculations are described in Appendix B,
which also describes how to verify that these calculations are
correct; this verification is much simpler than finding the vertices
in the first place! We find that the polytope $\mathcal{P}$ has $1492$
vertices modulo the symmetries of $D'$ (note that the symmetries of
$D'$ are not always affine linear symmetries of $\mathcal{P}$, even
though they are linear transformations in $\R^{17}$). Out of these
$54$ vertices have norm $2$, giving rise to at most $54$ distinct
elliptic divisors modulo $N \rtimes \Aut(D') = O(S)^+$.

Recall that the subgroup $\Aut(X)$ of $O(S)^+$ is normal, the quotient
being isomorphic to $S_6$ (and in fact, there is a complementary
subgroup, induced by the permutations of the Weierstrass points). This
is because within $\Aut(D')$, the subgroup $(\Z/2\Z)^5$ generated by
translations and a switch is normalized by its complementary subgroup
$S_6$, and is therefore normal. Hence the inverse image $\Aut(X)$ of
$(\Z/2\Z)^5$ under $O(S)^+ \rightarrow \Aut(D')$ is normal in
$O(S)^+$. Here we will just compute the elliptic divisors up to the
action of $O(S)^+$. This set turns out to be smaller than the set of
vertices of $\mathcal{P}$ of norm $2$ modulo its symmetries, because
there are some extra symmetries coming from $N$ which connect some of
the vertices.

The following table lists the elliptic divisors which are obtained
from the vertices of $\mathcal{P}$, up to symmetry. We use the
shorthand notation
$$
(c_0, c_1, \dots, c_5, c_{12} , \dots c_{15}, c_{23}, \dots, c_{25}, c_{34}, c_{35}, c_{45}, d)
$$ to represent the divisor class $\sum c_{\mu} N_{\mu} + dH$. For a
given elliptic divisor $D$, it is easy to check if it gives an
elliptic fibration with a section: we check that the divisor $D$ is
primitive in $S$, and then check whether there is an element of $S$
whose intersection with $D$ is $1$. If this condition fails, then $D$
will only provide a genus one fibration.

$$
\begin{array}{lcc}
\textrm{Elliptic divisor} & \textrm{Section} & \textrm{Number} \\
\hline \hline
(-1,-1,0,0,0,0,0,0,0,0,0,0,0,0,0,0,1) & \textrm{Yes} & 1\\
(-2, -1, 0, 0, 0, -1, 0, 0, 0, 0, 0, -1, -1, 0, 0, 0, 2) & \textrm{Yes} & 2 \\
(-2, -1, 0, 0, -1, 0, 0, 0, 0, 0, -1, 0, -1, 0, 0, 0, 2) & \textrm{Yes} & 3\\
(-2, -1, 0, 0, 0, 0, 0, 0, 0, 0, -1, -1, 0, -1, 0, 0, 2) & \textrm{Yes} & 4\\
(-2, -1, 0, 0, 0, -1, -1, 0, 0, 0, 0, 0, -1, 0, 0, 0, 2) & \textrm{No}\\
(-3, -3, -1, -1, 0, 0, -1, -1, 0, 0, -1, -2, 0, -2, 0, -1, 4)/2 & \textrm{Yes} & 5\\
(-4, -2, -1, 0, 0, -1, -1, 0, 0, -1, -1, -1, 0, -2, -1, -1, 4)/2 & \textrm{Yes} & 6 \\
(0, -1, 0, -1, -1, -1, 0, -1, 0, 0, -1, 0, -1, -1, 0, 0, 2) & \textrm{No} \\
(-1, -1, 0, 0, 0, 0, 0, 0, 0, 0, -1, -1, -1, -1, -1, -1, 2)/2 & \textrm{Yes} & 7 \\
(-3, -2, -1, 0, 0, 0, 0, 0, -1, 0, -1, -1, 0, -1, 0, 0, 3) & \textrm{Yes} & 8 \\
(-2, -2, 0, 0, 0, -1, 0, 0, 0, -1, 0, -2, -2, 0, 0, 0, 3) & \textrm{Yes} & 8A\\
(-3, -2, -1, 0, 0, 0, 0, 0, 0, -1, -1, -1, 0, -1, 0, 0, 3) & \textrm{Yes} & 9\\
(-2, -2, -1, 0, 0, -1, 0, -1, -1, 0, -2, 0, -1, -1, 0, 0, 3)/2 & \textrm{Yes} & 10\\
(-2, -1, -1, -1, -1, 0, -1, 0, 0, 0, 0, 0, -2, -1, 0, -2, 3) & \textrm{No} \\
(0, -1, -1, 0, 0, 0, -1, -1, -1, 0, -2, -2, 0, -1, 0, -2, 3) & \textrm{No} \\
(-3, -2, 0, 0, -1, 0, 0, 0, 0, -1, -1, -1, -1, 0, 0, 0, 3) & \textrm{Yes} & 11 \\
(-3, -2, 0, 0, 0, -1, 0, 0, 0, -1, -1, -1, -1, 0, 0, 0, 3) & \textrm{No} \\
(-2, -2, 0, 0, 0, -2, 0, -1, 0, 0, 0, -1, -2, 0, 0, 0, 3) & \textrm{Yes} & 11A \\
(-5, -4, 0, 0, 0, -3, -1, -1, -1, 0, -1, -3, -2, -1, 0, -2, 6)/2 & \textrm{Yes} & 12\\
(-6, -3, 0, 0, -3, 0, -1, -1, 0, -1, -2, -1, -2, -1, -2, -1, 6)/2 & \textrm{Yes} & 13\\
(0, 0, 0, 0, 0, 0, -1, -1, -1, -1, -2, -2, -1, -1, 0, -2, 3) & \textrm{No} \\
(-3, -1, 0, 0, -1, -1, -1, -1, 0, 0, 0, -1, -1, -1, -1, 0, 3)/2 & \textrm{Yes} & 14 \\
(-4, -2, 0, 0, 0, -2, 0, 0, -1, -1, 0, -2, -1, -1, 0, 0, 4) & \textrm{Yes} & 15
\end{array}
$$
$$
\begin{array}{lcc}
\textrm{Elliptic divisor} & \textrm{Section} & \textrm{Number} \\
\hline \hline
(-3, -2, 0, 0, -1, -2, 0, -1, 0, 0, 0, -2, -3, 0, 0, 0, 4) & \textrm{No} \\
(-3, -3, 0, 0, 0, -2, 0, -1, 0, -1, 0, -2, -2, 0, 0, 0, 4) & \textrm{Yes} & 15A\\
(-4, -2, 0, 0, 0, -2, -1, 0, -1, 0, 0, -2, -1, -1, 0, 0, 4) & \textrm{Yes}  & 16 \\
(-3, -3, 0, 0, 0, -2, 0, 0, 0, 0, -1, -2, -2, 0, 0, -1, 4) & \textrm{No}\\
(-4, -2, 0, 0, -1, -1, 0, -1, 0, 0, 0, -2, -2, 0, 0, -1, 4) & \textrm{Yes} & 17 \\
(-3, -3, 0, 0, 0, -2, -1, -1, 0, 0, 0, -2, -2, 0, 0, 0, 4) & \textrm{Yes} & 18\\
(-4, -2, 0, 0, 0, -2, 0, 0, 0, 0, -1, -2, -1, -1, 0, -1, 4)  & \textrm{No}\\
(-4, -2, 0, 0, 0, -2, 0, 0, -1, 0, -1, -2, -1, -1, 0, 0, 4)  & \textrm{Yes} & 18A\\
(-3, -3, -2, 0, 0, 0, 0, 0, 0, -2, -1, -1, -1, -1, -1, -1, 4)/2 & \textrm{Yes} & 19 \\
(-3, -3, 0, 0, 0, -2, -1, 0, 0, 0, 0, -2, -2, 0, 0, -1, 4) & \textrm{Yes} & 16A \\
(-4, -1, 0, -1, 0, -2, 0, 0, -1, 0, 0, -2, -2, -1, 0, 0, 4) & \textrm{No} \\
(-4, -2, 0, 0, 0, -2, -1, 0, -1, 0, 0, -2, -1, 0, 0, -1, 4) & \textrm{No} \\
(-4, -4, 0, 0, -2, 0, 0, -1, 0, -1, -2, -2, -2, 0, 0, 0, 5) & \textrm{Yes} & 20A \\
(-5, -2, 0, -1, 0, -2, 0, 0, -1, -1, 0, -3, -2, -1, 0, 0, 5) & \textrm{Yes} & 20 \\
(-1, 0, -2, 0, 0, -1, -3, -5, 0, -1, 0, -2, -2, 0, 0, -1, 5) & \textrm{No} \\
(-4, -3, 0, 0, -1, -2, 0, -2, -1, 0, -1, -2, -3, 0, -1, 0, 5)/2 & \textrm{Yes} & 21 \\
(-4, -3, 0, 0, -1, -2, 0, -1, 0, -1, 0, -3, -3, 0, 0, 0, 5)  & \textrm{Yes} & 20B\\
(-5, -2, 0, -1, 0, -2, 0, -1, -2, 0, -1, -2, -2, -1, -1, 0, 5)/2 & \textrm{Yes} & 22 \\
(-5, -2, 0, 0, -2, -1, 0, -1, -1, 0, -3, 0, -2, -1, 0, 0, 5) & \textrm{Yes} & 23\\
(-1, -3, 0, 0, 0, -2, -4, -1, 0, -5, 0, -4, -5, 0, 0, -1, 7) & \textrm{No}\\
(-3, -5, -3, -2, 0, -1, 0, 0, 0, -6, 0, 0, 0, -2, -1, -3, 7) & \textrm{No}\\
(-1, -1, -3, -2, -3, -4, 0, -1, 0, -7, -2, 0, -2, 0, 0, 0, 7) & \textrm{No} \\
(-3, -4, 0, -5, -2, 0, -1, 0, 0, 0, -1, 0, -1, -5, 0, -4, 7) & \textrm{No}\\
(-1, 0, -6, -3, -3, -1, 0, -2, 0, -3, -2, 0, -5, 0, 0, 0, 7) & \textrm{No}\\
(-6, -5, 0, 0, -3, 0, 0, -1, 0, -2, -3, -3, -2, -1, 0, 0, 7) & \textrm{Yes} & 24\\
(-2, 0, 0, -1, -1, -2, -7, -3, 0, -2, -1, 0, 0, -3, -4, 0, 7) & \textrm{No}\\
(-4, -5, 0, 0, 0, -5, 0, -3, -1, 0, 0, -2, -4, -1, -1, 0, 7) & \textrm{Yes} & 24A \\
(-2, -1, -7, 0, 0, -4, 0, -2, 0, -3, -1, 0, 0, -1, -2, -3, 7) & \textrm{No}\\
(-6, -4, 0, 0, -2, -2, 0, -2, -1, -1, 0, -4, -4, 0, 0, 0, 7) & \textrm{Yes} & 24B\\
(-3, -2, -4, 0, -3, -1, 0, -1, 0, -1, -7, 0, 0, 0, -2, -2, 7) & \textrm{Yes} & 24C\\
(-7, -5, 0, -2, -2, 0, 0, 0, 0, -4, -3, -3, -3, -1, -1, -1, 8)/2 & \textrm{Yes} & 25
\end{array}
$$

Note that the elliptic divisors we will ultimately use to compute with
each fibration are not always the ones listed above, but are related
by an element of $O(S)^+$. Fibrations with the same numerical label,
such as $8$ and $8A$, are equivalent, and we show this in the
analysis. The fibrations with different numerical labels are
inequivalent: they are easily distinguished by having different
configurations of reducible fibers.

\subsection{Results}

The analysis of the previous section leads to a (computer-assisted)
proof of our main theorem.

\begin{theorem}\label{mainthm}
There are exactly $25$ distinct elliptic fibrations with section on a
generic Jacobian Kummer surface $X = \Kum$, for a genus $2$ curve $C$
over an algebraically closed field of characteristic $0$, modulo the
action of $\Aut(X)$ and the permutations $S_6$ of the Weierstrass
points of $C$.
\end{theorem}

We list the Kodaira-N\'eron types of the reducible fibers, the torsion
subgroups and Mordell-Weil lattices of each of the twenty-five
fibrations in the table below. In the remainder of the paper, we will
analyze each of the fibrations. For background on elliptic surfaces,
we refer the reader to \cite{Sh2}.

Let $A_n, D_n, E_6, E_7, E_8$ be the usual root lattices, and for a
lattice $M$, let $M^*$ be its dual. The symbol $\langle \alpha
\rangle$ denotes the one-dimensional lattice generated by an element
of norm $\alpha$.

Let $P_n$ be the lattice whose Gram matrix is the $n \times n$ matrix
with $3$'s on the diagonal and $1$'s off the diagonal. The minimal
norm of $P_n$ is $3$, and its discriminant (the square of its
covolume) is easily seen to be $2^{n-1}(n+2)$.  It is an easy exercise
to check that the inverse of this Gram matrix (i.e. a Gram matrix for
the dual lattice $P_n^*$) is $1/(2n + 4)$ times the matrix which has
$n+1$'s on the diagonal and $-1$'s off the diagonal.

$$
\begin{array}{lccc}
\textrm{Fibration} & \textrm{Reducible Fibers} & \textrm{Torsion subgroup} & \textrm{Mordell-Weil lattice} \\
\hline \hline
1 & D_4^2 A_1^6 & \Z/2\Z \oplus \Z/2\Z  & \langle 1 \rangle \\
2 & D_6 D_4 A_1^4 & \Z/2\Z & \langle 1 \rangle \\
3 & D_6 A_3 A_1^6 & \Z/2\Z \oplus \Z/2\Z & 0 \\
4 & D_4^3 A_3 & \Z/2\Z  & 0 \\
5 & A_7 A_3^2 & \Z/2\Z &  P_2[\frac{1}{2}] \\
6 & A_7 A_3 A_1^2 & \Z/2\Z &  A_3^*[2] \\
7 & A_3^4 & \Z/2\Z & \langle 1 \rangle \oplus \langle 1 \rangle \oplus \langle 1 \rangle \\
%
8 & D_6^2 A_1^2 & 0 & \langle 1 \rangle \\
9 & D_6 D_5 A_1^4 & \Z/2\Z & 0\\
10 & A_5^2 A_1^2 & \Z/2\Z &  (A_2 \oplus A_1) [\frac{2}{3}] \\
11 & D_8 A_1^6 & \Z/2\Z & \langle 1 \rangle \\
12 & E_6 D_5 & 0 & P_4^* [4] \\ 
13 & E_6 D_4 & 0 & A_5^*[2] \\
14 & A_5^2 & 0 & A_2[ \frac{2}{3}] \oplus  A_2^*[2] \oplus \langle 1 \rangle \\
15 & D_8 D_4 A_3 & 0 & 0 \\
16 & E_7 D_4 A_1^3 & 0 & \langle 1 \rangle \\
17 & D_7 D_4^2 & 0 & 0 \\
18 & E_7 A_3 A_1^5 & \Z/2\Z & 0 \\
19 & D_5^2 & 0 & A_3^*[2] \oplus P_2[\frac{1}{2}] \\
20 & D_{10} A_1^4 & 0 & \langle 1 \rangle \\
21 & A_9 A_1^3 & \Z/2\Z & P_3^*[4] \\
22 & A_9 A_1 & 0 & (A_4^* \oplus A_1^*)[2] \\
23 & D_9 A_1^6 & \Z/2\Z & 0 \\
24 & E_8 A_1^6 & 0 & \langle 1 \rangle \\
25 & D_9 & 0 & P_6^*[4]\\
\end{array} 
$$

In the analysis of the fibrations, we will describe explicit
generators for the Mordell-Weil group and give the height pairing
matrix for a basis of the Mordell-Weil lattice. We shall omit the
lattice basis reduction step which shows that the Mordell-Weil lattice
is as listed above, since this step is easily performed with a
computer algebra package such as PARI/gp.

\subsection{Formulas for the Kummer surface}

The curve $C$ is a double cover of $\Proj^1$ branched at six
Weierstrass points. Since $k$ is algebraically closed, without loss of
generality we may put three of the Weierstrass points at
$0,1,\infty$. We assume from now on that the Weierstrass equation of
the genus $2$ curve $C$ is given by
$$
y^2 = x(x-1)(x-a)(x-b)(x-c).
$$

We may use the formulas from \cite{CF} to write down the equation of
the Kummer quartic surface. This is given by
\begin{equation}\label{kumeqn}
K_2 \, z_4^2 + K_1 \, z_4 + K_0 = 0
\end{equation}
with 
\begin{align*}
K_2 = \,& z_2^2-4\,z_1\,z_3 , \\
K_1 = \,& (-2\,z_2+4\,(a+b+c+1)\,z_1)\,z_3^2   \\
   & + (-2\,(b\,c+a\,c+c+a\,b+b+a)\,z_1\,z_2+ 4\,(a\,b\,c+b\,c+a\,c+a\,b)\,z_1^2 )\,z_3 \\
 &-2\,a\,b\,c\,z_1^2\,z_2 ,\\
K_0 =& \, z_3^4 -2\,(b\,c+a\,c+c+a\,b+b+a)\,z_1\,z_3^3 \\
 & +(4\,(a\,b\,c+b\,c+a\,c+a\,b)\,z_1\,z_2 \\
  &  +(a^2+b^2+c^2 - 2\,a\,b\,(a+b+1) - 2\,b\,c\,(b+c+1) - 2\,a\,c\,(a+c+1) \\
& +a^2\,b^2+ b^2\,c^2 +a^2\,c^2-2\,a\,b\,c\,(a+b+c+4))\,z_1^2)\,z_3^2 \\
&    +(-4\,a\,b\,c\,z_1\,z_2^2+ 4\,a\,b\,c\,(c+b+a+1)\,z_1^2\,z_2 \\
& -2\,a\,b\,c\,(b\,c+a\,c+c+a\,b+b+a)\,z_1^3)\,z_3 \\
& +a^2\,b^2\,c^2\,z_1^4 .
\end{align*}

We can also work out the locations of the nodes and the equations of
the tropes in projective 3-space. These are as follows:

\begin{align*}
n_0 &= [0: 0: 0: 1] & T_o &= [1: 0: 0: 0] \\
n_1 &= [0: 1: 0: 0] & T_1 &= [0: 0: 1: 0]\\
n_2 &= [0: 1: 1: 1] & T_2 &= [1: -1: 1: 0]\\
n_3 &= [0: 1: a: a^2] & T_3 &= [a^2: -a: 1: 0]\\
n_4 &= [0: 1: b: b^2] & T_4 &= [b^2: -b: 1: 0]\\
n_5 &= [0: 1: c: c^2] & T_5 &= [c^2: -c: 1: 0] \\
n_{12} &= [1: 1: 0: abc] & T_{12} &= [-abc: 0: -1: 1] \\
n_{13} &= [1: a: 0: bc] & T_{13} &= [-bc: 0: -a: 1]\\
n_{14} &= [1: b: 0: ca] & T_{14} &= [-ca: 0: -b: 1]\\
n_{15} &= [1: c: 0: ab] & T_{15} &= [-ab: 0: -c: 1]\\
n_{23} &= [1: a + 1: a: a(b + c)] & T_{23} &= [-a(b + c): a: -(a + 1): 1]\\
n_{24} &= [1: b + 1: b: b(c + a)] & T_{24} &= [-b(c + a): b: -(b + 1): 1]\\
n_{25} &= [1: c + 1: c: c(a + b)] & T_{25} &= [-c(a + b): c: -(c + 1): 1]\\
n_{34} &= [1: a + b: ab: ab(c + 1)] & T_{34} &= [-ab(c + 1): ab: -(a + b): 1]\\
n_{35} &= [1: a + c: ca: ac(b + 1)] & T_{35} &= [-ca(b + 1): ca: -(c + a): 1]\\
n_{45} &= [1: b + c: bc: bc(a + 1)] & T_{45} &= [-bc(a + 1): bc: -(b + c): 1]\\
\end{align*}

Here we have displayed the equations of the hyperplanes defining the
tropes in terms of coordinates on the dual projective space. For
instance, we can read out from the list above that the trope $T_3$ is
given by $a^2 z_1 - az_2 + z_3 = 0$.

\subsection{Igusa-Clebsch invariants} \label{igcl}

For later reference, we recall the definition and properties of the
Igusa-Clebsch invariants \cite{Cl,I,M} of a genus $2$ curve $C$. These
are actually classical covariants of a sextic form, which transform
appropriately under the action of $\PGL_2(k)$.

Given a sextic Weierstrass equation for $C$,

$$
y^2 = f(x) = \sum\limits_{i=0}^6 f_i x^i = f_6 \prod\limits_{i=1}^6 (x - \alpha_i),
$$
the first Igusa-Clebsch invariant of $f$ is defined to be
$$
I_2(f) = f_6^2 \sum (12)^2(34)^2(56)^2 := f_6^2 \sum (\alpha_{i_1} -\alpha_{i_2})^2 (\alpha_{i_3}-\alpha_{i_4})^2 (\alpha_{i_5}-\alpha_{i_6})^2
$$ 
where the sum is over all partitions $\{\{i_1,
i_2\},\{i_3,i_4\},\{i_5,i_6\}\}$ of $\{1,2,3,4,5,6\}$ into three
subsets of two elements each.  Similarly, we define
\begin{align*}
I_4(f) &= f_6^4 \sum (12)^2(23)^2(31)^2(45)^2(56)^2(64)^2 \\
I_6(f) &= f_6^6 \sum (12)^2(23)^2(31)^2(45)^2(56)^2(64)^2(14)^2(25)^2(36)^2 \\
I_{10}(f) &= f_6^{10} \prod (ij)^2.
\end{align*}
Note that $I_{10}(f)$ is the discriminant of the sextic polynomial $f$.

These transform covariantly under the action of $\GL_2(k)$: the sextic
$f$ transforms under the action of $g \in \GL_2(k)$ with
$$
g^{-1} = \left(\begin{array}{cc} a & b \\ c & d \end{array} \right)
$$
to the polynomial
$$
(g \cdot f)(x) = f\left(\frac{ax + b}{cx+d}\right) (cx+d)^6 ,
$$
and the Igusa-Clebsch invariants transform as
$$
I_d(g \cdot f) = \det(g)^{-3d} I_d(f).
$$ 
It is easy to see that there is a unique way to extend the
Igusa-Clebsch invariants to the case when $f$ is a quintic (i.e. one
of the Weierstrass points of the genus two curve is at infinity) so
that they continue to satisfy the transformation property above.

For the Igusa-Clebsch invariants of the Weierstrass equation
$$
y^2 = x\,(x-1)\,(x-a)\,(x-b)\,(x-c),
$$
we obtain the following expressions.
\begin{align*}
I_2&=2\,\Big(3\,\sigma_1^2-2\,(\sigma_2+4\,\sigma_3)\,\sigma_1+3\,\sigma_2^2-8\,\sigma_2+12\,\sigma_3\Big) \\
I_4&=4\,\Big(-3\,\sigma_3\,\sigma_1^3+(\sigma_2^2-\sigma_3\,\sigma_2+\sigma_3^2+3\,\sigma_3)\,\sigma_1^2+(-\sigma_2^2+11\,\sigma_3\,\sigma_2-3\,\sigma_3)\,\sigma_1 \\
& \quad \quad -3\,\sigma_2^3+(3\,\sigma_3+1)\,\sigma_2^2-3\,\sigma_3^2\,\sigma_2-18\,\sigma_3^2\Big) \\
I_6&= 2\,\Big(-12\,\sigma_3\,\sigma_1^5+ 2\,\big(2\,\sigma_2^2+5\,\sigma_3\,\sigma_2+12\,\sigma_3^2+6\,\sigma_3 \big)\,\sigma_1^4 \\
& \quad \quad  +\big(-4\,\sigma_2^3-2\,(9\,\sigma_3+2)\,\sigma_2^2+(10\,\sigma_3+59)\,\sigma_3\,\sigma_2-4\,(3\,\sigma_3^2+17\,\sigma_3+3)\,\sigma_3 \big)\,\sigma_1^3 \\
& \quad \quad + \big(4\,\sigma_2^4 -2\,(2\,\sigma_3+9)\,\sigma_2^3+4\,(\sigma_3^2+1)\,\sigma_2^2-(97\,\sigma_3+33)\,\sigma_3\,\sigma_2+16\,\sigma_3^3+5\,\sigma_3^2 \big)\,\sigma_1^2 \\
& \quad \quad + \big(10\,\sigma_2^4+(59\,\sigma_3+10)\,\sigma_2^3-(33\,\sigma_3+97)\,\sigma_3\,\sigma_2^2 \\
& \qquad \quad      +2\,(19\,\sigma_3^2+103\,\sigma_3+19)\,\sigma_3\,\sigma_2+3\,(25\,\sigma_3-7)\,\sigma_3^2 \big)\,\sigma_1 \\
& \quad \quad  -12\,\sigma_2^5+12\,(\sigma_3+2)\,\sigma_2^4-4\,(3\,\sigma_3^2+17\,\sigma_3+3)\,\sigma_2^3 \\
& \qquad \quad     +(5\,\sigma_3+16)\,\sigma_3\,\sigma_2^2-3\,(7\,\sigma_3-25)\,\sigma_3^2\,\sigma_2-18\,(\sigma_3^2+7\,\sigma_3+1)\,\sigma_3^2 \Big) \\
I_{10}&=a^2\,b^2\,c^2\,(a-1)^2\,(b-1)^2\,(c-1)^2\,(a-b)^2\,(b-c)^2\,(c-a)^2. \\
\end{align*}
Here, $\sigma_1, \sigma_2, \sigma_3$ are the elementary symmetric polynomials in $a,b,c$:
\begin{align*}
\sigma_1 &= a + b + c \\
\sigma_2 &= a\,b + b\,c + c\,a \\
\sigma_3 &= a\,b\,c.
\end{align*}

We will also use another expression $I_5$ which is not quite an
invariant: it is a square root of the discriminant $I_{10}$:
$$
I_5 = a\,b\,c\,(a-1)\,(b-1)\,(c-1)\,(a-b)\,(b-c)\,(c-a).
$$

\newpage

\section{Fibration 1}\label{f1}

This corresponds to the divisor class $H - N_0 - N_1 = 2T_0 + N_2 +
N_3 + N_4 + N_5$. We can see some of the classes of the sections and
the reducible components of the reducible fibers in the Dynkin diagram
below. As usual, the nodes represent roots in the N\'eron-Severi
lattice, which correspond to smooth rational curves of
self-intersection $-2$. In addition, we label the non-identity
component of the fiber containing $N_{23}$ by $N'_{23}$, etc.: this is
just convenient (non-standard) notation.

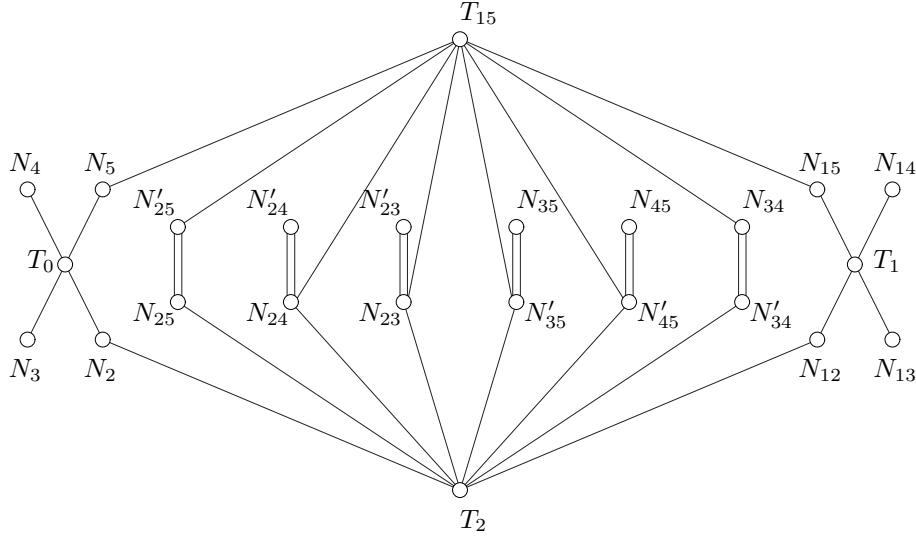
\begin{center}
\setlength{\unitlength}{1cm}
\begin{picture}(12.5,8)(0.0,0.0)
\filltype{white}

\path(1.0,4.0)(0.5,3.0)
\path(1.0,4.0)(0.5,5.0)
\path(1.0,4.0)(1.5,3.0) 
\path(1.0,4.0)(1.5,5.0)

\path(11.5,4.0)(11.0,3.0)
\path(11.5,4.0)(11.0,5.0)
\path(11.5,4.0)(12.0,3.0)
\path(11.5,4.0)(12.0,5.0)

\path(2.45,3.5)(2.45,4.5)
\path(2.55,3.5)(2.55,4.5)

\path(3.95,3.5)(3.95,4.5)
\path(4.05,3.5)(4.05,4.5)

\path(5.45,3.5)(5.45,4.5)
\path(5.55,3.5)(5.55,4.5)

\path(6.95,3.5)(6.95,4.5)
\path(7.05,3.5)(7.05,4.5)

\path(8.45,3.5)(8.45,4.5)
\path(8.55,3.5)(8.55,4.5)

\path(9.95,3.5)(9.95,4.5)
\path(10.05,3.5)(10.05,4.5)

\path(6.25,1.0)(1.5,3.0)
\path(6.25,1.0)(11.0,3.0)
\path(6.25,7.0)(1.5,5.0)
\path(6.25,7.0)(11.0,5.0)

\path(6.25,1.0)(5.5,3.5)
\path(6.25,7.0)(5.55,3.5)

\path(6.25,1.0)(7.0,3.5)
\path(6.25,7.0)(6.95,3.5)

\path(6.25,1.0)(4.0,3.5)
\path(6.25,7.0)(4.05,3.5)

\path(6.25,1.0)(8.5,3.5)
\path(6.25,7.0)(8.45,3.5)

\path(6.25,1.0)(10.0,3.5)
\path(6.25,7.0)(10.0,4.5)

\path(6.25,1.0)(2.5,3.5)
\path(6.25,7.0)(2.5,4.5)

\put(1.0,4.0){\circle*{0.2}}
\put(0.5,3.0){\circle*{0.2}}
\put(1.5,3.0){\circle*{0.2}}
\put(0.5,5.0){\circle*{0.2}}
\put(1.5,5.0){\circle*{0.2}}

\put(2.5,3.5){\circle*{0.2}}
\put(4.0,3.5){\circle*{0.2}}
\put(5.5,3.5){\circle*{0.2}}
\put(7.0,3.5){\circle*{0.2}}
\put(8.5,3.5){\circle*{0.2}}
\put(10.0,3.5){\circle*{0.2}}

\put(2.5,4.5){\circle*{0.2}}
\put(4.0,4.5){\circle*{0.2}}
\put(5.5,4.5){\circle*{0.2}}
\put(7.0,4.5){\circle*{0.2}}
\put(8.5,4.5){\circle*{0.2}}
\put(10.0,4.5){\circle*{0.2}}

\put(11.5,4.0){\circle*{0.2}}
\put(11.0,3.0){\circle*{0.2}}
\put(12.0,3.0){\circle*{0.2}}
\put(11.0,5.0){\circle*{0.2}}
\put(12.0,5.0){\circle*{0.2}}

\put(6.25,1.0){\circle*{0.2}}
\put(6.25,7.0){\circle*{0.2}}

\put(1.25,2.5){$N_2$}
\put(1.25,5.25){$N_5$}
\put(0.25,2.5){$N_3$}
\put(0.25,5.25){$N_4$}
\put(0.5,3.95){$T_0$}

\put(10.75,2.5){$N_{12}$}
\put(10.75,5.25){$N_{15}$}
\put(11.75,2.5){$N_{13}$}
\put(11.75,5.25){$N_{14}$}
\put(11.75,3.95){$T_1$}

\put(1.90,3.25){$N_{25}$}
\put(3.40,3.25){$N_{24}$}
\put(4.90,3.25){$N_{23}$}

\put(1.90,4.75){$N'_{25}$}
\put(3.40,4.75){$N'_{24}$}
\put(4.90,4.75){$N'_{23}$}

\put(7.10,3.25){$N'_{35}$}
\put(8.60,3.25){$N'_{45}$}
\put(10.10,3.25){$N'_{34}$}

\put(7.0,4.75){$N_{35}$}
\put(8.50,4.75){$N_{45}$}
\put(10.00,4.75){$N_{34}$}

\put(6.25,0.5){$T_2$}
\put(6.25,7.25){$T_{15}$}

\end{picture}
\end{center}

This elliptic fibration has two $D_4$ fibers, $2T_0 + N_2 + N_3 + N_4
+ N_5$ and $2T_1 + N_{12} + N_{13} + N_{14} + N_{15}$, as well as six
$A_1$ fibers as shown in the diagram. We take the identity section to
be $T_2$, and observe that $T_3,T_4,T_5$ are $2$-torsion
sections. Therefore this elliptic fibration has full $2$-torsion. The
trivial lattice (that is, the sublattice of $\NS(X)$ spanned by the
classes of the zero section and the irreducible components of fibers)
has rank $2 + 2 \cdot 4 + 6 \cdot 1 = 16$ and discriminant $2^6 \cdot
4^2 = 2^{10}$. A non-torsion section is given by $T_{12}$; its
translations by the torsion sections $T_3,T_4, T_5$ are the sections
$T_{13},T_{14}, T_{15}$ respectively. The height of any of these four
non-torsion sections is $4 - (1 + 1) - (1/2 + 1/2) = 1$. Therefore we
see that the sublattice of $\NS(X)$ generated by the trivial lattice,
the torsion sections, and any of these four non-torsion sections has
rank $17$ and absolute discriminant $2^{10} \cdot (1/2)^2 \cdot 1 =
2^6$. Consequently, it must equal all of $\mathrm{NS}(\Kum)$, which we
know to be of the same rank and discriminant. Therefore, the
Mordell-Weil rank for this elliptic fibration is $1$.

Now we describe a Weierstrass equation for this elliptic fibration.

First, the fact that the class of the fiber is $H - N_0 - N_1$
indicates that we must look for linear polynomials in $z_1, \dots,
z_4$ vanishing on the singular points $n_0 = [0:0:0:1]$ and $n_1=
[0:1:0:0]$ of the quartic Kummer surface with equation
\eqref{kumeqn}. Clearly, a basis for this space is given by the
monomials $z_1$ and $z_3$. Therefore, we set the elliptic parameter $t
= z_3/z_1$ and substitute for $z_3$ in equation
\eqref{kumeqn}. Simplifying the resulting (quadratic in $z_4$) equation,
we get the following form
\begin{align*}
\eta^2 &= 4\,t\,(\xi - t - 1)\,(a\,\xi - t - a^2)\,(b\,\xi - t - b^2)\,(c\,\xi - t - c^2) 
\end{align*}
with
\begin{align*}
\xi &= z_2/z_1 \\
 \eta &= (z_4/z_1)\left(\xi^2 - 4\,t\right) - \xi\, (t^2 + (a + b + c + a\,b + b\,c+ c\,a)\,t + a\,b\,c) \\ & \qquad + 2\,t\, ((a+b+c+1)\,t + (a\,b + b\,c + c\,a + a\,b\,c))
\end{align*}
This is the equation of a genus $1$ curve over $k(t)$, which clearly
has rational points. For instance, there are four Weierstrass points
given by $\xi = 1+t, (a^2+t)/a, (b^2 + t)/b, (c^2+ t)/c$. Translating
by the first of these, we obtain the Weierstrass equation
\begin{align}\label{ell1}
y^2 =& \,\big(x + 4\,(a-1)\,(b-1)\,c\,t\,(t-a)\,(t-b)\big) \cdot \notag\\
& \, \big(x + 4\,(b-1)\,(c-1)\,a\,t\,(t-b)\,(t-c)\big) \cdot \\
& \, \big(x+4\,(c-1)\,(a-1)\,b\,t\,(t-c)\,(t-a)) \notag
\end{align}
with
\begin{align*}
x &= \frac{4\,(a-1)\,(b-1)\,(c-1)\,t\,(t-a)\,(t-b)\,(t-c)}{\xi-1-t} \\
y &=  \frac{4\,(a-1)\,(b-1)\,(c-1)\,t\,(t-a)\,(t-b)\,(t-c)\,\eta}{(\xi-1-t)^2}.
\end{align*}

This elliptic fibration was the one described in \cite{Sh1}.

Now it is easy to check that the elliptic fibration given by equation
\eqref{ell1} has $D_4$ fibers at $t = 0$ and $\infty$, and six $A_1$
fibers at $t = a,b,c,ab,bc,ca$. Using the change of coordinates and
the equations of the nodes and tropes on the Kummer surface, we may
compute the components of the reducible fibers, and also some torsion
and non-torsion sections of the elliptic fibration.

The reducible fibers are as follows:
\begin{center}
\begin{tabular}{|l|c|c|} \hline
Position & Reducible fiber & Kodaira-N\'eron type \\
\hline \hline
$t=\infty$ & $N_2 + N_3 + N_4 + N_5 + 2T_0$ & $D_4$ \\
\hline
$t = 0$ & $N_{12} + N_{13} + N_{14} + N_{15} + 2T_1$ & $D_4$ \\
\hline
$t=a$ & $N_{23} + N'_{23}$ & $A_1$ \\ \hline
$t=b$ & $N_{24} + N'_{24}$ & $A_1$ \\ \hline
$t=c$ & $N_{25} + N'_{25}$ & $A_1$ \\ \hline
$t=ab$ & $N_{34} + N'_{34}$ & $A_1$ \\ \hline
$t=bc$ & $N_{45} + N'_{45}$ & $A_1$ \\ \hline
$t=ca$ & $N_{35} + N'_{35}$ & $A_1$ \\ \hline
\end{tabular}
\end{center}

Equations for some of the torsion and non-torsion section are given in
the tables below.
\begin{center}
\begin{tabular}{|c|l|} \hline
Torsion section & Equation \\
\hline \hline
$T_2$ & $x = y = \infty$ \\
\hline
$T_3$ & $x = -4\,a\,(b-1)\,(c-1)\,t\,(t-b)\,(t-c), \quad y = 0$ \\
\hline
$T_4$ & $x = -4\,b\,(c-1)\,(a-1)\,t\,(t-c)\,(t-a), \quad y = 0$ \\
\hline
$T_5$ & $x = -4\,c\,(a-1)\,(b-1)\,t\,(t-a)\,(t-b), \quad y = 0$ \\
\hline
\end{tabular}
\end{center}
\begin{center}
\begin{tabular}{|c|l|} \hline
Non-torsion section & Equation \\
\hline \hline
$T_{12}$ & $x = 4\,(t-a)\,(t-b)\,(t-c)\,(t-abc)$ \\
& $ y = 8\,(t-a)\,(t-b)\,(t-c)\,(t-ab)\,(t-ac)\,(t-bc)$ \\
\hline
$T_{13}$ & $x =  -4\,(b-1)\,(c-1)\,t\,(t-a)\,(at-bc)$ \\
& $ y = -8\,(b-1)\,(c-1)\,(b-a)\,(c-a)\,t^2\,(t-a)\,(t-bc)$ \\
\hline
$T_{14}$ & $x =  -4\,(c-1)\,(a-1)\,t\,(t-b)\,(bt-ca)$ \\
& $ y = -8\,(c-1)\,(a-1)\,(c-b)\,(a-b)\,t^2\,(t-b)\,(t-ca)$ \\
\hline
$T_{15}$ & $x =  -4\,(a-1)\,(b-1)\,t\,(t-c)\,(ct-ab)$ \\
& $y = -8\,(a-1)\,(b-1)\,(a-c)\,(b-c)\,t^2\,(t-c)\,(t-ab)$ \\
\hline
\end{tabular}
\end{center}

\section{Fibration 2}\label{f2}

This fibration corresponds to the divisor class $N_3 + N_4 + 2 T_0 + 2
T_2 + 2 N_2 + N_{12} + N_{23}$, a fiber of type $D_6$.
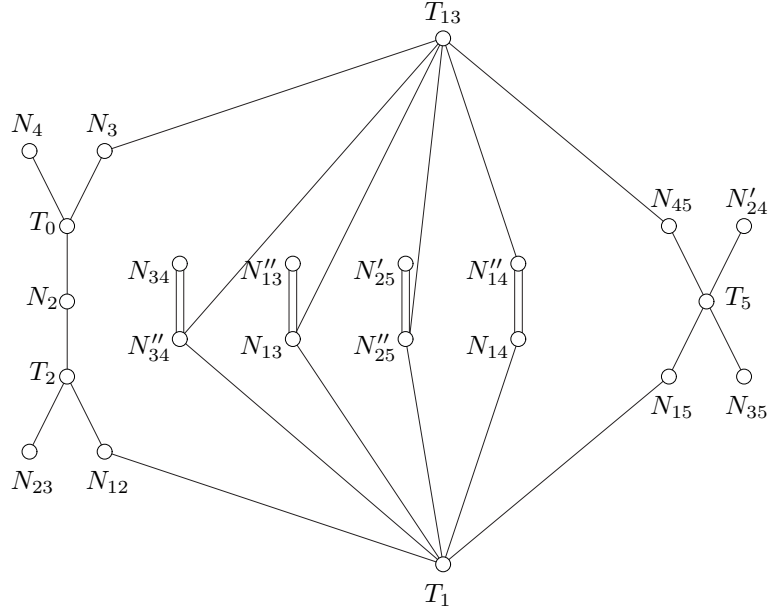
\begin{center}
\setlength{\unitlength}{1cm}
\begin{picture}(10.0,8)(0.0,0.0)
\filltype{white}

\path(1.0,3.0)(0.5,2.0)
\path(1.0,3.0)(1.5,2.0)
\path(1.0,3.0)(1.0,4.0)
\path(1.0,5.0)(1.0,4.0)
\path(1.0,5.0)(0.5,6.0)
\path(1.0,5.0)(1.5,6.0)

\path(2.45,3.5)(2.45,4.5)
\path(2.55,3.5)(2.55,4.5)
\path(3.95,3.5)(3.95,4.5)
\path(4.05,3.5)(4.05,4.5)
\path(5.45,3.5)(5.45,4.5)
\path(5.55,3.5)(5.55,4.5)
\path(6.95,3.5)(6.95,4.5)
\path(7.05,3.5)(7.05,4.5)

\path(9.5,4.0)(9.0,3.0)
\path(9.5,4.0)(9.0,5.0)
\path(9.5,4.0)(10.0,5.0)
\path(9.5,4.0)(10.0,3.0)

\path(6.0,0.5)(1.5,2.0)
\path(6.0,7.5)(1.5,6.0)

\path(6.0,0.5)(2.5,3.5)
\path(6.0,7.5)(2.5,3.5)

\path(6.0,0.5)(4.0,3.5)
\path(6.0,7.5)(4.0,3.5)

\path(6.0,0.5)(5.5,3.5)
\path(6.0,7.5)(5.55,3.5)

\path(6.0,0.5)(7.0,3.5)
\path(6.0,7.5)(7.0,4.5)

\path(6.0,0.5)(9.0,3.0)
\path(6.0,7.5)(9.0,5.0)

\put(0.5,2.0){\circle*{0.2}}
\put(1.5,2.0){\circle*{0.2}}
\put(0.5,6.0){\circle*{0.2}}
\put(1.5,6.0){\circle*{0.2}}
\put(1.0,3.0){\circle*{0.2}}
\put(1.0,4.0){\circle*{0.2}}
\put(1.0,5.0){\circle*{0.2}}

\put(2.5,3.5){\circle*{0.2}}
\put(4.0,3.5){\circle*{0.2}}
\put(5.5,3.5){\circle*{0.2}}
\put(7.0,3.5){\circle*{0.2}}

\put(2.5,4.5){\circle*{0.2}}
\put(4.0,4.5){\circle*{0.2}}
\put(5.5,4.5){\circle*{0.2}}
\put(7.0,4.5){\circle*{0.2}}

\put(9.0,3.0){\circle*{0.2}}
\put(9.0,5.0){\circle*{0.2}}
\put(10.0,3.0){\circle*{0.2}}
\put(10.0,5.0){\circle*{0.2}}
\put(9.5,4.0){\circle*{0.2}}

\put(6.0,0.5){\circle*{0.2}}
\put(6.0,7.5){\circle*{0.2}}

\put(0.25,1.5){$N_{23}$}
\put(1.25,1.5){$N_{12}$}
\put(0.25,6.25){$N_{4}$}
\put(1.25,6.25){$N_{3}$}

\put(0.5,2.95){$T_2$}
\put(0.45,3.95){$N_2$}
\put(0.5,4.95){$T_0$}

\put(1.80,3.30){$N''_{34}$}
\put(1.80,4.30){$N_{34}$}
\put(3.30,3.30){$N_{13}$}
\put(3.30,4.30){$N''_{13}$}
\put(4.80,3.30){$N''_{25}$}
\put(4.80,4.30){$N'_{25}$}
\put(6.30,3.30){$N_{14}$}
\put(6.30,4.30){$N''_{14}$}

\put(8.75,2.5){$N_{15}$}
\put(9.75,2.5){$N_{35}$}
\put(8.75,5.25){$N_{45}$}
\put(9.75,5.25){$N'_{24}$}
\put(9.75,3.95){$T_5$}

\put(5.75, 0.0){$T_1$}
\put(5.75,7.75){$T_{13}$}

\end{picture}
\end{center}

We compute that the other reducible fibers consist of one $D_4$ fiber
and four $A_1$ fibers, as in the figure above. We take $T_1$ to be the
the zero section, and note that $T_4$ is a $2$-torsion section,
whereas $T_{13}$ is a non-torsion section. The height of this section
is $4 - 6/4 - 1/2 -1 = 1$, and therefore the span of these sections
and the trivial lattice has rank $1 + 2 + 6 +4 + 4 = 17$ and
discriminant $4\cdot 4 \cdot 2^4 \cdot 1 /2^2 = 64$. So it must be the
full N\'eron-Severi lattice $\NS(X)$.

To obtain the Weierstrass equation of this elliptic fibration, we use
a $2$-neighbor step from fibration $1$, as follows. We compute
explicitly the space of sections of the line bundle $\sO_X(F_2)$ where
$F_2 = 2 T_0 + 2 T_2 + 2 N_2 + N_3 + N_4 + N_{12} + N_{23}$ is the
class of the fiber we are considering. The space of sections is
$2$-dimensional, and the ratio of two linearly independent global
sections will be an elliptic parameter for $X$, for which the class of
the fiber will be $F_2$. Notice that any global section has a pole of
order at most $2$ along $T_2$, the zero section of fibration
$1$. Also, it has at most a double pole along $N_2$, which is the
identity component of the $t = \infty$ fiber, a simple pole along
$N_{12}$, the identity component of the $t = 0$ fiber, and a simple
pole along $N_{23}$, the identity component of the $t = a$ fiber. We
deduce that a global section of $\sO_X(F_2)$ must have the form
$$
\frac{cx + a_0 + a_1t + a_2 t^2 + a_3 t^3 + a_4 t^4}{t(t-a)}.
$$

It is clear that $1$ is a global section of $\sO_X(F_2)$, so we may
subtract a term $a_2 t(t-a)$ from the numerator, and thereby assume
$a_2 = 0$. With this constraint, we will obtain a $1$-dimensional
space of sections, and the generator of this space will give us a
parameter on the base for fibration $2$. To pin down the remaining
section up to scaling, we may fix the scaling $c = 1$.

To obtain further conditions on the $a_i$, we look at the order of
vanishing along the non-identity components of the fibers at $t = 0,
a, \infty$. After some calculation, we find that an elliptic parameter
is given by
$$
w = \frac{x + 4\,(a-1)\,(b-1)\,c\,t^2\,(t-a)}{t\,(t-a)}.
$$

If we solve for $x$ in terms of $w$ and substitute in to right hand
side of the Weierstrass equation of fibration $1$, after simplifying
and dividing out some square factors (which we can absorb into $y^2$)
we obtain an equation for fibration $2$ in the form $y^2 = q_4(t)$,
where $q_4$ is a quartic in $t$ with coefficients in
$\Q(a,b,c,w)$. But in fact, the quartic factors into $t$ times a
linear factor in $t$ and a quadratic factor in $t$. We may write $t =
1/u$ and dividing $y^2$ by $t^4$, get a cubic equation $y^2 =
c_3(u)$. Finally, we can make admissible changes to the Weierstrass
equation, and also a linear fractional transformation of the base
parameter $w$, to get the following Weierstrass equation for fibration
$2$.

\begin{align*}
Y^2  =&\; X\,\biggl(X^2 + T\,X\,\Bigl(4\,(a-1)\,b\,c\,(c-b)\,T^2 -4\,(2\,a\,b\,c^2-b\,c^2-a\,c^2 -a\,b^2\,c  +2\,b^2\,c -\,b\,c \\
& -a^2\,b\,c-a^2\,c+2\,a\,c-a\,b^2+2\,a^2\,b-a\,b)\,T + 4\,(b-1)\,(c-a)\,(a\,c+b-2\,a)\Bigr) \\
& + T^2\, \Bigl( -16\,(a-1)\,a\,b\,(b-a)\,(c-1)\,c\,(c-b)\, T^3 + 16\,a\,(b-a)\,(c-1)\,(b^2\,c^2 +c^2\\
& + 2\,a\,b\,c^2-3\,b\,c^2-a\,c^2-3\,a\,b^2\,c +2\,b^2\,c-a^2\,b\,c+3\,a\,b\,c-b\,c+a^2\,b^2-a\,b^2)\,T^2 \\
&  -16\,a\,(b-1)\,(b-a)\,(c-1)\,(c-a)\,(2\,b\,c+a\,c-2\,c-2\,a\,b+b)\,T  \\
& +  16\,a\,(b-1)^2\,(b-a)\,(c-1)\,(c-a)^2\Bigr)\biggr).
\end{align*}

The new parameters $T,X,Y$ are related to the ones from fibration $1$
by the following equations.

\begin{align*}
T &= \frac{x+4\,(a-1)\,(b-1)\,c\,t\,(t-a)\,(t-b)}{4\,(a-1)\,b\,c\,(c-b)\,t\,(t-a)} \\
X &= \frac{-a\,(x + \mu_1)\,(x+\mu_2)\,(x+\mu_3)}{16\,(a-1)^2\,b^2\,c^2\,(c-b)^2\,t^4\,(t-a)^3} \\
Y &= \frac{a\,y\,(x+ \mu_2)\,(x+\mu_3)\,(x+\mu_4)}{64\,(a-1)^3\,b^3\,c^3\,(c-b)^3\,t^6\,(t-a)^4}
\end{align*}
with 
\begin{align*}
\mu_1 &= 4\,(a-1)\,b\,(c-1)\,t\,(t-a)\,(t-c) \\
\mu_2 &= 4\,(a-1)\,(b-1)\,c\,t\,(t-a)\,(t-b) \\
\mu_3 &= 4\,(b-1)\,c\,t\,(t-a)\,(a\,t-t-b\,c+b) \\
\mu_4 &= 4\,(a-1)\,c\,t\,(t-a)\,(b\,t-t-b\,c+b).
\end{align*}

In the following sections, we will simply display the final elliptic
parameter and the final Weierstrass equation for each fibration, but
leave out the transformation of coordinates (and the intermediate
cleaning up of the Weierstrass equation, including any linear
fractional transformations of the elliptic parameter).  If fibration
$n$ is obtained by a $2$-neigbhor step from fibration $m$, we will
write the elliptic parameter $t_n$ as a rational function of $t_m,
x_m$ $y_m$. Then we will display the Weierstrass equation connecting
$y_n$ and $x_n$ (with coefficients polynomials in $t_n$), except that
we will drop the subscripts and simply use the variables $t,x,y$ for
convenience. Also, we will give few details about the $2$-neighbor
step which carries one elliptic fibration into another: this is
accomplished as above, by computing the space of sections of the
linear system described by the class of the new fiber. The ratio of
two linearly independent sections gives an elliptic parameter, and
substituting for $x$ or $y$ in terms of the new parameter $w$ gives an
equation of the form
$$
z^2 = \textrm{quartic}(t)
$$

with coefficients in $\Q(a,b,c,w)$. One then finds an explicit point
on this curve. Its existence is guaranteed by the fact that we have a
genuine elliptic fibration with a section, and to actually find a
point, it is enough to guess a curve on the original fibration whose
intersection with the fiber is $1$, and compute its new coordinate $z$
using the change of coordinates so far. Finally, the rational point
may be used to transform the equation into a standard cubic
Weierstrass form. A brief description of the method, with some
details, is given in Appendix A and will be omitted in the main body
of the paper.

We return to the description of fibration $2$. Its reducible fibers
are as follows.

\begin{center}
\begin{tabular}{|l|c|c|} \hline
Position & Reducible fiber & Kodaira-N\'eron type \\
\hline \hline
$t=\infty$ & $2 T_0 + 2 T_2 + 2 N_2 + N_3 + N_4 + N_{12} + N_{23}$ & $D_6$ \\
\hline
$t = 0$ & $2 T_5 + N_{15} + N_{35} + N_{45} + N'_{24}$ & $D_4$ \\
\hline
$t=1$ & $N_{14} + N''_{14}$ & $A_1$ \\ \hline
$t=\frac{c-a}{c}$ & $N''_{34} + N_{34}$ & $A_1$ \\  \hline
$t=\frac{b-1}{b}$ & $N''_{25} + N'_{25}$ & $A_1$ \\  \hline
$t=\frac{(b-1)(c-a)}{(a-1)(c-b)}$ & $N_{13} + N''_{13}$ & $A_1$ \\ \hline
\end{tabular}
\end{center}

Some of the sections are indicated in the following table.

\begin{center}
\begin{tabular}{|c|l|} \hline
Section & Equation \\
\hline \hline
$T_1$ & $x = y = \infty$ \\
\hline
$T_4$ & $x = y = 0$ \\
\hline
$T_{13}$ & $x = 4\,a\,(b-a)\,(c-1)\,t^2$ \\
& $ y = -8\,a\,(b-1)\,(b-a)\,(c-1)\,(c-a)\,t^2\,(t-1)$ \\
\hline
$T_{15}$ & $x = -4\,(b\,t-b+1)\,(c\,t-c+a)\,((a-1)\,(c-b)\,t - (b-1)\,(c-a))$ \\
& $y = -8\,(b-1)\,(c-a)\,(t-1)\,(b\,t-b+1)\,(c\,t-c+a) \cdot$ \\
& $\qquad ((a-1)\,(c-b)\,t - (b-1)\,(c-a))$ \\
\hline
\end{tabular}
\end{center}

Note that the $T_{15}$ is the translation of the section $T_{13}$ by
the torsion section $T_4$.

\section{Fibration 3}

We go on to the next elliptic fibration. This corresponds to the
elliptic divisor $N_3 + N_5 + 2T_0 + 2N_2 + 2T_2 + N_{12} +
N_{24}$. It is therefore obtained by a $2$-neighbor step from
fibration $1$.

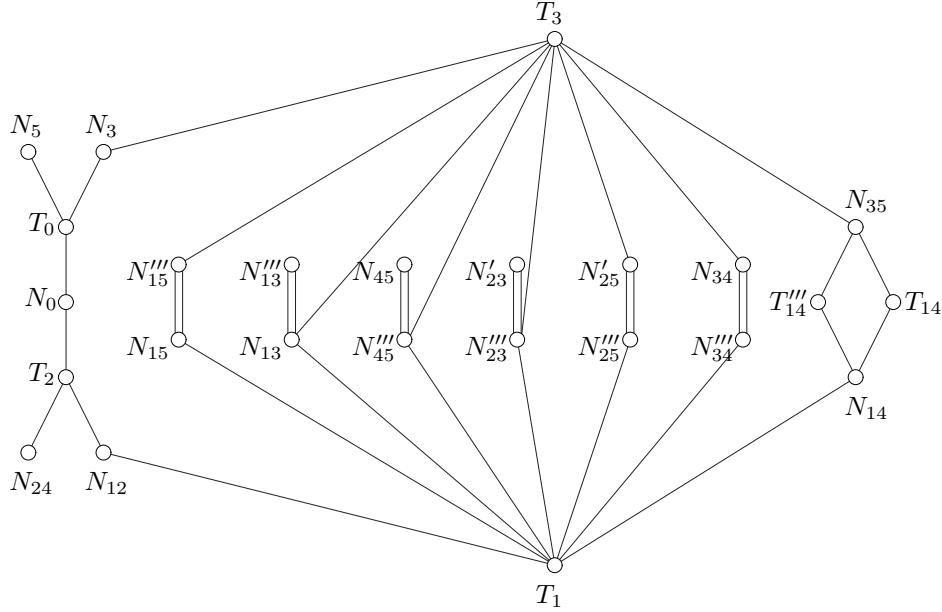
\begin{figure}[h]
\begin{center}
\setlength{\unitlength}{1cm}
\begin{picture}(12.0,8)(0.0,0.0)
\filltype{white}

\path(1.0,3.0)(0.5,2.0)
\path(1.0,3.0)(1.5,2.0)
\path(1.0,3.0)(1.0,4.0)
\path(1.0,5.0)(1.0,4.0)
\path(1.0,5.0)(0.5,6.0)
\path(1.0,5.0)(1.5,6.0)

\path(2.45,3.5)(2.45,4.5)
\path(2.55,3.5)(2.55,4.5)
\path(3.95,3.5)(3.95,4.5)
\path(4.05,3.5)(4.05,4.5)
\path(5.45,3.5)(5.45,4.5)
\path(5.55,3.5)(5.55,4.5)
\path(6.95,3.5)(6.95,4.5)
\path(7.05,3.5)(7.05,4.5)
\path(8.45,3.5)(8.45,4.5)
\path(8.55,3.5)(8.55,4.5)
\path(9.95,3.5)(9.95,4.5)
\path(10.05,3.5)(10.05,4.5)

\path(11.0,4.0)(11.5,3.0)
\path(11.0,4.0)(11.5,5.0)
\path(12.0,4.0)(11.5,5.0)
\path(12.0,4.0)(11.5,3.0)

\path(7.5,0.5)(1.5,2.0)
\path(7.5,7.5)(1.5,6.0)

\path(7.5,0.5)(2.5,3.5)
\path(7.5,7.5)(2.5,4.5)

\path(7.5,0.5)(4.0,3.5)
\path(7.5,7.5)(4.0,3.5)

\path(7.5,0.5)(5.5,3.5)
\path(7.5,7.5)(5.55,3.5)

\path(7.5,0.5)(7.0,3.5)
\path(7.5,7.5)(7.05,3.5)

\path(7.5,0.5)(8.5,3.5)
\path(7.5,7.5)(8.5,4.5)

\path(7.5,0.5)(10.0,3.5)
\path(7.5,7.5)(10.0,4.5)

\path(7.5,0.5)(11.5,3.0)
\path(7.5,7.5)(11.5,5.0)

\put(0.5,2.0){\circle*{0.2}}
\put(1.5,2.0){\circle*{0.2}}
\put(0.5,6.0){\circle*{0.2}}
\put(1.5,6.0){\circle*{0.2}}
\put(1.0,3.0){\circle*{0.2}}
\put(1.0,4.0){\circle*{0.2}}
\put(1.0,5.0){\circle*{0.2}}

\put(2.5,3.5){\circle*{0.2}}
\put(4.0,3.5){\circle*{0.2}}
\put(5.5,3.5){\circle*{0.2}}
\put(7.0,3.5){\circle*{0.2}}
\put(8.5,3.5){\circle*{0.2}}
\put(10.0,3.5){\circle*{0.2}}

\put(2.5,4.5){\circle*{0.2}}
\put(4.0,4.5){\circle*{0.2}}
\put(5.5,4.5){\circle*{0.2}}
\put(7.0,4.5){\circle*{0.2}}
\put(8.5,4.5){\circle*{0.2}}
\put(10.0,4.5){\circle*{0.2}}

\put(11.0,4.0){\circle*{0.2}}
\put(11.5,5.0){\circle*{0.2}}
\put(11.5,3.0){\circle*{0.2}}
\put(12.0,4.0){\circle*{0.2}}

\put(7.5,0.5){\circle*{0.2}}
\put(7.5,7.5){\circle*{0.2}}

\put(0.25,1.5){$N_{24}$}
\put(1.25,1.5){$N_{12}$}
\put(0.25,6.25){$N_{5}$}
\put(1.25,6.25){$N_{3}$}

\put(0.5,2.95){$T_2$}
\put(0.45,3.95){$N_0$}
\put(0.5,4.95){$T_0$}

\put(1.80,3.30){$N_{15}$}
\put(1.80,4.30){$N'''_{15}$}
\put(3.30,3.30){$N_{13}$}
\put(3.30,4.30){$N'''_{13}$}
\put(4.80,3.30){$N'''_{45}$}
\put(4.80,4.30){$N_{45}$}
\put(6.30,3.30){$N'''_{23}$}
\put(6.30,4.30){$N'_{23}$}
\put(7.80,3.30){$N'''_{25}$}
\put(7.80,4.30){$N'_{25}$}
\put(9.30,3.30){$N'''_{34}$}
\put(9.30,4.30){$N_{34}$}

\put(12.15,3.90){$T_{14}$}
\put(10.35,3.90){$T'''_{14}$}
\put(11.35,2.50){$N_{14}$}
\put(11.35,5.25){$N_{35}$}

\put(7.25, 0.0){$T_1$}
\put(7.25,7.75){$T_{3}$}

\end{picture}
\end{center}
\end{figure}

We may take $T_1$ as the zero section, since its intersection with the
class of the fiber is $1$. This fibration has $A_3, D_6$ and six $A_1$
fibers, full $2$-torsion, and Mordell-Weil rank $0$. Therefore the
sublattice of N\'eron-Severi lattice spanned by these sections and the
trivial lattice has rank $2 + 3 + 6 + 6 \cdot 1 = 15$, and
discriminant $4\cdot 4 \cdot 2^6 / 2^2 = 64$, which matches that of
$NS(X)$. Therefore it must be the full N\'eron-Severi lattice. The
$2$-torsion sections are $T_3$, $T_4$ and $T_5$.

The elliptic parameter is given by
$$
t_3 = -\frac{x_1}{a\,c\,t_1\,(t_1-b)} - \frac{4\,(a-1)\,(c-1)\,(b\,t_1-a\,c)}{a\,c}.
$$

A Weierstrass equation for this fibration is the following.

\begin{align*}
y^2 =&\, \Bigl(x + 4\,(a-1)\,b\,(c-b)\,t\,(t+4\,(b-a)\,(c-1))\Bigr)\cdot \\
& \,\Bigl(x - 4\,b\,(b-a)\,(c-1)\,t\,(t-4\,(a-1)\,(c-b))\Bigr)\cdot \\
& \,\Bigl(x - (t-4\,(a-1)\,(c-b))\,(t+ 4\,(b-a)\,(c-1))\,\cdot\\
& \, \, \,\quad \quad (a\,c\,t+ 4\,(a-1)\,(b-a)\,(c-1)\,(c-b))\Bigr).
\end{align*}

The reducible fibers are described in the following table.

\begin{center}
\begin{tabular}{|l|c|c|} \hline
Position & Reducible fiber & type \\
\hline \hline
$t=\infty$ & $N_3 + N_5 + 2T_0 + 2N_2 + 2T_2 + N_{12} + N_{24}$ & $D_6$ \\
\hline
$t = 0$ & $N_{14} + T_{14} + N_{35} + T'''_{14}$ & $A_3$ \\
\hline
$t=4(a-1)(c-b)$ & $N_{15} + N'''_{15}$ & $A_1$ \\ \hline
$t=-4(c-1)(b-a)$ & $N_{13} + N'''_{13}$ & $A_1$ \\ \hline
$t=-4(a-1)(c-1)(b-a)/a$ & $N'_{25} + N'''_{25}$ & $A_1$ \\ \hline
$t=-4(a-1)(b-a)(c-b)/a$ & $N_{45} + N'''_{45}$ & $A_1$ \\ \hline
$t=4(a-1)(c-1)(c-b)/c$ & $N'_{23} + N'''_{23}$ & $A_1$ \\ \hline
$t=-4(c-1)(b-a)(c-b)/c$ & $N_{34} + N'''_{34}$ & $A_1$ \\ \hline
\end{tabular}
\end{center}

The (torsion) sections are indicated in the following table.

\begin{center}
\begin{tabular}{|c|l|} \hline
Torsion section & Equation \\
\hline \hline
$T_1$ & $x = y = \infty$. \\
\hline
$T_3$ & $x = 4\,b\,(b-a)\,(c-1)\,t\,(t-4\,a\,c+4\,c+4\,a\,b-4\,b) \quad y = 0$ \\
\hline
$T_4$ & $x = (t-4\,(a-1)\,(c-b))\,(t+ 4\,(b-a)\,(c-1))\,\cdot$\\
& $\qquad (a\,c\,t+ 4\,(a-1)\,(b-a)\,(c-1)\,(c-b))$ \\
& $ y= 0$ \\
\hline
$T_5$ & $x = -4\,(a-1)\,b\,(c-b)\,t\,(t+4\,b\,c-4\,a\,c-4\,b+4\,a), \quad y = 0$ \\
\hline
\end{tabular}
\end{center}

This fibration was studied by Keum \cite{Ke1}.

\section{Fibration 4}

This fibration corresponds to the elliptic divisor $2T_2 + N'_{34} +
N_2 + N_{12} + N_{25}$, in the terminology of Section \ref{f1}
(fibration $1$). It is obtained by a $2$-neighbor step from fibration
$1$.

We may take $T_0$ to be the zero section. This fibration has three
$D_4$ fibers, an $A_3$ fiber, a $2$-torsion section $T_1$, and
Mordell-Weil rank $0$. The lattice spanned by the $2$-torsion section
and the trivial lattice has rank $2 + 3 \cdot 4 + 3 = 17$ and
discriminant $4^3 \cdot 4/2^2 = 64$, so it must be all of $\NS(X)$.

An elliptic parameter is given by 
$$
t_4 = -(c+a\,b) + \frac{4\,(a-1)\,(b-1)\,(c-a)\,(c-b)\,t_1\,(t_1-a\,b)\,(t_1-c)}{x_1 + 4\,(a-1)\,(b-1)\,t_1\,(t_1-c)\,(c\,t_1-a\,b)}.
$$

A Weierstrass equation is given by
\begin{align*}
y^2 & = x\,\Big(x^2 + 4\,x\,t\,(t+c+a\,b)\,(t+a\,c+b)\,(t+b\,c+a) \\
& \qquad + 16\,a\,b\,c\,(t+c+a\,b)^2\,(t+a\,c+b)^2\,(t+b\,c+a)^2\Big).
\end{align*}

This has the following reducible fibers.

\begin{center}
\begin{tabular}{|l|c|c|} \hline
Position & Reducible fiber & type \\
\hline \hline
$t=-(c+a\,b)$ & $2T_2 + N'_{34} + N_2 + N_{12} + N_{25}$ & $D_4$ \\
\hline
$t = -(b+a\,c)$ & $2T_3 + N_3 + N_{13} + N_{35} + N'_{24}$ & $D_4$ \\
\hline
$t= -(a+b\,c)$ & $2T_4 + N_4 + N_{14} + N_{45} + N'_{23}$ & $D_4$ \\
\hline
$t=\infty$ & $N_5 + T_{15} + N_{15} + T^{(4)}_{15}$ & $A_3$ \\
\hline
\end{tabular}
\end{center}

The torsion section $T_1$ is given by $x = y = 0$.

\section{Fibration 5}

This fibration corresponds to the elliptic divisor $T_2 + N_{25} +
T_{13} + N'_{34} $, in the terminology of fibration $1$. It is
obtained by a $2$-neighbor step from fibration $1$.

We may take $N_{13}$ to be the zero section. This fibration has an
$A_7$ fiber, two $A_3$ fibers, a $2$-torsion section, and Mordell-Weil
rank $2$. We will describe explicit sections with canonical height
pairing
$$ \frac{1}{2}
\left( \begin{array}{cc}
4 & 2 \\
2 & 3
\end{array} \right).
$$ The determinant of the height pairing matrix is $2$. This implies
that the sublattice of $\NS(X)$ generated by these sections, along
with the torsion section and the trivial lattice has rank $17$ and
discriminant $8 \cdot 4 \cdot 4 \cdot 2 /2^2 = 64$. Therefore it must
be all of $\NS(X)$.

An elliptic parameter is given by 
$$
t_5 = \frac{y_1 -8\,(b-1)\,(b-a)\,(c-1)\,(c-a)\,t_1^2\,(t_1-a)\,(t_1-b\,c)}{2\,(t_1-a\,b)\,(t_1-c)\,\big(x_1 + 4\,(b-1)\,(c-1)\,t_1\,(t_1-a)\,(a\,t_1-b\,c)\big)} + \frac{(b-1)\,(c-a)\,t_1}{(t_1-a\,b)\,(t_1-c)}.
$$

A Weierstrass equation for this elliptic fibration is given by

\begin{align*}
y^2 &= x\,\bigg(x^2 + x\,\Big(16\,(c-a\,b)^2\,t^4 + 64\,(b-1)\,(c-a)\,(c+a\,b)\,t^3  \\
& \quad \qquad + 32\,(2\,a\,b\,c^2-4\,b\,c^2-a\,c^2+3\,c^2-4\,a\,b^2\,c+2\,b^2\,c-a^2\,b\,c \\
& \quad \qquad +6\,a\,b\,c-b\,c +2\,a^2\,c-4\,a\,c+3\,a^2\,b^2-a\,b^2-4\,a^2\,b+2\,a\,b)\,t^2 \\
& \quad \qquad -64\,(a-1)\,(b-1)\,(c-a)\,(c-b)\,t + 16\,(a-1)^2\,(c-b)^2\Big)  \\
& \quad \qquad + 4096\,a\,(b-1)\,b\,(b-a)\,(c-1)\,c\,(c-a)\,(t-1)^2\,t^4\bigg).
\end{align*}

This has the following reducible fibers.

\begin{center}
\begin{tabular}{|l|c|c|} \hline
Position & Reducible fiber & type \\
\hline \hline
$t=\infty$ & $T_{13} + N'_{34} + T_2 + N_{25}$ & $A_3$ \\
\hline
$t = 0$ & $T_1 + N_{15} + T_{15} + N_5 + T_0 + N_4 + T_4 + N_{14}$ & $A_7$ \\
\hline
$t=1$ & $T_3 + N'_{24} + T_{12} + N_{35}$ & $A_3$ \\
\hline
\end{tabular}
\end{center}

The two non-torsion sections mentioned above are $N_{12}$ and
$N_{23}$, and the torsion section is $N_2$.

\begin{center}
\begin{tabular}{|c|l|} \hline
Section & Equation \\
\hline \hline
$N_{13}$ & $x = y = \infty$ \\
\hline
$N_2$ & $x = y = 0$ \\
\hline
$N_{12}$ & $x = 64\,(b-1)\,(b-a)\,(c-1)\,(c-a)\,(t-1)^2$ \\
& $y = 256\,(b-1)\,(b-a)\,(c-1)\,(c-a)\,(t-1)^2\,((c + a\,b)\,t^2+$ \\
& \qquad $2\,(b-1)\,(c-a)\,t-2\,b\,c+a\,c+c+a\,b+b-2\,a)$ \\
\hline
$N_{23}$ & $x = 64\,b\,(b-a)\,(c-1)\,c\,t^2$ \\
& $y = 256\,b\,(b-a)\,(c-1)\,c\,t^2\,((c + a\,b - 2\,a)\,t^2+$ \\
& \qquad $2\,(b-1)\,(c-a)\,t-(a-1)\,(c-b))$ \\
\hline
\end{tabular}
\end{center}

\section{Fibration 6}

This fibration corresponds to the elliptic divisor $T_5 + N'_{34} +
T_2 + N_{25}$, in the terminology of fibration $1$. It is obtained by
a $2$-neighbor step from fibration $1$.

We may take $N_{15}$ to be the zero section. This fibration has an
$A_7$ fiber $T_1 + N_{14} + T_4 + N_4 + T_0 + N_3 + T_3 + N_{13}$, an
$A_3$ fiber $T_5 + N'_{34} + T_2 + N_{25}$, two $A_1$ fibers
$T^{(6)}_{12} + T_{12}$ and $T_{15}+T^{(6)}_{15}$, a $2$-torsion
section $N_2$, and Mordell-Weil rank $3$. The divisors $N_{12}$,
$N_{23}$ and $N_{24}$ have the intersection pairing

$$
\frac{1}{2}\left( \begin{array}{ccc}
4 & 2 & 2 \\
2 & 3 & 1 \\
2 & 1 & 3 \\
\end{array} \right).
$$ 
The determinant of the height pairing matrix is $2$. This implies
that the sublattice of $\NS(X)$ generated by these sections, along
with the torsion section and the trivial lattice, has rank $17$ and
discriminant $8 \cdot 4 \cdot 2^2 \cdot 2 /2^2 = 64$. Therefore it
must be all of $\NS(X)$.

An elliptic parameter is given by 
$$
t_6 = \frac{y_1}{2\,(t_1-c)\,(t_1-a\,b)\,\big(x_1 + 4\,(a-1)\,(b-1)\,c\,t_1\,(t_1-a)\,(t_1-b) \big)}.
$$

A Weierstrass equation for this elliptic fibration is given by

\begin{align*}
y^2 &= x\,\bigg(x^2 + x\,\Big(16\,(c-a\,b)^2\,t^4 + 32\,(2\,a\,b\,c^2-b\,c^2 -a\,c^2-a\,b^2\,c +2\,b^2\,c-a^2\,b\,c\\
& \quad \qquad -b\,c+2\,a^2\,c-a\,c-a\,b^2-a^2\,b +2\,a\,b) \,t^2 +16\,(b-a)^2\,(c-1)^2 \Big) \\
& \quad \qquad -4096\,(a-1)\,a\,(b-1)\,b\,c\,(c-a)\,(c-b)\,(t-1)\,t^4\,(t+1) \bigg).
\end{align*}

This has the following reducible fibers.

\begin{center}
\begin{tabular}{|l|c|c|} \hline
Position & Reducible fiber & type \\
\hline \hline
$t=\infty$ & $T_5 + N'_{34} + T_2 + N_{25}$ & $A_3$ \\
\hline
$t = 0$ & $T_1 + N_{14} + T_4 + N_4 + T_0 + N_3 + T_3 + N_{13}$  & $A_7$ \\
\hline
$t=1$ & $T^{(6)}_{12} + T_{12}$ & $A_1$ \\
\hline
$t=-1$ & $T_{15}+T^{(6)}_{15}$  & $A_1$ \\
\hline
\end{tabular}
\end{center}

Generators for the Mordell-Weil group are in the table below.

\begin{center}
\begin{tabular}{|c|l|} \hline
Section & Equation \\
\hline \hline
$N_{15}$ & $x = y = \infty$ \\
\hline
$N_2$ & $x = y = 0$ \\
\hline
$N_{12}$ & $x = -64\,(a-1)\,(b-1)\,(c-a)\,(c-b)\,(t-1)\,(t+1)$ \\
& $y = -256\,(a-1)\,(b-1)\,(c-a)\,(c-b)\,(t-1)\,(t+1)\,(c\,t^2$ \\
& \qquad $+a\,b\,t^2+b\,c+a\,c-2\,c-2\,a\,b+b+a)$ \\
\hline
$N_{23}$ & $x = -64\,(a-1)\,b\,c\,(c-b)\,t^2$ \\
& $y = -256\,(a-1)\,b\,c\,(c-b)\,t^2\,((c + a\,b - 2\,a)\,t^2+(b-a)\,(c-1))$ \\
\hline
$N_{24}$ & $x = -64\,a\,(b-1)\,c\,(c-a)\,t^2$ \\
& $y = -256\,a\,(b-1)\,c\,(c-a)\,t^2\,((c + a\,b - 2\,b)\,t^2-(b-a)\,(c-1))$ \\
\hline
\end{tabular}
\end{center}

\section{Fibration 7}

This fibration corresponds to the elliptic divisor $N_2 + T_2 + N_{12}
+ T_{12}$, in the terminology of fibration $1$. It is obtained by a
$2$-neighbor step from fibration $1$.

We may take $T_0$ to be the zero section. This fibration has four
$A_3$ fibers indicated in the table below, a $2$-torsion section
$T_1$, and Mordell-Weil rank $3$. The sections $N_{34}$, $N_{35}$ and
$N_{45}$ are all of height $1$ and are orthogonal. Therefore the
determinant of the height pairing matrix is $1$. This implies that the
sublattice of $\NS(X)$ generated by these sections, along with the
torsion section and the trivial lattice, has rank $17$ and
discriminant $4^4 / 2^2 \cdot 1 = 64$. Therefore it must be all of
$\NS(X)$.

An elliptic parameter is given by 
$$
t_7 = \frac{2\,(a-1)\,(b-1)\,(c-1)\,t_1\,\big(x_1 -4\,(t_1-a)\,(t_1-b)\,(t_1-c)\,(t_1-a\,b\,c)\big)}{y_1 + 2\,\big(t_1^2-(a+b+c-1)\,t_1+a\,b\,c\big)\,x_1 + 8\,(a-1)\,(b-1)\,(c-1)\,t_1^2\,(t_1-a)\,(t_1-b)\,(t_1-c)} - 1.
$$

A Weierstrass equation for this elliptic fibration is given by

\begin{align*}
y^2 &= x\,\bigg(x^2 + x\,\Big((
a^2 + b^2 + c^2 - 2\,a\,b - 2\,b\,c - 2\,c\,a - 2\,a - 2\,b - 2\,c + 1)\,t^4   \\
&  \qquad \quad + 8\,(a\,b\,c+a\,b+b\,c+c\,a)\,t^3 -2\,((a+b+c+10)\,a\,b\,c \\
& \qquad \quad   + (a+b)\,(b+c)\,(c+a) + a\,b + b\,c + c\,a)\,t^2 + 8\,a\,b\,c\,(a+b+c+1)\,t \\
& \qquad \quad  + (a-1)^2\,b^2\,c^2 + (b-1)^2\,a^2\,c^2 + (c-1)^2\,a^2\,b^2 - 2\,a\,b\,c\,(a\,b\,c+a+b+c) \Big) \\
& \qquad \quad + 16\,a\,b\,c\,(t-1)^2\,(t-a)^2\,(t-b)^2\,(t-c)^2 \bigg).
\end{align*}

This has the following reducible fibers.

\begin{center}
\begin{tabular}{|l|c|c|} \hline
Position & Reducible fiber & type \\
\hline \hline
$t=1$ & $N_2 + T_2 +  N_{12} + T_{12}$ & $A_3$ \\
\hline
$t=a$ & $N_3 + T_3 + N_{13} + T_{13}$ & $A_3$ \\
\hline
$t=b$ & $N_4 + T_4 + N_{14} + T_{14}$ & $A_3$ \\
\hline
$t=c$ & $N_5 + T_5 + N_{15} + T_{15}$ & $A_3$ \\
\hline
\end{tabular}
\end{center}

Generators for the Mordell-Weil group are in the table below.

\begin{center}
\begin{tabular}{|c|l|} \hline
Section & Equation \\
\hline \hline
$T_0$ & $x = y = \infty$ \\
\hline
$T_1$ & $x = y = 0$ \\
\hline
$N_{34}$ & $x = 4\,a\,b\,(t-1)\,(t-a)\,(t-b)\,(t-c)$ \\
& $y = 4\,a\,b\,(t-1)\,(t-a)\,(t-b)\,(t-c)\cdot$ \\
& \qquad $\,((c-b-a+1)\,t^2-2\,(c-a\,b)\,t-a\,b\,c+b\,c+a\,c-a\,b)$ \\
\hline
$N_{35}$ & $x = 4\,a\,c\,(t-1)\,(t-a)\,(t-b)\,(t-c)$ \\
& $y = 4\,a\,c\,(t-1)\,(t-a)\,(t-b)\,(t-c)\cdot$ \\
& \qquad $\,((c-b+a-1)\,t^2-2\,(a\,c-b)\,t+a\,b\,c-b\,c+a\,c-a\,b)$ \\
\hline
$N_{45}$ & $x =  4\,b\,c\,(t-1)\,(t-a)\,(t-b)\,(t-c)$ \\
& $y = 4\,b\,c\,(t-1)\,(t-a)\,(t-b)\,(t-c)\cdot$ \\
& \qquad $\,((c+b-a-1)\,t^2-2\,(b\,c-a)\,t+a\,b\,c+b\,c-a\,c-a\,b)$ \\
\hline
\end{tabular}
\end{center}

\section{Fibration 8}

This fibration corresponds to the elliptic divisor $N_{13} + N_{15} +
2T_1 + 2N_{12} + 2T_2 + N_{25} + N'_{34}$, in the terminology of
fibration $1$. It is obtained by a $2$-neighbor step from fibration
$1$.

We may take $T_3$ to be the zero section. This fibration has two $D_6$
fibers and two $A_1$ fibers and Mordell-Weil rank $1$. The section
$T_{15}$ has height $1$ and is a generator for the Mordell-Weil group,
since the sublattice of $\NS(X)$ it generates along with the trivial
lattice has rank $17$ and discriminant $4^2 \cdot 2^2 \cdot 1 =
64$. Therefore it must be all of $\NS(X)$.

An elliptic parameter is given by 
$$
t_8 = \frac{-a\,b\,c\,\big(x_1 + 4\,(a-1)\,b\,(c-1)\,t_1\,(t_1-a)\,(t_1-c)\big)}{4\,(b-a)\,(c-1)\,t_1^2\,(t_1-a\,b)\,(t_1-c)}.
$$

A Weierstrass equation for this elliptic fibration is given by

\begin{align*}
y^2 &= x^3 + x^2 \,t\,\Big( 2\,(b-a)\,(c-1)\,t^2  -(2\,a\,b\,c^2-b\,c^2-3\,a^2\,c^2+2\,a\,c^2 -a\,b^2\,c+2\,b^2\,c\\
& \quad +2\,a^2\,b\,c-6\,a\,b\,c+2\,b\,c +2\,a^2\,c-a\,c+2\,a\,b^2 -3\,b^2-a^2\,b+2\,a\,b)\,t \\
& \quad  -2\,a\,b\,c\,(a-1)\,(c-b) \Big) + x \,t^2\,(t-b)\,(t-a\,c)\,\Big( (b-a)^2\,(c-1)^2\,t^2 \\
& \quad -(a-1)\,(b-a)\,(c-1)\,(c-b)\,(b\,c-3\,a\,c+c+a\,b-3\,b+a)\,t \\
& \quad +  a\,b\,c\,(a-1)^2\,(c-b)^2 \Big) + (a-1)^2\,(b-a)^2\,(c-1)^2\,(c-b)^2\,t^4\,(t-b)^2\,(t-a\,c)^2.
\end{align*}

This has the following reducible fibers.

\begin{center}
\begin{tabular}{|l|c|c|} \hline
Position & Reducible fiber & type \\
\hline \hline
$t=\infty$ & $N_{13} + N_{15} + 2T_1 + 2N_{12} + 2T_2 + N_{25} + N'_{34}$ & $D_6$ \\
\hline
$t=0$ & $N_3 + N_5 + 2T_0 + 2N_4 + 2T_4 + N_{45} + N'_{23}$ & $D_6$ \\
\hline
$t=b$ & $N_{35} + N^{(8)}_{35}$ & $A_1$ \\
\hline
$t=ac$ & $N'_{24} + N^{(8)}_{24}$ & $A_1$ \\
\hline
\end{tabular}
\end{center}

The section $T_{15}$ is defined by
$$
x = 0, \quad y = (a-1)\,(b-a)\,(c-1)\,(c-b)\,t^2\,(t-b)\,(t-a\,c).
$$

\section{Fibration 8A}

This is the elliptic fibration corresponding to the elliptic divisor
$N_3 + N_{13} + 2T_{13} + 2N_{45} + 2T_5 + N_{35} + N'_{24}$. It may
be obtained by a $2$-neighbor step from fibration $2$. But note that
translation by the section $-T_{13}$ of fibration $2$ transforms this
elliptic divisor to $N_{12} + N_{13} + 2T_1 + 2N_{15} + 2T_5 + N_{35}
+ N'_{24}$, which is related to the elliptic divisor of fibration $8$
by an automorphism in $\Aut(D')$. Therefore this fibration is not new.

\section{Fibration 9}

This fibration corresponds to the elliptic divisor $N_{14} + N_{13} +
2T_1 + 2N_{12} + 2T_2 + N_{25} + N'_{34}$, in the terminology of
fibration $1$. It is obtained by a $2$-neighbor step from fibration
$1$.

We may take $T_4$ to be the zero section. This fibration has a $D_6$
fiber, a $D_5$ fiber, and four $A_1$ fibers, as well as a $2$-torsion
section. These account for a sublattice of $\NS(X)$ of rank $17$ and
discriminant $4^2 \cdot 2^4 / 2^2 = 64$, so they generate all of
$\NS(X)$. The Mordell-Weil rank of the elliptic fibration is zero.

An elliptic parameter is given by 
$$
t_9 = \frac{x_1 + 4\,(a-1)\,(b-1)\,t_1\,(t_1-c)\,(c\,t_1-a\,b)}{t_1^2\,(t_1-c)\,(t_1-a\,b) }.
$$

A Weierstrass equation for this elliptic fibration is given by

\begin{align*}
y^2 &= x\,\bigg(x^2 + 2\,x\,t\,\Big(-\,a\,b\,c\,t^2 -2\,(2\,a\,b\,c^2-b\,c^2-a\,c^2  -a\,b^2\,c+2\,b^2\,c-a^2\,b\,c-b\,c\\
& \qquad \quad +2\,a^2\,c-a\,c-a\,b^2-a^2\,b+2\,a\,b)\,t-16\,(a-1)\,(b-1)\,(c-a)\,(c-b) \Big) \\
& \qquad \quad + t^2\,(a\,t+4\,(a-1)\,(c-b))\,(b\,t+4\,(b-1)\,(c-a)) \cdot \\
& \qquad \qquad  \,(a\,c\,t+4\,(b-1)\,(c-a))\,(b\,c\,t+4\,(a-1)\,(c-b)) \bigg).
\end{align*}

This has the following reducible fibers.

\begin{center}
\begin{tabular}{|l|c|c|} \hline
Position & Reducible fiber & type \\
\hline \hline
$t=\infty$ & $N_{14} + N_{13} + 2T_1 + 2T_{12} + 2T_2 + N_{25} + N'_{34}$ & $D_6$ \\
\hline
$t=0$ & $N_4 + N_3 + 2T_0 + 2N_5 + T_{15} + T^{(9)}_{15}$ & $D_5$ \\
\hline
$t =-4\,(a-1)\,(c-b)/a$ & $N'_{23} + N^{(9)}_{23}$ & $A_1$ \\
\hline
$t =-4\,(b-1)\,(c-a)/b$ & $N^{(9)}_{24} + N'_{24}$ & $A_1$ \\
\hline
$t = -4\,(b-1)\,(c-a)/(a\,c)$ & $N^{(9)}_{35} + N_{35}$ & $A_1$ \\
\hline
$t =-4\,(a-1)\,(c-b)/(b\,c)$ & $N_{45} + N^{(9)}_{45}$ & $A_1$ \\
\hline
\end{tabular}
\end{center}

The $2$-torsion section $T_3$ is given by $x = y = 0$.

\section{Fibration 10}

This fibration corresponds to the elliptic divisor $T_1 + N_{15} +
T_{15} + N_{24} + T_2 + N_{12}$, in the terminology of fibration
$1$. It is obtained by a $2$-neighbor step from fibration $1$.

We may take $N_{14}$ to be the zero section. This fibration has two
$A_5$ fibers and two $A_1$ fibers, a $2$-torsion section $T_{13}$, and
Mordell-Weil rank $3$. The divisors $N_2$, $N_5$ and $N'_{34}$ have
the intersection pairing

$$
\frac{1}{3}\left( \begin{array}{ccc}
4 & 0 & 2 \\
0 & 4 & 0 \\
2 & 0 & 4 \\
\end{array} \right).
$$ 
The determinant of the height pairing matrix is $16/9$. Therefore
the sublattice of $\NS(X)$ generated by these sections, along with the
torsion section and the trivial lattice, has rank $17$ and
discriminant $(6^2 \cdot 2^2 /2^2) \cdot 16/9= 64$. Therefore it must
be all of $\NS(X)$.

An elliptic parameter is given by 
$$
t_{10} = \frac{y_1 -8\,(a-1)\,(b-1)\,(a-c)\,(b-c)\,t_1^2\,(t_1-c)\,(t_1-a\,b)}{(x_1 + 4\,(a-1)\,(b-1)\,t_1\,(t_1-c)\,(c\,t_1-a\,b))\,t_1\,(t_1-b)} - \frac{2\,(a-1)\,(c-b)}{t_1-b}.
$$

A Weierstrass equation for this elliptic fibration is given by
\begin{align*}
y^2 &= x\,\bigg(x^2 + x\,\Big(b^2\,t^4  -8\,(a-1)\,b\,(c-b)\,t^3  -8\,(4\,a\,b\,c^2-2\,b\,c^2-2\,a\,c^2-2\,a\,b^2\,c \\
& \quad \qquad +b^2\,c-2\,a^2\,b\,c -3\,a\,b\,c+4\,b\,c+4\,a^2\,c-2\,a\,c+4\,a\,b^2-3\,b^2  -2\,a^2\,b \\
& \quad \qquad +a\,b)\,t^2 + 32\,(a-1)\,(b-a)\,(c-1)\,(c-b)\,t + 16\,(b-a)^2\,(c-1)^2 \Big) \\
& \quad \qquad + 128\,(a-1)\,a\,(b-1)\,c\,(c-a)\,(c-b)\,t^3\,(t+2\,c-2)\,(b\,t-2\,b+2\,a) \bigg).
\end{align*}

This has the following reducible fibers.
\begin{center}
\begin{tabular}{|l|c|c|} \hline
Position & Reducible fiber & type \\
\hline \hline
$t=\infty$ & $T_1 + N_{15} + T_{15} + N_{24} + T_2 + N_{12}$ & $A_5$ \\
\hline
$t = 0$ & $T_{14} + N_4 + T_0 + N_3 + T_3 + N_{35}$  & $A_5$ \\
\hline
$t=-2\,(c-1)$ & $T_{25} + N'_{23}$ & $A_1$ \\
\hline
$t=2\,(b-a)/b$ & $T^{(10)}_{45}+N_{45}$  & $A_1$ \\
\hline
\end{tabular}
\end{center}
Generators for the Mordell-Weil group are in the table below.
\begin{center}
\begin{tabular}{|c|l|} \hline
Section & Equation \\
\hline \hline
$N_{14}$ & $x = y = \infty$ \\
\hline
$T_{13}$ & $x = y = 0$ \\
\hline
$N_2$ & $x = 16\,(a-1)\,a\,c\,(c-b)\,t^2$ \\
& $y = 16\,(a-1)\,a\,c\,(c-b)\,t^2\,(b\,t^2+4\,(b-a)\,(c-1)\,(t-1))$ \\
\hline
$N_5$ & $x = 16\,a\,(b-1)\,c\,(c-a)\,t^2$ \\
& $y = -16\,a\,(b-1)\,c\,(c-a)\,t^2\,(b\,t^2+4\,(b-a)\,(c-1))$ \\
\hline
$N'_{34}$ & $x = 16\,(a-1)\,(b-1)\,c\,(c-a)\,t^2$ \\
& $y =  -16\,(a-1)\,(b-1)\,c\,(c-a)\,t^2\,(b\,t^2+4\,(c-b)\,t-4\,(b-a)\,(c-1))$\\
\hline
\end{tabular}
\end{center}

\section{Fibration 11}

This corresponds to the elliptic divisor $N_5 + N_3 + 2T_0 + 2N_2 +
2T_2 + 2N_{12} + 2T_1 + N_{13} + N_{14}$. It is obtained by a
$2$-neighbor step from fibration $1$.

We may take $T_5$ as the zero section. This fibration has a $D_8$
fiber, six $A_1$ fibers, a $2$-torsion section $T_4$, and Mordell-Weil
rank $1$. A non-torsion section is given by $T_{14}$; it has height
$1$ and so the sublattice of $\NS(X)$ it generates with the torsion
section and the trivial lattice has rank $17$ and discriminant $4
\cdot 2^6 /2^2 \cdot 1 = 64$. It is therefore all of $\NS(X)$.

An elliptic parameter is given by 
$$
t_{11} = \frac{x_1 + 4\,(a-1)\,a\,(b-1)\,b\,c\,t_1}{4\,(a-1)\,(c-b)\,t_1^2} + \frac{b\,(c-1)\,t_1 -a\,c^2-b^2\,c+b\,c+a\,c}{c-b}.
$$

A Weierstrass equation for this fibration is the following.

\begin{align*}
y^2 &= x\,\bigg(x^2 + x\,\Big(4\,(a-1)\,(c-b)\,t^3 -4\,(2\,a\,b\,c^2-b\,c^2-3\,a^2\,c^2 +2\,a\,c^2 -a\,b^2\,c \\
& \qquad \quad +2\,b^2\,c+2\,a^2\,b\,c-3\,a\,b\,c-b\,c +2\,a^2\,c-a\,c-a\,b^2 -a^2\,b+2\,a\,b)\,t^2 \\
& \qquad \quad + 4\,(a\,b^2\,c^3-4\,a^2\,b\,c^3+a\,b\,c^3+3\,a^3\,c^3 -a^2\,c^3+b^3\,c^2+a^2\,b^2\,c^2 \\
& \qquad \quad -4\,a\,b^2\,c^2-b^2\,c^2-a^3\,b\,c^2+6\,a^2\,b\,c^2+a\,b\,c^2-4\,a^3\,c^2+a^2\,c^2-a\,b^3\,c \\
& \qquad \quad +a^2\,b^2\,c+3\,a\,b^2\,c+a^3\,b\,c -4\,a^2\,b\,c-a\,b\,c+a^3\,c -a^2\,b^2+a^2\,b)\,t  \\
& \qquad \quad + 4\,a\,(b-a)\,(c-1)\,c\,(a\,b\,c^2-a^2\,c^2-b^2\,c-a\,b\,c+b\,c+a^2\,c+a\,b^2-a\,b) \Big)  \\
& \qquad \quad + 16\,a\,(b-1)\,b\,(b-a)\,(c-1)\,c\,(c-a)\,t\,(t+a\,c-b)\cdot \\
& \qquad \qquad (t-b\,c+a\,c+a\,b-a)\,(t-b\,c+a\,c+c-a)\bigg).
\end{align*}

\begin{center}
\setlength{\unitlength}{1cm}
\begin{picture}(11.0,8)(0.0,0.0)
\filltype{white}

\path(1.0,2.0)(0.5,1.0)
\path(1.0,3.0)(1.0,2.0)
\path(1.0,2.0)(1.5,1.0)
\path(1.0,3.0)(1.0,4.0)
\path(1.0,5.0)(1.0,4.0)
\path(1.0,5.0)(1.0,6.0)
\path(1.0,6.0)(0.5,7.0)
\path(1.0,6.0)(1.5,7.0)

\path(2.45,3.5)(2.45,4.5)
\path(2.55,3.5)(2.55,4.5)
\path(3.95,3.5)(3.95,4.5)
\path(4.05,3.5)(4.05,4.5)
\path(5.45,3.5)(5.45,4.5)
\path(5.55,3.5)(5.55,4.5)
\path(6.95,3.5)(6.95,4.5)
\path(7.05,3.5)(7.05,4.5)
\path(8.45,3.5)(8.45,4.5)
\path(8.55,3.5)(8.55,4.5)
\path(9.95,3.5)(9.95,4.5)
\path(10.05,3.5)(10.05,4.5)

\path(6.5,0.5)(1.5,1.0)
\path(6.5,7.5)(1.5,7.0)

\path(6.5,0.5)(2.5,3.5)
\path(6.5,7.5)(2.5,4.5)

\path(6.5,0.5)(4.0,3.5)
\path(6.5,7.5)(4.0,3.5)

\path(6.5,0.5)(5.5,3.5)
\path(6.5,7.5)(5.5,3.5)

\path(6.5,0.5)(7.0,3.5)
\path(6.5,7.5)(6.95,3.5)

\path(6.5,0.5)(8.5,3.5)
\path(6.5,7.5)(8.5,3.5)

\path(6.5,0.5)(10.0,3.5)
\path(6.5,7.5)(10.0,4.5)

\put(0.5,1.0){\circle*{0.2}}
\put(1.5,1.0){\circle*{0.2}}
\put(0.5,7.0){\circle*{0.2}}
\put(1.5,7.0){\circle*{0.2}}
\put(1.0,3.0){\circle*{0.2}}
\put(1.0,4.0){\circle*{0.2}}
\put(1.0,5.0){\circle*{0.2}}
\put(1.0,2.0){\circle*{0.2}}
\put(1.0,6.0){\circle*{0.2}}

\put(2.5,3.5){\circle*{0.2}}
\put(4.0,3.5){\circle*{0.2}}
\put(5.5,3.5){\circle*{0.2}}
\put(7.0,3.5){\circle*{0.2}}
\put(8.5,3.5){\circle*{0.2}}
\put(10.0,3.5){\circle*{0.2}}

\put(2.5,4.5){\circle*{0.2}}
\put(4.0,4.5){\circle*{0.2}}
\put(5.5,4.5){\circle*{0.2}}
\put(7.0,4.5){\circle*{0.2}}
\put(8.5,4.5){\circle*{0.2}}
\put(10.0,4.5){\circle*{0.2}}

\put(6.5,0.5){\circle*{0.2}}
\put(6.5,7.5){\circle*{0.2}}

\put(0.25,0.5){$N_{3}$}
\put(1.25,0.5){$N_{5}$}
\put(0.25,7.25){$N_{13}$}
\put(1.25,7.25){$N_{14}$}

\put(0.5,1.95){$T_0$}
\put(0.4,2.95){$N_2$}
\put(0.5,3.95){$T_2$}
\put(0.25,4.95){$N_{12}$}
\put(0.5,5.95){$T_1$}

\put(1.80,3.30){$N_{45}$}
\put(1.55,4.30){$N^{(11)}_{45}$}
\put(3.30,3.30){$N_{35}$}
\put(3.05,4.30){$N^{(11)}_{35}$}
\put(4.55,3.30){$N^{(11)}_{25}$}
\put(4.80,4.30){$N'_{25}$}
\put(7.15,3.30){$N'_{24}$}
\put(7.15,4.30){$N^{(11)}_{24}$}
\put(8.65,3.30){$N^{(11)}_{34}$}
\put(8.65,4.30){$N_{34}$}
\put(10.15,3.30){$N'_{23}$}
\put(10.15,4.30){$N^{(11)}_{23}$}

\put(6.25, 0.0){$T_5$}
\put(6.25,7.75){$T_{14}$}

\end{picture}
\end{center}
The reducible fibers are described in the following table.
\begin{center}
\begin{tabular}{|l|c|c|} \hline
Position & Reducible fiber & type \\
\hline \hline
$t=\infty$ & $N_5 + N_3 + 2T_0 + 2N_2 + 2T_2 + 2N_{12}+ 2T_1  $ & $D_8$ \\
& \quad $  + N_{13} + N_{14}$ & \\
\hline
$t = 0$ & $N_{35} + N^{(11)}_{35}$ & $A_1$ \\
\hline
$t = b\,c - a\,c -c + a$ & $N^{(11)}_{34} + N_{34}$ & $A_1$ \\
\hline
$t = (b-a)\,c$ & $N_{45} + N^{(11)}_{45}$ & $A_1$ \\
\hline
$t = b\,c - a\,c - a\,b + a$ & $N^{(11)}_{25} + N'_{25}$ & $A_1$ \\
\hline
$t = -a\,(c-1)$ & $N'_{23} + N^{(11)}_{23}$ & $A_1$ \\
\hline
$t = b-a\,c$ & $N'_{24} + N^{(11)}_{24}$ & $A_1$ \\
\hline
\end{tabular}
\end{center}
Generators for the Mordell-Weil group are in the table below.
\begin{center}
\begin{tabular}{|c|l|} \hline
Torsion section & Equation \\
\hline \hline
$T_5$ & $x = y = \infty$ \\
\hline
$T_4$ & $x = y = 0$ \\
\hline
$T_{14}$ & $x = 4\,a\,(b-1)\,b\,(b-a)\,(c-1)\,c\,(c-a)$ \\
& $y = -8\,a\,(b-1)\,b\,(b-a)\,(c-1)\,c\,(c-a)\,(t+a\,c-a)\cdot$\\
& \quad \quad $(t-b\,c+a\,c)$ \\
\hline
\end{tabular}
\end{center}

We note for future reference that the sum of the sections $T_4$ and
$T_{14}$ in this fibration gives the section $T_{15}$.

\section{Fibration 11A}

This fibration corresponds to the divisor $N_2 + N_4 + 2T_0 + 2N_3 +
2T_{13} + 2N_{45} + 2T_5 + N_{15} + N_{35} $. It may be obtained by a
$2$-neighbor step from fibration $2$. First note that translation by
the section $-T_{15}$ of fibration $2$ transforms this divisor to $N_2
+ N_3 + 2T_0 + 2N_4 + 2T_4 + 2N_{45} + 2T_5 + N_{15} + N_{35}$, which
is related to the elliptic divisor of fibration $11$ by an element of
$\Aut(D')$. Hence this fibration is not new.

\section{Fibration 12}

This corresponds to the elliptic divisor $T_2 + 2N_2 + 3T_0 + 2N_3 +
T_4 + 2N_4 + T_{13}$ in fibration $2$. It is obtained from a
$2$-neighbor step from fibration $2$. The automorphism which is
translation by the section $T_4$ in fibration $2$ transforms it to the
elliptic divisor $T_0 + 2N_2 + 3T_2 + 2N_{12} + T_1 + 2N_{23} +
T_{15}$, which is the one that we shall work with.

We may take $N_4$ as the zero section. This fibration has an $E_6$ and
a $D_5$ fiber, and Mordell-Weil rank $4$.

The divisors $N_{13}$, $N''_{14}$, $N''_{34}$ and $N''_{25}$ have the
intersection pairing

$$
\frac{1}{3}\left( \begin{array}{rrrr}
5 & 1 & -1 & -1 \\
1 & 5 & 1 & 1 \\
-1 & 1 & 5 & -1 \\
-1 & 1 &- 1 & 5 \\
\end{array} \right).
$$

The determinant of the height pairing matrix is $16/3$. Therefore the
sublattice of $\NS(X)$ generated by these sections, along with the
trivial lattice has rank $17$ and discriminant $3 \cdot 4 \cdot 16/3=
64$. So it must be all of $\NS(X)$.

An elliptic parameter is given by 
$$
t_{12} = \frac{y_2 + 2\,(b-1)\,(c-a)\,(t_2-1)\,x_2}{(b-a)\,(c-1)\,\big(x_2 + 4\,(b\,t_2-b+1)\,(c\,t_2-c+a)\,(a\,c\,t_2-c\,t_2-a\,b\,t_2+b\,t_2-b\,c+c+a\,b-a) \big)}.
$$

A Weierstrass equation for this elliptic fibration is given by
\begin{align*}
y^2 &= x^3 + x^2\,t\,\Big((a\,b\,c^2-2\,b\,c^2+a\,c^2-2\,a\,b^2\,c+b^2\,c+a^2\,b\,c+3\,a\,b\,c +b\,c\\
& \quad -2\,a^2\,c-2\,a\,c+a\,b^2-2\,a^2\,b-2\,a\,b+3\,a^2)\,t -4\,a\,(b-1)\,(c-a)\Big) \\
& \quad - x\,t^2\, \Big((a-1)\,(b-1)\,b\,(b-a)\,(c-1)\,c\,(c-a)\,(c-b)\,t^3+ 2\,(a\,b^2\,c^4-b^2\,c^4 -a^2\,b\,c^4\\
& \qquad +a\,b\,c^4+a^2\,b^3\,c^3-3\,a\,b^3\,c^3+b^3\,c^3-a^3\,b^2\,c^3+a^2\,b^2\,c^3 +3\,a\,b^2\,c^3 +b^2\,c^3+a^3\,b\,c^3\\
& \qquad -3\,a^2\,b\,c^3-3\,a\,b\,c^3+a^3\,c^3+a^2\,c^3-a^2\,b^4\,c^2 +a\,b^4\,c^2+a^3\,b^3\,c^2+3\,a^2\,b^3\,c^2+a\,b^3\,c^2\\
& \qquad -b^3\,c^2-3\,a^3\,b^2\,c^2-9\,a^2\,b^2\,c^2-3\,a\,b^2\,c^2 +a^4\,b\,c^2+5\,a^3\,b\,c^2+11\,a^2\,b\,c^2+a\,b\,c^2 -a^4\,c^2\\
& \qquad -5\,a^3\,c^2-a^2\,c^2 +a^2\,b^4\,c-a\,b^4\,c-3\,a^3\,b^3\,c-3\,a^2\,b^3\,c+a\,b^3\,c+a^4\,b^2\,c +11\,a^3\,b^2\,c \\
& \qquad +5\,a^2\,b^2\,c+a\,b^2\,c-5\,a^4\,b\,c-11\,a^3\,b\,c-5\,a^2\,b\,c+4\,a^4\,c+4\,a^3\,c+a^3\,b^3+a^2\,b^3-a^4\,b^2 \\
& \qquad -5\,a^3\,b^2-a^2\,b^2+4\,a^4\,b+4\,a^3\,b-3\,a^4)\,t^2 + 4\,a\,(b-1)\,(c-a)\,(a\,b\,c^2-2\,b\,c^2+a\,c^2\\
& \qquad -2\,a\,b^2\,c+b^2\,c+a^2\,b\,c+4\,a\,b\,c+b\,c-3\,a^2\,c-2\,a\,c+a\,b^2 -2\,a^2\,b-3\,a\,b+4\,a^2)\,t\\
& \qquad  -8\,a^2\,(b-1)^2\,(c-a)^2 \Big)/2 \\
& \quad + t^4\,\Big((a-1)\,b\,(b-a)\,(c-1)\,c\,(c-b)\,t^2 -4\,a\,(b-1)\,(b-a)\,(c-1)\,(c-a)\,t \\
& \qquad -4\,a\,(b-1)\,(c-a)\,(a\,c+b-2\,a)\Big)^2/16 .
\end{align*}

This fibration has the following reducible fibers.
\begin{center}
\begin{tabular}{|l|c|c|} \hline
Position & Reducible fiber & type \\
\hline \hline
$t=\infty$ & $T_0 + 2N_2 + 3T_2 + 2N_{12} + T_1 + 2N_{23} + T_{15}$ & $E_6$ \\
\hline
$t = 0$ & $N'_{24} + N_{35} + 2T_5 + 2N_{45} + T_4 + T_{13}$  & $D_5$ \\
\hline
\end{tabular}
\end{center}

Generators for the Mordell-Weil group are in the table below.

\begin{center}
\begin{tabular}{|c|l|} \hline
Section & Equation \\
\hline \hline
$N_{4}$ & $x = y = \infty$ \\
\hline
$N_{13}$ & $x =(b-a)\,(c-1)\,(b\,c-a)\,t^2 $\\
& $y = -t^2\,((a-1)\,b\,(b-a)\,(c-1)\,c\,(c-b)\,t^2 -4\,(b-1)\,b\,(b-a)\,(c-1)\,c\,(c-a)\,t$ \\
& \qquad \qquad $ + 4\,a\,(b-1)\,(c-a)\,(2\,b\,c-a\,c-b) )/4$ \\
\hline
$N''_{14}$ & $x = 0$\\
& $y = t^2\,((a-1)\,b\,(b-a)\,(c-1)\,c\,(c-b)\,t^2 -4\,a\,(b-1)\,(b-a)\,(c-1)\,(c-a)\,t$ \\
& \qquad \qquad $-4\,a\,(b-1)\,(c-a)\,(a\,c+b-2\,a))/4$ \\
\hline
$N''_{34}$ & $x =-a\,(b-1)\,(a\,c-a\,b+b-a)\,t^2$\\
& $y = -t^2\,((a-1)\,b\,(b-a)\,(c-1)\,c\,(c-b)\,t^2 + 4\,(a-1)\,a\,(b-1)\,b\,(c-a)\,(c-b)\,t$ \\
& \qquad \qquad $+ 4\,a\,(b-1)\,(c-a)\,(a\,c-2\,a\,b+b))/4$ \\
\hline
$N''_{25}$ & $x = -(c-a)\,(a\,c-c+b-a)\,t^2$\\
& $y = -t^2\,((a-1)\,b\,(b-a)\,(c-1)\,c\,(c-b)\,t^2+ 4\,(a-1)\,(b-1)\,c\,(c-a)\,(c-b)\,t$ \\
& \qquad \qquad $+ 4\,a\,(b-1)\,(c-a)\,(a\,c-2\,c+b))/4$ \\
\hline
\end{tabular}
\end{center}

\section{Fibration 13}

This fibration corresponds to the elliptic divisor $T_4 + T_1 +
2N_{14} + T_{14} + T'''_{14}$, in the terminology of fibration $3$. It
is obtained from fibration $3$ by a $2$-neighbor step. We may take
$N_{24}$ to be the zero section. This elliptic fibration has an $E_6$
fiber, a $D_4$ fiber, and Mordell-Weil rank $5$. The sections $N_{13},
N_{15}, N'''_{34}, N'''_{45}, N'''_{23}$ have intersection pairing

$$
\frac{1}{3}\left( \begin{array}{rrrrr}
5 & 1 & 1 & -1 & -1 \\
1 & 5 & -1 & 1 & 1 \\
1 & -1 & 5 & 1 & 1 \\
-1 & 1 & 1 & 5 & -1 \\
-1 & 1 & 1 & -1 & 5 \\
\end{array} \right).
$$

The determinant of the intersection matrix is $16/3$. Therefore the
sublattice of $\NS(X)$ generated by these sections and the trivial
lattice has rank $17$ and discriminant $3 \cdot 4 \cdot (16/3) = 64$.

An elliptic parameter is given by 
$$
t_{13} = \frac{(a-1)\,(b-a)\,(c-1)\,(c-b)\,y_3}{t_3^2\,\Big(x_3 -\big(t_3-4\,(a-1)\,(c-b)\big)\,\big(t_3+4\,(b-a)\,(c-1)\big)\,\big(a\,c\,t_3 + 4\,(a-1)\,(b-a)\,(c-1)\,(c-b)\big) \Big)}.
$$

The Weierstrass equation for this elliptic fibration turns out to be
invariant under all the permutations of the Weierstrass points of the
genus $2$ curve $C$. It may be written in terms of the Igusa-Clebsch
invariants of the curve as

$$
y^2 = x^3  - 108\,x\,t^4\,(48\,t^2 + I_4)  + 108\,t^4\,\Big(72\,I_2\,t^4 + (4\,I_4\,I_2 - 12\,I_6)\,t^2 + 27\,I_{10}\Big).
$$

The reducible fibers are as follows.

\begin{center}
\begin{tabular}{|l|c|c|} \hline
Position & Reducible fiber & type \\
\hline \hline
$t=\infty$ & $T_4 + T_1 + 2N_{14} + T_{14} + T'''_{14}$ & $D_4$ \\
\hline
$t = 0$ & $T_2 + 2N_2 + 3T_0 + 2N_3 + T_3 + 2N_5 + T_5$  & $E_6$ \\
\hline
\end{tabular}
\end{center}

Generators for the Mordell-Weil group are in the table below.

\begin{center}
\begin{tabular}{|c|l|} \hline
Section & Equation \\
\hline \hline
$N_{24}$ & $x = y = \infty$ \\
\hline
$N_{13}$ & $x = 12\,(a\,b\,c^2+b\,c^2-2\,a\,c^2+a\,b^2\,c-2\,b^2\,c-2\,a^2\,b\,c+b\,c+a^2\,c +a\,c+a\,b^2+a^2\,b-2\,a\,b)\,t^2$\\
& $y =  54\,t^2\,(4\,(b\,c-a\,c-c-a\,b+b+a)\,t^2  + I_5)$\\
\hline
$N_{15}$ & $x = -12\,(2\,a\,b\,c^2-b\,c^2-a\,c^2-a\,b^2\,c-b^2\,c-a^2\,b\,c+2\,b\,c+2\,a^2\,c -a\,c+2\,a\,b^2 -a^2\,b-a\,b)\,t^2$ \\
& $y = -54\,t^2\,(4\,(b\,c+a\,c-c-a\,b-b+a)\,t^2 + I_5)$\\
\hline
$N'''_{34}$ & $x = 12\,(a\,b\,c^2+b\,c^2-2\,a\,c^2-2\,a\,b^2\,c+b^2\,c+a^2\,b\,c-2\,b\,c +a^2\,c+a\,c+a\,b^2 -2\,a^2\,b+a\,b)\,t^2$ \\
& $y = -54\,t^2\,(-4\,(b\,c-a\,c-c+a\,b-b+a)\,t^2 + I_5)$\\
\hline
$N'''_{45}$ & $x = 12\,(a\,b\,c^2-2\,b\,c^2+a\,c^2-2\,a\,b^2\,c+b^2\,c+a^2\,b\,c+b\,c-2\,a^2\,c+a\,c+a\,b^2+a^2\,b-2\,a\,b)\,t^2$\\
& $y = 54\,t^2\,(4\,(b\,c-a\,c+c+a\,b-b-a)\,t^2 + I_5)$\\
\hline
$N'''_{23}$ & $x = -12\,(2\,a\,b\,c^2-b\,c^2-a\,c^2-a\,b^2\,c+2\,b^2\,c-a^2\,b\,c-b\,c-a^2\,c+2\,a\,c-a\,b^2+2\,a^2\,b-a\,b)\,t^2$\\
& $y = 54\,t^2\,(-4\,(b\,c+a\,c-c-a\,b+b-a)\,t^2 + I_5)$\\
\hline
\end{tabular}
\end{center}

with $I_5 = a\,b\,c\,(a-1)\,(b-1)\,(c-1)\,(a-b)\,(b-c)\,(c-a)$ being a
square root of $I_{10}$.

\section{Fibration 14}

This fibration comes from the elliptic divisor $T_1 + N_{13} + T_3 +
N_{23} + T_2 + N_{12}$. It is obtained from fibration $1$ by a
$2$-neighbor step.

We may take $N_{15}$ to be the zero section. This fibration has two
$A_5$ fibers and Mordell-Weil rank $5$. The divisors $N_2$, $N_3$,
$N_{24}$, $N_{34}$ and $T_{12}$ have the intersection pairing

$$
\frac{1}{3}\left( \begin{array}{ccccc}
4 & 0 & 0 & 2 & 0 \\
0 & 4 & 2 & 0 & 4 \\
0 & 2 & 4 & 0 & 2 \\
2 & 0 & 0 & 4 & 0 \\
0 & 4 & 2 & 0 & 7 \\
\end{array} \right).
$$

The determinant of the height pairing matrix is $16/9$. Therefore the
sublattice of $\NS(X)$ generated by these sections, along with the
trivial lattice, has rank $17$ and discriminant $6^2 \cdot 16/9=
64$. So it must be all of $\NS(X)$.

An elliptic parameter is given by 
$$
t_{14} = \frac{y_1}{2\,(a-1)\,(c-b)\,t_1^2\,(t_1-a)\,\big(x_1 + 4\,a\,(b-1)\,(c-1)\,t_1\,(t_1-b)\,(t_1-c) \big)}.
$$

A Weierstrass equation for this elliptic fibration is given by

\begin{align*}
y^2 &= x^3 + x^2\,\Big( (a-1)^2\,a^2\,(c-b)^2\,t^4 + 2\,(2\,a\,b\,c^2+2\,b\,c^2-4\,a\,c^2+2\,a\,b^2\,c+2\,b^2\,c\\
& \quad -4\,a^2\,b\,c-6\,a\,b\,c-4\,b\,c+5\,a^2\,c+5\,a\,c-4\,a\,b^2+5\,a^2\,b+5\,a\,b-6\,a^2)\,t^2 + 1\Big) \\
& \quad -8\,(b-1)\,(b-a)\,(c-1)\,(c-a)\,t^2\,x\,\Big((a-1)^2\,a^2\,(c-b)^2\,t^4 -2\,(a\,c^2+a^2\,b\,c \\
& \quad +a\,b\,c+ b\,c-2\,a^2\,c-2\,a\,c+a\,b^2-2\,a^2\,b-2\,a\,b+3\,a^2)\,t^2+1 \Big) \\
& \quad +16\,(b-1)^2\,(b-a)^2\,(c-1)^2\,(c-a)^2\,t^4\,\Big((a-1)^2\,a^2\,(c-b)^2\,t^4 +2\,a\,(a\,c+c \\
& \quad +a\,b+b-2\,a)\,t^2 +1\Big).
\end{align*}

This has the following reducible fibers.

\begin{center}
\begin{tabular}{|l|c|c|} \hline
Position & Reducible fiber & type \\
\hline \hline
$t=\infty$ & $T_1 + N_{13} + T_3 + N_{23} + T_2 + N_{12}$ & $A_5$ \\
\hline
$t = 0$ & $T_5 + N_5 + T_0 + N_4 + T_4 + N_{45}$  & $A_5$ \\
\hline
\end{tabular}
\end{center}

Generators for the Mordell-Weil group are in the table below.

\begin{center}
\begin{tabular}{|c|l|} \hline
Section & Equation \\
\hline \hline
$N_{14}$ & $x = y = \infty$ \\
\hline
$N_2$ & $x = -4\,(b-1)\,(c-a)\,(a\,c+b-a)\,t^2$ \\
& $y = -4\,(b-1)\,b\,c\,(c-a)\,t^2\,(a^2\,c\,t^2-a\,c\,t^2-a^2\,b\,t^2+a\,b\,t^2-1)$ \\
\hline
$N_3$ & $x =-4\,(b-a)\,(c-1)\,(c+a\,b-a)\,t^2 $ \\
& $y = 4\,b\,(b-a)\,(c-1)\,c\,t^2\,(a^2\,c\,t^2-a\,c\,t^2-a^2\,b\,t^2+a\,b\,t^2+1)$ \\
\hline
$N_{24}$ & $x = -4\,(b-1)\,(b-a)\,(c-a)\,t^2$\\
& $y = -4\,(b-1)\,(b-a)\,c\,(c-a)\,t^2\,(a^2\,c\,t^2-a\,c\,t^2-a^2\,b\,t^2+a\,b\,t^2+1)$ \\
\hline
$N_{34}$ & $x= -4\,a\,(b-1)\,(b-a)\,(c-1)\,t^2$ \\
& $y = 4\,(b-1)\,(b-a)\,(c-1)\,c\,t^2\,(a^2\,c\,t^2-a\,c\,t^2-a^2\,b\,t^2+a\,b\,t^2-1)$ \\
\hline
$T_{12}$ & $x= -4\,(b-1)\,(c-a)\,t\,(a^2\,c\,t^2-a\,c\,t^2-a^2\,b\,t^2+a\,b\,t^2+c\,t+a\,b\,t$\\
& \qquad $-2\,a\,t+1)$\\
& $y = -4\,(b-1)\,(c-a)\,t\,\Big((a-1)^2\,a^2\,(c-b)^2\,t^4 $\\
& \qquad $-(a-1)\,a\,(c-b)\,(b\,c+a\,c-3\,c-3\,a\,b+b+3\,a)\,t^3$ \\
& \qquad $-2\,(a\,c^2-c^2+a^2\,b\,c-3\,a\,b\,c+b\,c-2\,a^2\,c+3\,a\,c-a^2\,b^2+a\,b^2$ \\
& \qquad$+3\,a^2\,b-2\,a\,b-a^2)\,t^2 -(b\,c+a\,c-3\,c-3\,a\,b+b+3\,a)\,t +1 \Big)$ \\
\hline
\end{tabular}
\end{center}

\section{Fibration 15}

This fibration corresponds to the elliptic divisor $N_2 + N_3 + 2T_0 +
2N_4 + 2T_4 + 2N_{45} + 2T_5 + N_{35} + N'_{24} $, in the notation of
fibration $2$. Under the automorphism which is translation by $T_4$ in
fibration $2$, this divisor is taken to the elliptic divisor $N_2 +
N_{23} + 2T_2 + 2N_{12} + 2T_1 + 2N_{15} + 2T_5 + N_{35} + N'_{24}$,
which we shall work with.

We may take $T_0$ to be the zero section. This fibration is obtained
from fibration $2$ by a $2$-neighbor step. It has $D_8$, $D_4$ and
$A_3$ fibers. The trivial lattice has rank $17$ and discriminant $4
\cdot 4 \cdot 4 = 64$, and must therefore be all of $\NS(X)$. The
Mordell-Weil rank is $0$.

An elliptic parameter is given by 
$$
t_{15} = \frac{x_2}{4\,a\,(b-a)\,(c-1)\,t_2^2} \, .
$$

A Weierstrass equation for this fibration is given by

\begin{align*}
y^2 &= x^3 + x^2\,t\,\Big(a\,(b-a)\,(c-1)\,t^2 -(2\,a\,b\,c^2-b\,c^2-a\,c^2-a\,b^2\,c+2\,b^2\,c \\
& \quad -a^2\,b\,c-b\,c-a^2\,c+2\,a\,c-a\,b^2+2\,a^2\,b -a\,b)\,t + b^2\,c^2+2\,a\,b\,c^2\\
& \quad -3\,b\,c^2-a\,c^2+c^2-3\,a\,b^2\,c+2\,b^2\,c-a^2\,b\,c+3\,a\,b\,c-b\,c+a^2\,b^2-a\,b^2\Big) \\
& \quad + x\,(a-1)\,(b-1)\,b\,c\,(c-a)\,(c-b)\,(t-1)\,t^2\,(a\,c\,t+b\,t-2\,a\,t-2\,b\,c \\
& \quad -a\,c+2\,c+2\,a\,b-b)+ (a-1)^2\,(b-1)^2\,b^2\,c^2\,(c-a)^2\,(c-b)^2\,(t-1)^2\,t^3.
\end{align*}

It has the following reducible fibers.

\begin{center}
\begin{tabular}{|l|c|c|} \hline
Position & Reducible fiber & type \\
\hline \hline
$t=\infty$ & $N_2 + N_{23} + 2T_2 + 2N_{12} + 2T_1 + 2N_{15} + 2T_5 + N_{35} + N'_{24}$ & $D_8$\\
\hline
$t = 0$ & $N_4 + 2T_4 + N_{34} + N''_{13} + N'_{25}$ & $D_4$\\
\hline
$t = 1$ & $N_3 + T_{13} + N''_{14} + T^{(15)}_{13}$ & $A_3$ \\
\hline
\end{tabular}
\end{center}

\section{Fibration 15A}

This fibration corresponds to the elliptic divisor $N_2 + N_4 + 2T_0 +
2N_3 + 2T_{13} + 2N_{45} + 2T_5 + N_{35} + N'_{24}$. It is obtained
from fibration $2$ by a $2$-neighbor step. However, translation by
$-T_{13}$ in fibration $2$ turns this divisor into $N_2 + N_{23} +
2T_2 + 2N_{12} + 2T_1 + 2N_{15} + 2T_5 + N_{35} + N'_{24}$, which is
the elliptic divisor of fibration $15$. Therefore this fibration is
not new.

\section{Fibration 16}

This elliptic fibration comes from the divisor $N_{23} + 2T_2+ 3N_2 +
4T_0 + 2N_3 + 3N_4 + 2T_4 + N_{45}$ in the terminology of fibration
$2$. Translation by $T_4$ in that fibration takes this to the elliptic
divisor $N_{15} + 2T_1 + 3N_{12} + 4T_2 + 2N_{23} + 3N_2 + 2T_0 + N_3$
that we shall work with.

We may take $T_5$ to be the zero section. This fibration has an $E_7$
fiber, a $D_4$ fiber, and three $A_1$ fibers. The Mordell-Weil group
has rank $1$, and is generated by a section $T_{13}$ of height $4- 3/2
- 1/2 -1/2 - 1/2 = 1$. The sublattice of $\NS(X)$ generated by this
section, along with the trivial lattice, has rank $17$ and
discriminant $2 \cdot 4 \cdot 2^3 \cdot 1 = 64$ and must be all of
$\NS(X)$.

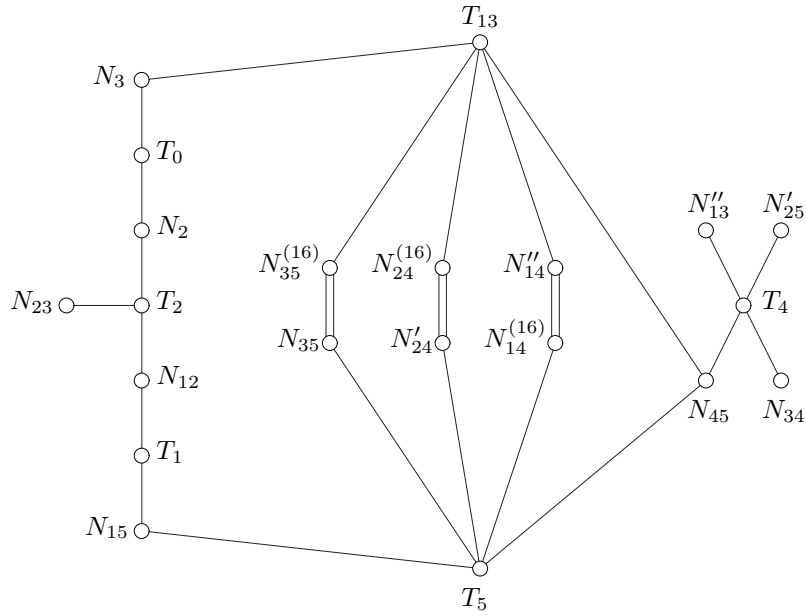
\begin{figure}[h]
\begin{center}
\setlength{\unitlength}{1cm}
\begin{picture}(10.0,8)(0.0,0.0)
\filltype{white}

\path(1.5,1.0)(1.5,2.0)
\path(1.5,2.0)(1.5,3.0)
\path(1.5,3.0)(1.5,4.0)
\path(1.5,4.0)(1.5,5.0)
\path(1.5,5.0)(1.5,6.0)
\path(1.5,6.0)(1.5,7.0)
\path(1.5,4.0)(0.5,4.0)

\path(3.95,3.5)(3.95,4.5)
\path(4.05,3.5)(4.05,4.5)
\path(5.45,3.5)(5.45,4.5)
\path(5.55,3.5)(5.55,4.5)
\path(6.95,3.5)(6.95,4.5)
\path(7.05,3.5)(7.05,4.5)

\path(9.5,4.0)(9.0,3.0)
\path(9.5,4.0)(9.0,5.0)
\path(9.5,4.0)(10.0,5.0)
\path(9.5,4.0)(10.0,3.0)

\path(6.0,0.5)(1.5,1.0)
\path(6.0,7.5)(1.5,7.0)

\path(6.0,0.5)(4.0,3.5)
\path(6.0,7.5)(4.0,4.5)

\path(6.0,0.5)(5.5,3.5)
\path(6.0,7.5)(5.5,4.5)

\path(6.0,0.5)(7.0,3.5)
\path(6.0,7.5)(7.0,4.5)

\path(6.0,0.5)(9.0,3.0)
\path(6.0,7.5)(9.0,3.0)

\put(1.5,1.0){\circle*{0.2}}
\put(1.5,2.0){\circle*{0.2}}
\put(1.5,3.0){\circle*{0.2}}
\put(1.5,4.0){\circle*{0.2}}
\put(1.5,5.0){\circle*{0.2}}
\put(1.5,6.0){\circle*{0.2}}
\put(1.5,7.0){\circle*{0.2}}
\put(0.5,4.0){\circle*{0.2}}

\put(4.0,3.5){\circle*{0.2}}
\put(5.5,3.5){\circle*{0.2}}
\put(7.0,3.5){\circle*{0.2}}

\put(4.0,4.5){\circle*{0.2}}
\put(5.5,4.5){\circle*{0.2}}
\put(7.0,4.5){\circle*{0.2}}

\put(9.0,3.0){\circle*{0.2}}
\put(9.0,5.0){\circle*{0.2}}
\put(10.0,3.0){\circle*{0.2}}
\put(10.0,5.0){\circle*{0.2}}
\put(9.5,4.0){\circle*{0.2}}

\put(6.0,0.5){\circle*{0.2}}
\put(6.0,7.5){\circle*{0.2}}

\put(0.75,0.95){$N_{15}$}
\put(1.7,1.95){$T_1$}
\put(1.7,2.95){$N_{12}$}
\put(1.7,3.95){$T_2$}
\put(1.7,4.95){$N_2$}
\put(1.7,5.95){$T_0$}
\put(0.85,6.95){$N_3$}
\put(-0.25,3.95){$N_{23}$}

\put(3.30,3.45){$N_{35}$}
\put(3.05,4.45){$N^{(16)}_{35}$}
\put(4.80,3.45){$N'_{24}$}
\put(4.55,4.45){$N^{(16)}_{24}$}
\put(6.05,3.45){$N^{(16)}_{14}$}
\put(6.30,4.45){$N''_{14}$}

\put(8.75,2.5){$N_{45}$}
\put(9.75,2.5){$N_{34}$}
\put(8.75,5.25){$N''_{13}$}
\put(9.75,5.25){$N'_{25}$}
\put(9.75,3.95){$T_4$}

\put(5.75, 0.0){$T_5$}
\put(5.75,7.75){$T_{13}$}

\end{picture}
\end{center}
\end{figure}

An elliptic parameter is given by 
$$
t_{16} = \frac{-x_2}{4\,a\,(b-1)\,(b-a)\,(c-1)\,(c-a)\,t_2}.
$$

A Weierstrass equation for this elliptic fibration is given by

\begin{align*}
y^2 &= x^3 - 4\,x^2\,t\,\Big((2\,a\,b\,c^2-b\,c^2-a\,c^2-a\,b^2\,c+2\,b^2\,c-a^2\,b\,c-b\,c \\
& \quad -a^2\,c+2\,a\,c-a\,b^2+2\,a^2\,b-a\,b)\,t -(2\,b\,c+a\,c-2\,c-2\,a\,b+b)\Big)  \\
& \quad -16\,x\,t^2\,(b\,t-a\,t-1)\,(a\,c\,t-a\,t-1)\,\Big((a-1)\,(b-1)\,b\,c\,(c-a)\,(c-b)\,t \\
& \quad  -(b^2\,c^2+2\,a\,b\,c^2-3\,b\,c^2-a\,c^2+c^2-3\,a\,b^2\,c+2\,b^2\,c-a^2\,b\,c+3\,a\,b\,c \\
& \quad -b\,c +a^2\,b^2-a\,b^2)\Big) \\
& \quad + 64\,(a-1)\,(b-1)\,b\,c\,(c-a)\,(c-b)\,t^3\,(b\,t-a\,t-1)^2\,(a\,c\,t-a\,t-1)^2.
\end{align*}

The equation of the section $T_{13}$ is
\begin{align*}
x &=  4\,(b\,t-a\,t-1)\,(a\,c\,t-a\,t-1), \\
y &= 8\,(b\,t-a\,t-1)\,(a\,c\,t-a\,t-1)\,(b\,c\,t-c\,t-a\,b\,t+a\,t+1).
\end{align*}

The reducible fibers are as follows:

\begin{center}
\begin{tabular}{|l|c|c|} \hline
Position & Reducible fiber & type \\
\hline \hline
$t=\infty$ & $N_{15} + 2T_1 + 3N_{12} + 4T_2 + 2N_{23} + 3N_2 + 2T_0 + N_3$ & $E_7$\\
\hline
$t = 0$ & $N_{45} + 2T_4 + N_{34} + N'_{25} + N''_{13}$ & $D_4$\\
\hline
$t = \frac{1}{b-a}$ & $N'_{24} + N^{(16)}_{24}$ & $A_1$ \\
\hline
$t = \frac{1}{a\,(c-1)}$ & $N_{35} + N^{(16)}_{35}$ & $A_1$ \\
\hline
$t = \frac{-1}{(b-1)\,(c-a)}$ & $N^{(16)}_{14} + N''_{14} $ & $A_1$ \\
\hline
\end{tabular}
\end{center}

\section{Fibration 16A}

This elliptic fibration arises from the elliptic divisor $N_{23} +
2T_2 + 3N_2 + 4T_0 + 2N_4 + 3N_3 + 2T_{13} + N_{13}$. It is obtained
from fibration $2$ by a $2$-neighbor step. But translation by
$-T_{13}$ transforms this divisor to $N_{13} + 2T_1 + 3N_{12} + 4T_2 +
2N_{23} + 3N_2 + 2T_0 + N_4$, which is in the orbit of the elliptic
divisor of fibration $16$ by $\Aut(D')$. Therefore this fibration is
not new.

\section{Fibration 17}

This fibration corresponds to the elliptic divisor $2T_4 + N_{14} +
N''_{13} + N_{34} + N'_{25}$, in the notation of fibration $2$. Under
the automorphism which is translation by $T_4$ in that fibration, this
divisor is taken to the elliptic divisor $2T_1 + N_{14} + N_{13} +
N''_{34} + N''_{25}$, which we shall work with.

We may take $T_4$ to be the zero section. This fibration is obtained
from fibration $2$ by a $2$-neighbor step. It has a $D_7$ and two
$D_4$ fibers. The trivial lattice has rank $17$ and discriminant $4
\cdot 4 \cdot 4 = 64$, and must therefore be all of $\NS(X)$. The
Mordell-Weil rank is $0$.

An elliptic parameter is given by 
$$
t_{17} = -\frac{x_2}{4(t_2-1)\,(b\,t_2-b+1)\,(c\,t_2-c+a)\,\big((a-1)\,(c-b)\,t_2 -(b-1)\,(c-a)\big)} - \frac{1}{t_2-1}.
$$

A Weierstrass equation for this fibration is given by

\begin{align*}
y^2 &= x^3 - x^2\,(t-1)\,t\,\Big((a-1)\,b\,c\,(c-b)\,t + (b\,c^2-a\,c^2+a\,b^2\,c-a^2\,b\,c\\
& \quad -3\,a\,b\,c-b\,c+2\,a^2\,c+2\,a\,c -a\,b^2+2\,a^2\,b+2\,a\,b-3\,a^2) \Big) \\
& \quad + a\,(b-1)\,(c-1)\,(b-a)\,(c-a)\,(b\,c-a\,c-c-a\,b-b+3\,a)\,(t-1)^2\,t^2\,x \\
& \quad + a^2\,(b-1)^2\,(b-a)^2\,(c-1)^2\,(c-a)^2\,(t-1)^3\,t^3.
\end{align*}

It has the following reducible fibers.

\begin{center}
\begin{tabular}{|l|c|c|} \hline
Position & Reducible fiber & type \\
\hline \hline
$t = \infty$ & $N_4 + N_3 + 2T_0 + 2N_2 + 2T_2 + N_{23} + T_{15} + T^{(17)}_{15}$ & $D_7$\\
\hline
$t = 0$ & $N_{14} + 2T_1 + N_{13} + N''_{34} + N''_{25}$ & $D_4$\\
\hline
$t = 1$ & $N_{45} + 2T_5 + N_{35} + N'_{24} + N^{(17)}_{45}$ & $D_4$ \\
\hline
\end{tabular}
\end{center}

\section{Fibration 18}

This fibration corresponds to the elliptic divisor $N_{23} + 2T_2 +
3N_2 + 4T_0 + 2N_4 + 3N_3 + 2T_{13} + N_{45}$, in the terminology of
fibration $2$. The automorphism which is the inverse of translation by
$T_{13}$ in that fibration transforms this divisor to $N_{15} + 2T_1 +
3N_{12} + 4T_2 + 2N_{23} + 3N_2 + 2T_0 + N_4$, which is the divisor we
shall work with.

We may take $T_5$ to be the zero section. This fibration is obtained
from fibration $2$ by a $2$-neighbor step. It has an $E_7$ fiber, an
$A_3$ fiber and five $A_1$ fibers, as well as a $2$-torsion
section. The trivial lattice and the torsion section span a sublattice
of $\NS(X)$ of rank $17$ and discriminant $2 \cdot 4 \cdot 2^5 /2^2 =
64$. Hence this sublattice is all of $\NS(X)$. The Mordell-Weil rank
of this fibration is $0$.

An elliptic parameter is given by 
$$
t_{18} = -\frac{x_2}{4\,a\,(b-1)\,(b-a)\,(c-1)\,(c-a)\,t_2} + \frac{t_2}{(b-1)\,(c-a)}.
$$

A Weierstrass equation for this fibration is given by

\begin{align*}
y^2 &= x\,\Big(x^2 + x\,\big(-4\,(2\,a\,b\,c^2-b\,c^2-a\,c^2-a\,b^2\,c+2\,b^2\,c-a^2\,b\,c-b\,c -a^2\,c+2\,a\,c\\
& \quad -a\,b^2+2\,a^2\,b-a\,b)\,t^2 + 8\,(b\,c+a\,c-c-a\,b+b-a)\,t-8\big)-16\,(b\,t-a\,t-1)\cdot \\
& \qquad (a\,c\,t-a\,t-1)\,(a\,c\,t-c\,t-a\,b\,t+b\,t-1)\,(b\,c\,t-a\,b\,t-1)\,(b\,c\,t-c\,t-1)\Big).
\end{align*}

It has the following reducible fibers.

\begin{center}
\begin{tabular}{|l|c|c|} \hline
Position & Reducible fiber & type \\
\hline \hline
$t=\infty$ & $N_{15} + 2T_1 + 3N_{12} + 4T_2 + 2N_{23} + 3N_2 + 2T_0 + N_4$ & $E_7$\\
\hline
$t = 0$ & $N_{45} + T_{13} + N''_{14} + T^{(18)}_{13}$ & $A_3$\\
\hline
$t = \frac{1}{b-a}$ & $N'_{24} + N^{(18)}_{24}$ & $A_1$ \\
\hline
$t = \frac{1}{a\,(c-1)}$ & $N_{35} + N^{(18)}_{35}$ & $A_1$ \\
\hline
$t = \frac{1}{c\,(b-1)}$ & $N^{(18)}_{34} + N_{34}$ & $A_1$ \\
\hline
$t = \frac{1}{(a-1)(c-b)}$ & $N^{(18)}_{13} + N''_{13}$ & $A_1$ \\
\hline
$t = \frac{1}{b \,(c-a)}$ & $N^{(18)}_{25} + N'_{25}$ & $A_1$ \\
\hline
\end{tabular}
\end{center}

The $2$-torsion section $T_4$ is $x=y=0$.

\section{Fibration 18A}

This fibration corresponds to the elliptic divisor $N_{12} + 2T_2 +
3N_2 + 4T_0 + 2N_3 + 3N_4 + 2T_4 + N_{45}$. It is, however, in the
orbit of the (modified) elliptic divisor $N_{15} + 2T_1 + 3N_{12} +
4T_2 + 2N_{23} + 3N_2 + 2T_0 + N_4$ of fibration $18$ under
$\Aut(D')$, so this fibration is not new.

\section{Fibration 19}

This fibration comes from the elliptic divisor $T_{12} + T_2 + 2N_{12}
+ 2T_1 + N_{13} + N_{14}$. It is obtained from fibration $1$ by a
$2$-neighbor step.

We may take $N'_{23}$ to be the zero section. This fibration has two
$D_5$ fibers and Mordell-Weil rank $5$. The divisors $N_{24}$,
$N_{25}$, $T_3$, $T_{14}$ and $N'_{45}$ have the intersection pairing

$$
\frac{1}{2}\left( \begin{array}{ccccc}
4 & 2 & 2 & 2 & 0 \\
2 & 6 & 3 & 1 & 2 \\
2 & 3 & 3 & 1 & 0 \\
2 & 1 & 1 & 3 & 0 \\
0 & 2 & 0 & 0 & 4 \\
\end{array} \right).
$$ 
The determinant of the height pairing matrix is $4$. Therefore the
sublattice of $\NS(X)$ generated by these sections, along with the
trivial lattice, has rank $17$ and discriminant $4^2 \cdot 4= 64$. So
it must be all of $\NS(X)$.

An elliptic parameter is given by 

\begin{align*}
t_{19} &= \frac{y_1 + 8\,(t_1-a)\,(t_1-b)\,(t_1-c)\,(t_1-a\,b)\,(t_1-b\,c)\,(t_1-a\,c)}{(a-1)\,(b-1)\,(c-a)\,(c-b)\, t_1^2 \big(x_1 -4\,(t_1-a)\,(t_1-b)\,(t_1-c)\,(t_1-a\,b\,c)\big)} \\
& \quad + \frac{2(t_1-a\,b)\,(t_1-c)}{(a-1)\,(b-1)\,(c-a)\,(c-b)\, t_1^2}.
\end{align*}

A Weierstrass equation for this elliptic fibration is given by

\begin{align*}
y^2 &= x^3+x^2\,t\,\Big(-18\,(a-1)\,a\,(b-1)\,b\,c\,(c-a)\,(c-b)\,t^2+36\,(3\,b^2\,c^2-2\,a\,b\,c^2 \\
& \quad -2\,b\,c^2+a\,c^2-2\,a\,b^2\,c-2\,b^2\,c+a^2\,b\,c+3\,a\,b\,c+b\,c+a^2\,c-2\,a\,c+a\,b^2 \\
& \quad -2\,a^2\,b+a\,b)\,t-72\Big) \\
& \quad -648\,(b-a)\,(c-1)\,x\,t^3\,\Big((a-1)\,a\,(b-1)\,b\,c\,(c-a)\,(c-b)\,(2\,b\,c-c-a\,b)\,t^2 \\
& \quad -2\,(3\,b^3\,c^3-a\,b^2\,c^3-4\,b^2\,c^3+a\,b\,c^3+b\,c^3-4\,a\,b^3\,c^2-b^3\,c^2+a^2\,b^2\,c^2 \\
& \quad +6\,a\,b^2\,c^2+b^2\,c^2+a^2\,b\,c^2-4\,a\,b\,c^2-a^2\,c^2+a\,c^2+a^2\,b^3\,c+a\,b^3\,c-4\,a^2\,b^2\,c \\
& \quad +a\,b^2\,c-a^3\,b\,c+3\,a^2\,b\,c-a\,b\,c+a^3\,b^2-a^2\,b^2)\,t+4\,(2\,b\,c-c-a\,b)\Big) \\
& \quad +2916\,(b-a)^2\,(c-1)^2\,t^4\,\Big((a-1)\,a\,(b-1)\,b\,c\,(c-a)\,(c-b)\,t^2 \\
& \quad -4\,(b-1)\,b\,c\,(c-a)\,t+4\Big)^2.
\end{align*}

This has the following reducible fibers.

\begin{center}
\begin{tabular}{|l|c|c|} \hline
Position & Reducible fiber & type \\
\hline \hline
$t=\infty$ & $T_{12} + T_2 + 2N_{12} + 2T_1 + N_{13} + N_{14}$ & $D_5$ \\
\hline
$t = 0$ & $T_5 + T_{15} + 2N_5 + 2T_0 + N_3 + N_4$  & $D_5$ \\
\hline
\end{tabular}
\end{center}

Generators for the Mordell-Weil group are in the table below.

\begin{center}
\begin{tabular}{|c|l|} \hline
Section & Equation \\
\hline \hline
$N'_{23}$ & $x = y = \infty$ \\
\hline
$N_{24}$ & $x= 18\,t\,\big((a-1)\,a\,(b-1)\,b\,c\,(c-a)\,(c-b)\,t^2 -2\,(b^2\,c^2-a\,b\,c^2-b\,c^2 $ \\
& \quad \quad $-a\,b^2\,c+3\,a\,b\,c+a^2\,c-a\,c-a^2\,b)\,t+4\big)$ \\
& $y = -54\,(c-1)\,t^2\,\big((a-1)\,a\,(b-1)\,b\,(b+a)\,c\,(c-a)\,(c-b)\,t^2$ \\& \quad \quad $ + 4\,a\,b\,(c^2+a\,b\,c-2\,b\,c-2\,a\,c+c+a\,b)\,t +4\,(b+a))$ \\
\hline
$N_{25}$ & $x = 18\,( (b-1)\,c\,(c-a)\,t - 2)\,((a-1)\,a\,b\,(c-a)^2\,(c-b)\,t^2 $ \\
& \quad \quad $-2\,(c-a)\,(b\,c-2\,a\,b+a)\,t - 4)/(c-a)^2$ \\
& $y = -54\,\big((a-1)\,a\,(b-1)\,b\,c\,(c-a)^4\,(c-b)\,(b\,c-a\,c+2\,a\,b-b-a)\,t^4
$ \\
& \quad \quad $-4\,a\,(b-1)\,(c-a)^3\,(2\,a\,b\,c^2-b\,c^2-a\,c^2-3\,a\,b^2\,c+2\,b^2\,c+a^2\,b\,c$ \\
& \quad \quad $+a\,b\,c-b\,c-a^2\,b^2+a\,b^2)\,t^3+ 4\,(c-a)^2\,(2\,a\,b\,c^2-b\,c^2-a\,c^2$ \\
& \quad \quad $-4\,a\,b^2\,c+2\,b^2\,c+2\,a^2\,b\,c+5\,a\,b\,c-b\,c-a^2\,c-3\,a\,c-6\,a^2\,b^2$ \\
& \quad \quad $+2\,a\,b^2+6\,a^2\,b-a\,b-a^2)\,t^2-16\,(c-a)\,(c-3\,a\,b+2\,a)\,t$ \\
& \quad \quad $-32\big)/(c-a)^3$ \\
\hline
$T_{3}$ & $x = -36\,(b-a)\,(c-1)\,(b\,c-a)\,t^2$ \\
& $y = 54\,(b-a)\,(c-1)\,t^2\,\big((a-1)\,a\,(b-1)\,b\,c\,(c-a)\,(c-b)\,t^2$ \\
& \quad \quad $ -4\,a\,(b-1)\,(c-a)\,t + 4\big)$ \\
\hline
$T_{14}$ & $x = 0$\\
& $y = -54\,(b-a)\,(c-1)\,t^2\,\big((a-1)\,a\,(b-1)\,b\,c\,(c-a)\,(c-b)\,t^2$ \\ & \quad \quad $-4\,(b-1)\,b\,c\,(c-a)\,t+4\big)$ \\
\hline
$N'_{45}$ & $x = 18\,t\,((a-1)\,a\,(c-b)\,t - 2)\,((b-1)\,b\,c\,(c-a)\,t - 2)$ \\
& $y = -54\,(b-a)\,(c-1)\,t^2\,\big((a-1)\,a\,(b-1)\,b\,c\,(c-a)\,(c-b)\,t^2 - 4\big)$\\
\hline
\end{tabular}
\end{center}

\section{Fibration 20}

This fibration corresponds to the elliptic divisor $N_{23} + N_{12} +
2T_2 + 2N_2 + 2T_0 + 2N_4 + 2T_4 + 2N_{45} + 2T_5 + N_{35} + N'_{24}$,
in the terminology of section $2$. The translation by the section
$T_4$ in that fibration takes it to the divisor $N_4 + N_3 + 2T_0 +
2N_2 + 2T_2 + 2N_{12} + 2T_1 + 2N_{15} + 2T_5 + N_{35} + N'_{24}$, and
we shall work with this elliptic divisor.

The corresponding elliptic fibration is obtained from fibration $2$ by
a $2$-neighbor step. We may take $T_4$ to be the zero section. This
fibration has a $D_{10}$ and four $A_1$ fibers. The Mordell-Weil group
has rank $1$ and is generated by the section $T_{13}$ of height
$1$. The resulting sublattice of $\NS(X)$ has rank $17$ and
discriminant $4 \cdot 2^4 \cdot 1= 64$, and must therefore be all of
$\NS(X)$.

An elliptic parameter is given by 
$$
t_{20} = \frac{x_2}{t_2^2} + 4\,(a-1)\,b\,c\,(c-b)\,t_2.
$$

 A Weierstrass equation for this fibration is

\begin{align*}
y^2 &= x^3 - 2\,x^2\,\Big(t^3 -2\,(3\,b^2\,c^2-2\,a\,b\,c^2-2\,b\,c^2-2\,a\,c^2+3\,c^2-2\,a\,b^2\,c-2\,b^2\,c \\
& \quad -2\,a^2\,b\,c+12\,a\,b\,c-2\,b\,c-2\,a^2\,c-2\,a\,c+3\,a^2\,b^2-2\,a\,b^2-2\,a^2\,b-2\,a\,b \\
& \quad +3\,a^2)\,t^2-16\,(a\,b^2\,c^4-b^2\,c^4-a^2\,b\,c^4+a\,b\,c^4+a^2\,b^3\,c^3-3\,a\,b^3\,c^3+b^3\,c^3 \\
& \quad -a^3\,b^2\,c^3+3\,a^2\,b^2\,c^3-a\,b^2\,c^3+b^2\,c^3-a^3\,b\,c^3 +3\,a^2\,b\,c^3-3\,a\,b\,c^3-a^3\,c^3 \\
& \quad +a^2\,c^3-a^2\,b^4\,c^2+a\,b^4\,c^2+a^3\,b^3\,c^2-a^2\,b^3\,c^2+3\,a\,b^3\,c^2-b^3\,c^2+3\,a^3\,b^2\,c^2 \\
& \quad -10\,a^2\,b^2\,c^2+3\,a\,b^2\,c^2-a^4\,b\,c^2+3\,a^3\,b\,c^2-a^2\,b\,c^2+a\,b\,c^2+a^3\,c^2-a^2\,c^2 \\
& \quad +a^2\,b^4\,c-a\,b^4\,c-3\,a^3\,b^3\,c+3\,a^2\,b^3\,c-a\,b^3\,c+a^4\,b^2\,c-a^3\,b^2\,c+3\,a^2\,b^2\,c \\
& \quad -a\,b^2\,c+a^4\,b\,c-3\,a^3\,b\,c+a^2\,b\,c+a^3\,b^3-a^2\,b^3-a^4\,b^2+a^3\,b^2)\,t \\
& \quad -32\,(a-1)\,a\,b\,(b-a)\,(c-1)\,c\,(c-b)\,(a\,b\,c^2 +b^2\,c+a^2\,b\,c +b\,c+a^2\,c\\
& \quad -2\,a\,c+a\,b^2-2\,a^2\,b+a\,b -2\,b\,c^2+a\,c^2-2\,a\,b^2\,c) \Big) \\
& \quad + x \,(t -4\,a\,(b-a)\,(c-1))\,(t + 4\,(a-1)\,a\,b\,(c-b))\,(t + 4\,(a-1)\,c\,(c-b)) \cdot \\
& \qquad (t -4\,b\,(b-a)\,(c-1)\,c)\,\Big(t^2 -4\,t\,(3\,b^2\,c^2-a\,b\,c^2-4\,b\,c^2-a\,c^2+3\,c^2 -4\,a\,b^2\,c\\
& \quad -b^2\,c-a^2\,b\,c+12\,a\,b\,c-b\,c-a^2\,c-4\,a\,c+3\,a^2\,b^2-a\,b^2-4\,a^2\,b-a\,b+3\,a^2) \\
& \quad -16\,(b\,c^2-a\,c^2+a\,b^2\,c-b^2\,c-a^2\,b\,c+a\,c+a^2\,b-a\,b) \cdot \\
& \qquad (a\,b\,c^2-b\,c^2-a\,b^2\,c+b\,c+a^2\,c-a\,c+a\,b^2-a^2\,b)\Big) \\
& \quad + 4\,(b-1)^2\,(c-a)^2\,(t - 4\,a\,(b-a)\,(c-1))^2\,(t + 4\,(a-1)\,a\,b\,(c-b))^2 \cdot \\
& \qquad (t + 4\,(a-1)\,c\,(c-b))^2\,(t  -4\,b\,(b-a)\,(c-1)\,c)^2.
\end{align*}

The reducible fibers are as follows:

\begin{center}
\begin{tabular}{|l|c|c|} \hline
Position & Reducible fiber & type \\
\hline \hline
$t=\infty$ & $N_4 + N_3 + 2T_0 + 2N_2 + 2T_2 + 2N_{12}+ 2T_1 $  & $D_{10}$\\
& $+ 2N_{15} + 2T_5 + N_{35} + N'_{24}$ & \\
\hline
$t = -4\,a\,b\,(a-1)\,(c-b)$ & $N_{34} + N^{(20)}_{34}$ & $A_1$ \\
\hline
$t = 4\,b\,c\,(b-a)\,(c-1)$ & $N''_{13} + N^{(20)}_{13}$ & $A_1$ \\
\hline
$t =  4\,a\,(b-a)\,(c-1)$ & $N^{(20)}_{14} + N''_{14}$ & $A_1$ \\
\hline
$t = -4\,c\,(a-1)\,(c-b)$ & $N'_{25} + N^{(20)}_{25}$ & $A_1$ \\
\hline
\end{tabular}
\end{center}

The equation for the section $T_{13}$, which is a generator of the Mordell-Weil group, is given by:
\begin{align*}
x &= 0, \\
y &= -2\,(b-1)\,(c-a)\,(t - 4\,a\,(b-a)\,(c-1))\,(t + 4\,(a-1)\,a\,b\,(c-b))\cdot \\
& \qquad (t + 4\,(a-1)\,c\,(c-b))\,(t  -4\,b\,(b-a)\,(c-1)\,c).
\end{align*}

\section{Fibration 20A}

This fibration corresponds to the elliptic divisor $N'_{24} + N_{35} +
2T_{14} + 2N_{14} + 2T_1 + 2N_{12} + 2T_2 + 2N_2 + 2T_0 + N_3 +
N_5$. It may be obtained from fibration $11$ by a $2$-neighbor
step. However, translation by the section $-T_{14}$ in fibration $11$
transforms this divisor into $N'_{24} + N_{35} + 2T_5 + 2N_5 + 2T_0 +
2N_2 + 2T_2 + 2N_{12} + 2T_1 + N_{13} + N_{14}$. This is related to
the (modified) elliptic divisor $N_4 + N_3 + 2T_0 + 2N_2 + 2T_2 +
2N_{12} + 2T_1 + 2N_{15} + 2T_5 + N_{35} + N'_{24}$ of fibration $20$
by an element of $\Aut(D')$, so this fibration is not new.

\section{Fibration 20B}

This fibration corresponds to the elliptic divisor $N_{12} + N_{23} +
2T_2 + 2N_2 + 2T_0 + 2N_3 + 2T_{13} + 2N_{45} + 2T_5 + N_{35} +
N'_{24}$. It is a $2$-neighbor step away from fibration $2$. But
translation by the section $-T_{13}$ turns this divisor into the
divisor $N_3 + N_4 + 2T_0 + 2N_2 + 2T_2 + 2N_{12} + 2T_1 + 2N_{15} +
2T_5 + N_{35} + N'_{24}$, which equals the (modified) elliptic divisor
of fibration $20$. Hence this fibration is not new.

\section{Fibration 21}

This fibration corresponds to the elliptic divisor $T_5 + N_{15} + T_1
+ N_{12} + T_2 + N_2 + T_0 + N_3 + T_{13} + N_{45}$. It is obtained
from fibration $2$ by a $2$-neighbor step.

We may take $N'_{24}$ as the zero section. This elliptic fibration has
an $A_9$ fiber, three $A_1$ fibers, a $2$-torsion section and
Mordell-Weil rank $3$. The divisors $N_4$, $N_{14}$ and $N_{23}$ have
the intersection pairing

$$
\frac{1}{5}\left( \begin{array}{ccc}
8 & 6 & 2 \\
6 & 12 & 4 \\
2 & 4 & 8 \\
\end{array} \right).
$$

The determinant of the height pairing matrix is $16/5$. This implies
that the sublattice of $\NS(X)$ generated by these sections, along
with the torsion sectionand the trivial lattice, has rank $17$ and
discriminant $10 \cdot 2^3 \cdot (16/5) /2^2 = 64$. Therefore it must
be all of $\NS(X)$.

An elliptic parameter is given by 
$$
t_{21} = \frac{y_2 - 8\,a\,(b-1)\,(c-1)\,(b-a)\,(c-a)\,t_2^2\,(t_2-1)}{2\,t_2\, \big(x_2 - 4\,a\,(b-a)\,(c-1)\,t_2^2 \big)}.
$$

A Weierstrass equation for this fibration is

\begin{align*}
y^2 &= x^3 + x^2\,\Big(t^4 + 2\,(2\,a\,b\,c^2-b\,c^2-a\,c^2-a\,b^2\,c+2\,b^2\,c-a^2\,b\,c-3\,a\,b\,c -b\,c\\
& \quad +2\,a^2\,c+2\,a\,c-a\,b^2+2\,a^2\,b+2\,a\,b-3\,a^2)\,t^2 + 8\,a\,(b-1)\,(b-a)\,(c-1)\,(c-a)\, t  \\
& \quad + (b^2\,c^4-2\,a\,b\,c^4+a^2\,c^4 -2\,a\,b^3\,c^3+4\,a^2\,b^2\,c^3+2\,a\,b^2\,c^3-2\,b^2\,c^3 -2\,a^3\,b\,c^3\\
& \quad -2\,a^2\,b\,c^3+2\,a\,b\,c^3+a^2\,b^4\,c^2 -2\,a^3\,b^3\,c^2+2\,a^2\,b^3\,c^2+4\,a\,b^3\,c^2+a^4\,b^2\,c^2\\
& \quad -2\,a^3\,b^2\,c^2-15\,a^2\,b^2\,c^2-2\,a\,b^2\,c^2+b^2\,c^2 +12\,a^3\,b\,c^2+6\,a^2\,b\,c^2-6\,a^3\,c^2\\
& \quad -2\,a^2\,b^4\,c+2\,a^3\,b^3\,c -2\,a^2\,b^3\,c-2\,a\,b^3\,c+6\,a^3\,b^2\,c+12\,a^2\,b^2\,c-6\,a^4\,b\,c\\
& \quad -10\,a^3\,b\,c-6\,a^2\,b\,c+4\,a^4\,c+4\,a^3\,c +a^2\,b^4-6\,a^3\,b^2+4\,a^4\,b+4\,a^3\,b-3\,a^4)\Big)\\
& \quad +16\,a\,b\,c\,(a-1)\,(b-1)\,(c-1)\,(b-a)\,(c-a)\,(c-b)\,(t+a\,b-a)\,(t+c-a) \cdot \\
& \qquad (t-b\,c+a\,c+b-a)\,x.
\end{align*}

This has the following reducible fibers.

\begin{center}
\begin{tabular}{|l|c|c|} \hline
Position & Reducible fiber & type \\
\hline \hline
$t=\infty$ & $T_5 + N_{15} + T_1 + N_{12} + T_2 + N_2 $ & $A_9$ \\
& $+ T_0 + N_3 + T_{13} + N_{45}$ & \\
\hline
$t = a-c$ & $N^{(21)}_{25} + N'_{25} $  & $A_1$ \\
\hline
$t=-a\,(b-1)$ & $N^{(21)}_{34} + N_{34}$ & $A_1$ \\
\hline
$t=(b-a)\,(c-1)$ & $T_{14} + N''_{13}$  & $A_1$ \\
\hline
\end{tabular}
\end{center}

Generators for the Mordell-Weil group are in the table below.

\begin{center}
\begin{tabular}{|c|l|} \hline
Section & Equation \\
\hline \hline
$N'_{24}$ & $x = y = \infty$ \\
\hline
$V^{(21)}$ & $x = y = 0$ \\
\hline
$N_4$ & $x = 4\,(a-1)\,(b-1)\,b\,(b-a)\,c\,(c-a)\,(c-b)$ \\
& $y = -4\,(a-1)\,(b-1)\,b\,(b-a)\,c\,(c-a)\,(c-b)\,(t^2 + 2\,a\,(c-1)\,t -(b\,c^2$ \\
& \qquad $-a\,c^2+a\,b^2\,c-2\,b^2\,c-a^2\,b\,c-a\,b\,c+b\,c+2\,a^2\,c+a\,b^2-a^2))$ \\
\hline
$N_{14}$ & $x =-4\,a\,(b-1)\,(c-1)\,(c-a)\,(t+b-a)^2 $ \\
& $y = -4\,a\,(b-1)\,(c-1)\,(c-a)\,(t+b-a)\,(t^3 + (b-a)\,t^2 -(b\,c^2-a\,c^2$ \\
& \qquad $+a\,b^2\,c-2\,b^2\,c-a^2\,b\,c+a\,b\,c+b\,c+a\,b^2-2\,a\,b+a^2)\,t -(b-a)\cdot$ \\
& \qquad $(2\,a\,b\,c^2-b\,c^2-a\,c^2-a\,b^2\,c-a^2\,b\,c+a\,b\,c+b\,c+a\,b^2-2\,a\,b+a^2))$ \\
\hline
$N_{23}$ & $x = 4\,(a-1)\,a\,(b-1)\,b\,(c-1)\,c\,(c-a)\,(c-b)$ \\
& $y = 4\,(a-1)\,a\,(b-1)\,b\,(c-1)\,c\,(c-a)\,(c-b)\,(t^2 + 2\,(b-a)\,t $ \\
& \qquad $+ 2\,a\,b\,c^2-b\,c^2-a\,c^2-a\,b^2\,c-a^2\,b\,c+a\,b\,c+b\,c+a\,b^2-2\,a\,b+a^2) $ \\
\hline
\end{tabular}
\end{center}

The class of the torsion section $V^{(21)}$ is $H - N_0 - N_1 - N_{24}
- N_{35} + T_{25}$: this is obtained by computing its intersections
with a basis of $NS(X) \otimes \Q$.

\section{Fibration 22}

This fibration corresponds to the elliptic divisor $T_5 + N_{45} + T_4
+ N_4 + T_0 + N_2 + T_2 + N_{12} + T_1 + N_{15}$. It is obtained from
fibration $2$ by a $2$-neighbor step.

We may take $N'_{24}$ to be the zero section. The elliptic fibration
has an $A_9$ fiber, an $A_1$ fiber, and Mordell-Weil rank $5$. The
sections $N_3, N_{13}, N_{23}, N_{34}$ and $T_{13}$ have the
intersection pairing below.

$$
\frac{1}{5}\left( \begin{array}{ccccc}
8 & 6 & 2 & 4 & 2 \\
6 & 12 & 4 & 8 & 4 \\
2 & 4 & 8 & 6 & 8 \\
4 & 8 & 6 & 12 & 6 \\
2 & 4 & 8 & 6 & 13 \\
\end{array} \right).
$$ 
The determinant of the height pairing matrix is $16/5$. Therefore the
sublattice of $\NS(X)$ generated by these sections, along with the
trivial lattice, has rank $17$ and discriminant $10 \cdot 2 \cdot
(16/5) = 64$. So it must be all of $\NS(X)$.

An elliptic parameter is given by 
$$
t_{22} = \frac{y_2}{t_2x_2}.
$$

A Weierstrass equation for this elliptic fibration is given by

\begin{align*}
y^2 &= x^3 + x^2\,\Big(t^4 + 8\,(2\,a\,b\,c^2-b\,c^2-a\,c^2-a\,b^2\,c+2\,b^2\,c-a^2\,b\,c-b\,c-a^2\,c+2\,a\,c \\
& \quad -a\,b^2+2\,a^2\,b-a\,b)\,t^2 + 16\,(b^2\,c^4-2\,a\,b\,c^4+a^2\,c^4-2\,a\,b^3\,c^3+4\,a\,b^2\,c^3-2\,b^2\,c^3 \\
& \quad +2\,a^3\,b\,c^3-2\,a^2\,b\,c^3+2\,a\,b\,c^3-2\,a^3\,c^3+a^2\,b^4\,c^2 -2\,a^3\,b^3\,c^2+4\,a^2\,b^3\,c^2\\
& \quad +a^4\,b^2\,c^2-2\,a^3\,b^2\,c^2-4\,a^2\,b^2\,c^2-2\,a\,b^2\,c^2+b^2\,c^2-2\,a^4\,b\,c^2+4\,a^3\,b\,c^2 +a^4\,c^2\\
& \quad -2\,a^2\,b^4\,c+2\,a^3\,b^3\,c-2\,a^2\,b^3\,c+2\,a\,b^3\,c+4\,a^2\,b^2\,c-2\,a\,b^2\,c-2\,a^3\,b\,c+a^2\,b^4 \\
& \quad -2\,a^2\,b^3+a^2\,b^2) \Big) -1024\,(a-1)\,a\,(b-1)\,b\,(b-a)\,(c-1)\,c\,(c-a)\,(c-b) \cdot\\
& \qquad (b\,c+a\,c-c-a\,b+b-a)\,t^2\,x \\
& \quad + 65536\,(a-1)^2\,a^2\,(b-1)^2\,b^2\,(b-a)^2\,(c-1)^2\,c^2\,(c-a)^2\,(c-b)^2\,t^2.
\end{align*}

This has the following reducible fibers.

\begin{center}
\begin{tabular}{|l|c|c|} \hline
Position & Reducible fiber & type \\
\hline \hline
$t=\infty$ & $T_5 + N_{45} + T_4 + N_4 + T_0 + N_2 + T_2 + N_{12} + T_1 + N_{15}$ & $A_9$ \\
\hline
$t = 0$ & $T_3 + N''_{14}$  & $A_1$ \\
\hline
\end{tabular}
\end{center}

Generators for the Mordell-Weil group are in the table below.

\begin{center}
\begin{tabular}{|c|l|} \hline
Section & Equation \\
\hline \hline
$N'_{24}$ & $x = y = \infty$ \\
\hline
$N_3$ & $x = 64\,(a-1)\,(b-1)\,b\,(b-a)\,c\,(c-a)\,(c-b)$ \\
& $y = -64\,(a-1)\,(b-1)\,b\,(b-a)\,c\,(c-a)\,(c-b)\,(t^2 -4\,(b\,c^2-a\,c^2$ \\
& \quad \quad $+a\,b^2\,c-2\,b^2\,c-a^2\,b\,c+b\,c+a^2\,c+a\,b^2-a\,b))$ \\
\hline
$N_{13}$ & $x = -16\,(a-1)\,a\,(c-1)\,(c-b)\,(t-2\,b\,c+2\,a\,c)\,(t+2\,b\,c-2\,a\,c)$ \\
& $y = -16\,(a-1)\,a\,(c-1)\,(c-b)\,(t^4  -4\,t^2\,(b^2\,c^2-4\,a\,b\,c^2+b\,c^2$ \\
& \quad \quad $+3\,a^2\,c^2-a\,c^2+a\,b^2\,c-2\,b^2\,c-a^2\,b\,c+2\,a\,b\,c +b\,c-a^2\,c$ \\
& \quad \quad $+a\,b^2-a\,b)-16\,(b-a)^2\,c^2\,(2\,a\,b\,c^2-b\,c^2-2\,a^2\,c^2+a\,c^2$ \\
& \quad \quad $-a\,b^2\,c+a^2\,b\,c-2\,a\,b\,c+b\,c+a^2\,c+a\,b^2-a\,b)
)$ \\
\hline
$N_{23}$ & $x = 64\,(a-1)\,a\,(b-1)\,b\,(c-1)\,c\,(c-a)\,(c-b)$ \\
& $y = 64\,(a-1)\,a\,(b-1)\,b\,(c-1)\,c\,(c-a)\,(c-b)\,(t^2 + 4\,(2\,a\,b\,c^2$ \\
& \quad \quad $-b\,c^2-a\,c^2-a\,b^2\,c-a^2\,b\,c+b\,c+a^2\,c+a\,b^2-a\,b))$ \\
\hline
$N_{34}$ & $x = -16\,(b-1)\,(b-a)\,c\,(t-2\,a\,c+2\,a\,b)\,(t+2\,a\,c-2\,a\,b)$ \\
& $y = 16\,(b-1)\,(b-a)\,c\,(t^4 + 4\,t^2\,(2\,a\,b\,c^2-b\,c^2-a^2\,c^2-a\,c^2$ \\
& \quad \quad $-a\,b^2\,c+a^2\,b\,c+2\,a\,b\,c+b\,c-a^2\,c-a^2\,b^2 -a\,b^2+2\,a^2\,b$ \\
& \quad \quad $-a\,b)-16\,a^2\,(c-b)^2\,(b\,c^2-a\,c^2-a\,b^2\,c+a^2\,b\,c+2\,a\,b\,c$ \\
& \quad \quad $-b\,c-a^2\,c-a\,b^2+a\,b))$ \\
\hline
$Q$ & $x = 0$, \quad $y = 256\,a\,b\,c\,(a-1)\,(b-1)\,(c-1)\,(b-a)\,(c-a)\,(c-b)\,t$ \\
\hline
\end{tabular}
\end{center}

Note that we replaced the section $T_{13}$ by a section $Q$ for which
$x_Q = 0$, since this section has a much simpler expression. It is
easily checked that $T_{13} = Q + N_{23}$ in the Mordell-Weil group,
so the five non-zero sections above do form a basis of the
Mordell-Weil group.

\section{Fibration 23}

This fibration corresponds to the elliptic divisor $N_{12} + N_{24} +
2T_2 + 2N_2 + 2T_0 + 2N_3 + 2T_3 + 2N_{35} + T_{14} + T'''_{14}$, in
the terminology of fibration $3$. We may translate by the torsion
section $T_3$ in that fibration to get a new elliptic divisor $N_3 +
N_5 + 2T_0 + 2N_2 + 2T_2 + 2N_{12} + 2T_1 + 2N_{14} + T_{14} +
T'''_{14}$, which we shall work with. The corresponding elliptic
fibration is obtained from fibration $3$ by a $2$-neighbor step.

We may take $T_3$ to be the zero section. This fibration has a $D_9$
fiber, six $A_1$ fibers, a $2$-torsion section $T_5$ and Mordell-Weil
rank $0$. The trivial lattice and the $2$-torsion section span a
sublattice of $\NS(X)$ of rank $17$ and discriminant $4 \cdot 2^6 /2^2
= 64$, which must therefore be all of $\NS(X)$.

An elliptic parameter is given by 
\begin{align*}
t_{23} &= \frac{x_3}{4\,t_3^2} - \frac{a\,c\,t_3}{4} + \frac{4\,(a-1)\,b\,(b-a)\,(c-1)\,(c-b)}{t_3} - \frac{1}{3} \bigg( 2\,a\,b\,c^2-b\,c^2-3\,a^2\,c^2+2\,a\,c^2 \\
& \qquad -a\,b^2\,c+2\,b^2\,c+2\,a^2\,b\,c -6\,a\,b\,c+2\,b\,c+2\,a^2\,c-a\,c+2\,a\,b^2-3\,b^2-a^2\,b+2\,a\,b \bigg).
\end{align*}

This fibration was studied in \cite{Ku}. The resulting Weierstrass
equation is independent of the level $2$ structure, and can be
expressed as
$$
y^2 = x^3 - 2\left(t^3 - \frac{I_4}{12}t + \frac{I_2I_4 - 3I_6}{108}\right) x^2 + \left( \left(t^3 - \frac{I_4}{12}t + \frac{I_2I_4 - 3I_6}{108}\right)^2 + I_{10} \left(t - \frac{I_2}{24} \right) \right) x.
$$

Here $I_2, I_4, I_6, I_{10}$ are the Igusa-Clebsch invariants of the
genus $2$ curve $C$, described in Section \ref{igcl}.

The reducible fibers are described in the following table.

\begin{center}
\begin{tabular}{|l|c|c|} \hline
Position & Reducible fiber & type \\
\hline \hline
$t=\infty$ & $N_3 + N_5 + 2T_0 + 2N_2 + 2T_2 + 2N_{12}$ & $D_9$ \\
& $+ 2T_1 + 2N_{14} + T_{14} + T'''_{14}$ & \\
\hline
$t = (a\,b\,c^2-2\,b\,c^2+a\,c^2+a\,b^2\,c$ & & \\
\qquad $+b^2\,c-2\,a^2\,b\,c+b\,c+a^2\,c$  & $N'_{25} + N^{(23)}_{25}$ & $A_1$ \\
\qquad $-2\,a\,c-2\,a\,b^2+a^2\,b+a\,b)/3$ & & \\
\hline
$t = (-2\,a\,b\,c^2+b\,c^2+a\,c^2+a\,b^2\,c$ & & \\
\qquad $-2\,b^2\,c+a^2\,b\,c+b\,c+a^2\,c$ & $N^{(23)}_{23} + N'_{23}$ & $A_1$ \\
\qquad $-2\,a\,c+a\,b^2-2\,a^2\,b+a\,b)/3$ & & \\
\hline
$t = (-2\,a\,b\,c^2+b\,c^2+a\,c^2+a\,b^2\,c$ & & \\
\qquad $+b^2\,c+a^2\,b\,c-2\,b\,c-2\,a^2\,c$ & $N'''_{15} + N^{(23)}_{15}$ & $A_1$ \\
\qquad $+a\,c-2\,a\,b^2+a^2\,b+a\,b)/3$ & & \\
\hline
$t = (a\,b\,c^2+b\,c^2-2\,a\,c^2-2\,a\,b^2\,c$ & & \\
\qquad $+b^2\,c+a^2\,b\,c-2\,b\,c+a^2\,c$ & $N_{34} + N^{(23)}_{34}$ & $A_1$ \\
\qquad $+a\,c+a\,b^2-2\,a^2\,b+a\,b)/3$ & & \\
\hline
$t = (a\,b\,c^2+b\,c^2-2\,a\,c^2+a\,b^2\,c$ & & \\
\qquad $-2\,b^2\,c-2\,a^2\,b\,c+b\,c+a^2\,c$ & $N^{(23)}_{13} + N'''_{13}$ & $A_1$ \\
\qquad $+a\,c+a\,b^2+a^2\,b-2\,a\,b)/3$ & & \\
\hline
$t = (a\,b\,c^2-2\,b\,c^2+a\,c^2-2\,a\,b^2\,c$ & & \\
\qquad $+b^2\,c+a^2\,b\,c+b\,c-2\,a^2\,c$ & $N^{(23)}_{45} + N_{45}$ & $A_1$ \\
\qquad $+a\,c+a\,b^2+a^2\,b-2\,a\,b)/3$ & & \\
\hline
\end{tabular}
\end{center}

The torsion section $T_5$ is given by $x = y = 0$.

\section{Fibration 24}

This fibration corresponds to the elliptic divisor $2T_{15} + 4N_5 +
6T_0 + 3N_3 + 5N_2 + 4T_2 + 3N_{12} + 2T_1 + N_{14}$. Translation by
the section $-T_{15}$ of fibration $11$ transforms this to the
elliptic divisor $2T_5 + 4N_5 + 6T_0 + 3N_3 + 5N_2 + 4T_2 + 3N_{12} +
2T_1 + N_{14}$, which we shall work with. It is obtained from
fibration $11$ by a $2$-neighbor step.

We may take $T_4$ to be the zero section. This fibration has an $E_8$
fiber, six $A_1$ fibers, and Mordell-Weil rank $1$. The section
$T_{14}$ has height $4 - 6(1/2) = 1$, and so the sublattice of
$\NS(X)$ it generates with the trivial lattice has rank $17$ and
discriminant $1 \cdot 2^6 \cdot 1 = 64$. Therefore it must be all of
$\NS(X)$.

An elliptic parameter is given by 
$$
t_{24} = \frac{(a-1)\,(c-b)\,x_{11}}{a\,b\,c\,(b-1)\,(c-1)\,(b-a)\,(c-a)} + 4\,t_{11} -2\,(b\,c-a\,c-c-a\,b+b+a).
$$

A Weierstrass equation for this elliptic fibration is

\begin{align*}
y^2 &= x^3-48\,I_4\,x^2 - 12\,x\,(3\,t^2-2\,I_2)\,(3\,t^2\,I_4  + 432\,t\,I_5 - 48\,I_6 +14\,I_2\,I_4) \\
& \quad + 16\,\Big(729\,I_5\,t^7 - 27\,(3\,I_6-I_2\,I_4)\,t^6 -1458\,I_2\,I_5\,t^5 + 54\,(3\,I_2\,I_6+6\,I_4^2 -I_2^2\,I_4)\,t^4  \\
& \quad + 972\,(32\,I_4+I_2^2)\,I_5\,t^3-36\,(96\,I_4\,I_6+3\,I_2^2\,I_6-20736\,I_5^2-20\,I_2\,I_4^2-I_2^3\,I_4)\,t^2 \\
& \quad  -216\,I_5\,(768\,I_6-160\,I_2\,I_4+I_2^3)\,t + 8\,(1152\,I_6^2-480\,I_2\,I_4\,I_6+3\,I_2^3\,I_6+50\,I_2^2\,I_4^2-I_2^4\,I_4) \Big).
\end{align*}

It has the following reducible fibers.

\begin{center}
\begin{tabular}{|l|c|c|} \hline
Position & Reducible fiber & type \\
\hline \hline
$t=\infty$ & $2T_{15} + 4N_5 + 3N_3 + 6T_0 + 5N_2+ 4T_2  $ & $E_8$\\
& $+ 3N_{12} + 2T_1 + N_{14}$ & \\
\hline
$t = -2\,(b\,c-a\,c-c-a\,b+b+a)$ & $N^{(11)}_{35} + N^{(24)}_{35}$ & $A_1$\\
\hline
$t = -2\,(b\,c+a\,c-c-a\,b-b+a)$ & $N^{(11)}_{24} + N^{(24)}_{24}$ & $A_1$ \\
\hline
$t = 2\,(b\,c-a\,c-c+a\,b-b+a)$ & $N_{34} + N^{(24)}_{34}$ & $A_1$ \\
\hline
$t = 2\,(b\,c-a\,c+c-a\,b-b+a)$ & $N'_{25} + N^{(24)}_{25}$ & $A_1$ \\
\hline
$t = 2\,(b\,c+a\,c-c-a\,b+b-a)$ & $N^{(24)}_{45} + N^{(11)}_{45}$ & $A_1$ \\
\hline
$t = -2\,(b\,c-a\,c+c+a\,b-b-a)$ & $N^{(24)}_{23} + N^{(11)}_{23}$ & $A_1$ \\
\hline
\end{tabular}
\end{center}

The section $T_{14}$ is given by
\begin{align*}
x &=  (3\,t^2-2\,I_2)^2/4, \\
y &= -(27\,t^6-54\,I_2\,t^4+(-576\,I_4+36\,I_2^2)\,t^2-27648\,I_5\,t+3072\,I_6-640\,I_2\,I_4-8\,I_2^3)/8. 
\end{align*}

As it stands, the Weierstrass equation above is almost invariant of
the level $2$ structure. The only problem is the appearance of the
square root $I_5$ of $I_{10}$ rather than the genuine covariant
$I_{10}$. This can easily be remedied, as follows. Letting
$$
\alpha = \frac{I_2^3 \cdot I_4}{I_{10}}, \qquad \beta = \frac{I_2^2 \cdot I_6}{I_{10}}, \qquad \gamma = \frac{I_2^5}{I_{10}}
$$ 
be a complete set of invariants of the genus $2$ curve, and $\mu =
I_5/I_2^2$, we can scale $t$ by $\mu$ and $x$ and $y$ by $\mu^4$ and
$\mu^6$ respectively, to make the convert the above equation into the
Weierstrass equation below, which depends solely on $\alpha, \beta,
\gamma$:

\begin{align*}
y^2 &= x^3-48\,\alpha\,\gamma\,x^2 -12\,x\,\gamma\,(3\,t^2 - 2\,\gamma)\,(3\,\alpha\,t^2+432\,\gamma\,t+(-48\,\beta+14\,\alpha)\,\gamma) \\
& \quad + 16\,\Big( 729\,\gamma^2\,t^7 -27\,(3\,\beta-\alpha)\,\gamma^2\,t^6 -1458\,\gamma^3\,t^5 + 54\,\gamma^2\,(3\,\beta\,\gamma-\alpha\,\gamma+6\,\alpha^2)\,t^4 \\
& \quad + 972\,\gamma^3\,(\gamma+32\,\alpha)\,t^3 -36\,\gamma^3\,(3\,\beta\,\gamma-\alpha\,\gamma-20736\,\gamma+96\,\alpha\,\beta-20\,\alpha^2)\,t^2 \\
& \quad -216\,\gamma^4\,(\gamma+768\,\beta-160\,\alpha)\,t + 8\,\gamma^4\,(3\,\beta\,\gamma-\alpha\,\gamma+1152\,\beta^2-480\,\alpha\,\beta+50\,\alpha^2) \Big).
\end{align*}

\section{Fibration 24A}

This elliptic fibration corresponds to the divisor $$(-4, -5, 0, 0, 0,
-5, 0, -3, -1, 0, 0, -2, -4, -1, -1, 0, 7),$$ and is therefore
obtained by a $2$-neighbor step from the original elliptic divisor of
fibration 11A. However, we used the translation by $-T_{15}$ in
fibration $2$ to obtain the new elliptic divisor for fibration
11A. Therefore, we apply this automorphism of $NS(X)$, which comes
from an automorphism of $X$, to our elliptic divisor and obtain the
elliptic divisor $$(-10, -3, 0, 0, 0, -5, 0, -1, -3, 0, 0, -2, -2, -3,
-1, 0, 9).$$ Now there are two explicit automorphisms in $\Aut(D')$
which takes the new elliptic divisor for the elliptic divisor of
fibration 11A to that of fibration 11. One of them is the permutation
of Weierstrass points, given by $(0,1,5) (2,4)$ in cycle
notation. This transformation converts our elliptic divisor to the
divisor class $$(-10, -5, 0, -1, -2, 0, 0, 0, 0, -3, -3, -2, -3, 0,
-1, 0, 9).$$ This elliptic divisor is related to that of fibration 11
by a $2$-neighbor step. Finally, in the elliptic fibration 11, we may
translate by the section $T_{14}$. This takes the elliptic divisor
above to $$(-7, -4, 0, 0, -3, 0, 0, -1, 0, -2, -2, -2, -3, 0, -1, -1,
7),$$ which is the (modified) elliptic divisor of fibration
$24$. Therefore this fibration is not new.

\section{Fibration 24B}

This elliptic fibration corresponds to the elliptic divisor $2T_1 +
4N_{12} + 6T_2 + 3N_{23} + 5N_2 + 4T_0 + 3N_3 + 2T_{13} + N_{45}$. It
may be obtained by a $2$-neighbor step from fibration $16$. However,
the translation by $-T_{13}$ in that fibration sends this to the
divisor $2T_0 + 4N_2 + 6T_2 + 3N_{23} + 5N_{12} + 4T_1 + 3N_{15} +
2T_5 + N_{45}$, which is in the $\Aut(D')$ orbit of the (modified)
elliptic divisor $2T_5 + 4N_5 + 6T_0 + 3N_3 + 5N_2 + 4T_2 + 3N_{12} +
2T_1 + N_{14}$ of fibration $24$. Therefore this fibration is not new.

\section{Fibration 24C}

The elliptic divisor $(-3, -2, -4, 0, -3, -1, 0, -1, 0, -1, -7, 0, 0,
0, -2, -2, 7)$ corresponding to this elliptic fibration is in the
$\Aut(D')$-orbit of the (modified) elliptic divisor $2T_5 + 4N_5 +
6T_0 + 3N_3 + 5N_2 + 4T_2 + 3N_{12} + 2T_1 + N_{14}$ of fibration
$24$. Hence this fibration is not new.

\section{Fibration 25}

This fibration corresponds to the elliptic divisor $T_5 + T_{15} +
2N_5 + 2T_0 + 2N_2 + 2T_2 + 2N_{12} + 2T_1 + N_{13} + N_{14}$. It is
obtained from fibration $11$ by a $2$-neighbor step. This is the most
complicated of the twenty-five elliptic fibrations, in terms of the
Weierstrass equation, the elliptic parameter and the sections, and it
is the one of highest Mordell-Weil rank.

We may take $N^{11}_{35}$ to be the zero section. This fibration has a
$D_9$ fiber and Mordell-Weil rank $6$. The divisors $N_{45}$,
$N'_{24}$, $N_{34}$, $N'_{23}$, $T_{14}$ and $N^{11}_{25}$ have the
intersection pairing

$$
\frac{1}{4}\left( \begin{array}{rrrrrr}
12 & 4 & 8 & 4 & 6 & 4 \\
4 & 12 & 8 & 4 & 2 & 4 \\
8 & 8 & 16 & 8 & 8 & 8 \\
4 & 4 & 8 & 12 & 6 & 4 \\
6 & 2 & 8 & 6 & 7 & 2 \\
4 & 4 & 8 & 4 & 2 & 12
\end{array} \right).
$$

The determinant of the height pairing matrix is $16$. This implies
that the sublattice of $\NS(X)$ generated by these sections along with
the identity section and the components of the reducible fiber has
rank $17$ and discriminant $64$. Therefore it must be all of $\NS(X)$.

An elliptic parameter is given by 
\begin{align*}
t_{25} &= \left( \frac{ 3\,(y_{11}+2\,(t_{11}+a\,c-a)\,(t_{11}-b\,c+a\,c)\,x_{11} ) }{ x_{11}-4\,t_{11}\,(t_{11}+a\,c-b)\,(t_{11}-b\,c+a\,c+a\,b-a)       \,(t_{11}-b\,c+a\,c+c-a) } \right) \\
& \qquad -6\,(a-1)\,(c-b)\,t_{11}  + 2\,(2\,a\,b\,c^2-b\,c^2-3\,a^2\,c^2+2\,a\,c^2-a\,b^2\,c+2\,b^2\,c \\
& \qquad +2\,a^2\,b\,c-3\,a\,b\,c-b\,c +2\,a^2\,c-a\,c-a\,b^2-a^2\,b+2\,a\,b).
\end{align*}

A Weierstrass equation for this fibration is given by

\begin{align*}
y^2 &= x^3 -6\,x^2\,\Big(t^3-3\,I_4\,t+(6\,I_6-2\,I_2\,I_4) \Big) + 69984\,x\,I_{10}\,(8\,t-I_2)  \\
& \quad + 419904\,I_{10}\,\Big(t^4+I_2\,t^3+6\,I_4\,t^2+(-48\,I_6+13\,I_2\,I_4)\,t+6\,I_2\,I_6+9\,I_4^2-2\,I_2^2\,I_4\Big)
\end{align*}
where $I_2, I_4, I_6, I_{10}$ are the Igusa-Clebsch invariants of the
curve $C$, described in section \ref{igcl}.

The $D_9$ fiber $T_{15} + T_5 + 2N_5 + 2T_0 + 2N_2 + 2T_2 + 2N_{12} +
2T_1 + N_{13} + N_{14}$ is at $t = \infty$, with $T_{15}$ being the
identity component.

Next, we write down six sections which generate the Mordell-Weil
lattice. In fact, we will not use the sections $P_1 = N_{45}$, $P_2 =
N'_{24}$, $P_3 = N_{34}$, $P_4 = N'_{23}$, $P_5 = T_{14}$ and $P_6 =
N^{11}_{25}$, but the linear combinations of them given by
\begin{align*}
Q_1 &= -P_3 + P_5 + P_6 \\
Q_2 &= P_4 - P_5 \\
Q_3 &= P_5 \\
Q_4 &= -P_1 + P_5 \\
Q_5 &= P_2 - P_3 + P_5 \\
Q_6 &= -P_3 + P_5
\end{align*}
where the addition refers to addition of points on the elliptic curve
over $k(t)$ and not addition in $\NS(X)$: to get the latter, one has
to add some linear combinations of fibral components which may be
easily deduced.

The advantage of these linear combinations is that these sections now
all have height $7/4$, which is the smallest norm in the Mordell-Weil
lattice (in fact, up to sign they are all the minimal vectors of the
Mordell-Weil lattice), while still generating it.

We now list just the $x$-coordinates of these sections, for brevity.

\begin{center}
\begin{tabular}{|c|l|} \hline
Section & Formula for $x(t)$ \\
\hline \hline
$Q_1$ & $-648\,(a-1)\,a\,(b-1)\,b\,(b-a)\,(c-1)\,c\,(c-a)\,(c-b)\,(b\,c+a\,c-c-a\,b-b+a)$ \\
\hline
$Q_2$ & $-648\,(a-1)\,a\,(b-1)\,b\,(b-a)\,(c-1)\,c\,(c-a)\,(c-b)\,(b\,c-a\,c+c+a\,b-b-a)$ \\
\hline
$Q_3$ & $-648\,(a-1)\,a\,(b-1)\,b\,(b-a)\,(c-1)\,c\,(c-a)\,(c-b)\,(b\,c-a\,c-c-a\,b+b+a)$ \\
\hline
$Q_4$ & $648\,(a-1)\,a\,(b-1)\,b\,(b-a)\,(c-1)\,c\,(c-a)\,(c-b)\,(b\,c+a\,c-c-a\,b+b-a)$ \\
\hline
$Q_5$ & $648\,(a-1)\,a\,(b-1)\,b\,(b-a)\,(c-1)\,c\,(c-a)\,(c-b)\,(b\,c-a\,c+c-a\,b-b+a)$  \\
\hline
$Q_6$ & $648\,(a-1)\,a\,(b-1)\,b\,(b-a)\,(c-1)\,c\,(c-a)\,(c-b)\,(b\,c-a\,c-c+a\,b-b+a)$ \\
\hline
\end{tabular}
\end{center}

\section{Appendix A}\label{appa}

In this appendix we describe some basic computations which are needed
in the main body of the paper, to convert between different elliptic
fibrations. We learned the technique of ``elliptic hopping'' using
$2$-neighbors from Noam Elkies.

\subsection{Global sections of some divisors on an elliptic curve}

We describe, for completeness, the global sections of some simple line
bundles on an elliptic curve $E$. In this paper, we use these results
in the following manner: we have a K3 surface $X$ over a field $K$
with an elliptic fibration over $\Proj^1_K$, with a zero section $O$,
and perhaps some other sections $P,Q$ etc. Let $F$ be the class of a
fiber. Then for an effective divisor $D = mO + nP + kQ + G$, where $G$
is the class of an effective vertical divisor (i.e. fibers and
components of reducible fibers), we would like to compute the global
sections of $\sO_X(D)$. Any such global section gives a section of the
generic fiber, which is an elliptic curve $E$ over $K(t)$. Therefore,
if we have a basis $\{s_1,\dots, s_r\}$ of the global sections of $mO
+ nP + kQ$ over $K(t)$, we can assume that any global section of
$\sO_X(D)$ must be of the form
$$
b_1(t) s_1 + \dots + b_r(t) s_r
$$ 
with $b_i(t) \in K(t)$. We can then use the information from $G$,
which gives us conditions about the zeros and poles of the functions
$a_i$, to find the linear space cut out by $H^0(X, \sO_X(D))$.

We also describe what happens when we make the transformation to the
new elliptic parameter, and how to convert the resulting genus $1$
curve to a double cover of $\Proj^1$ branched at $4$ points. Later we
will describe how to convert these to Weierstrass form.

First, we consider the case $D = 2O$. Then $1$ and $x$ are a basis for
the global sections of $\sO_E(D)$. Therefore, for an elliptic K3
surface $X$ given by a minimal proper model of
$$
y^2 = x^3 + a_2(t) x^2 + a_4(t)x + a_6(t),
$$ 
with $a_i \in K[t]$ of degree at most $2i$, and an elliptic divisor
$F' = 2O + G$ with $G$ effective and vertical, we obtain two global
sections $1$ and $a(t) +b(t)x$ for some fixed $a(t), b(t)$. The ratio
of these two sections gives the new elliptic parameter $u$. Therefore
we set $x = (u - b(t))/a(t)$ and substitute into the Weierstrass
equation, to obtain
$$
y^2 = g(t,u)
$$ 
Since the generic fiber of this surface over $\Proj^1_u$ is a curve
of genus $1$ (by construction of $F'$), we must have after absorbing
square factors into $y^2$, that $g$ is a polynomial of degree $3$ or
$4$ in $t$.

Next, consider $D = O + P$ where $P = (x_0, y_0)$ is not a $2$-torsion
section. Then $1$ and $(y+y_0)/(x-x_0)$ are global sections of
$\sO_E(D)$. For an elliptic K3 surface as above and $F' = O + P + G$,
we obtain global sections $1$ and $a(t) + b(t)
(y+y_0)/(x-x_0)$. Setting the latter equal to $u$ as before, we solve
for $y$ to get
$$
y = \frac{(u-a(t))(x-x_0)}{b(t)} - y_0.
$$
Substituting into the Weierstrass equation we get
$$
\left(\frac{(u-a(t))(x-x_0)}{b(t)} - y_0\right)^2 = x^3 + a_2(t) x^2 + a_4(t)x + a_6(t).
$$ 
Note that when $x = x_0$, the left and right hand sides of this
equation both evaluate to $y_0^2$. Therefore their difference is a
multiple of $(x-x_0)$ in the polynomial ring $K(t)[x]$, which may be
cancelled to give an equation $g(x,t,u) =0$ which is quadratic in
$x$. Completing the square converts it to an equation
$$
x^2 = h(t,u)
$$ 
and, as before, we can argue that after absorbing square factors
into the left hand side, $h$ must be cubic or quartic in $t$.

The final case we need to consider is $D = O + T$, where $T$ is a
$2$-torsion section, which we may take to be $(0,0)$, while taking an
equation of $X$ of the form
$$
y^2 = x^3 + a_2(t) x^2 + a_4(t)x.
$$ 
Now $1$ and $y/x$ are sections of $\sO_E(D)$. For an elliptic K3
surface and $F' = O + T + G$, we obtain global sections $1$ and $a(t)
+ b(t) y/x$. Setting the latter equal to $u$, we obtain
$$
y = (u-a(t))x/b(t).
$$
Substituting into the Weierstrass equation results in
$$
\left(\frac{(u-a(t)) x}{b(t)} \right)^2 = x^3 + a_2(t) x^2 + a_4(t)x
$$

We cancel a factor of $x$ from both sides, obtaining a quadratic
equation, and proceed as in the previous case.

In fact, these are all the cases we need to consider. For if we have
an elliptic divisor $F'$ on an elliptic K3 surface with fiber class
$F$ satisfying $F' \cdot F = 2$, then decomposing $F'$ into horizontal
and vertical components $F' = F'_h + F'_v$, we see that $F'_h \cdot F
+ F'_v \cdot F = F' \cdot F = 2$. But $F'_v$ consists of components of
fibers, by definition, so its intersection with $F$ is zero. Therefore
$F'_h \cdot F = 2$. The only possibilities are $2P$ and $P + Q$ for
sections $P,Q$. We can translate the first one using an automorphism
to get $2O$, and the second to get $O + T$ or $O + P$ depending on
whether the class of the section $[P-Q]$ is $2$-torsion or not.

\subsection{Conversion to Weierstrass form}

We recall how to convert a genus $1$ curve over a field $K$ with
affine model
$$
y^2 = a_4 x^4 + a_3 x^3 + a_2 x^2 + a_1 x + a_0
$$ 
to Weierstrass form, given a $K$-rational point on the curve
(references for this method are \cite{Conn, Ru}). First, if the given
point is at $\infty$, then $a_4 = \alpha^2$ is a square, and we may
change coordinates to $\eta = y/x^2$ and $\xi = 1/x$ to get a point
$(0,\alpha)$ on
$$
\eta^2 = a_0 \xi^4 + a_1 \xi^3 + a_2 \xi^2 + a_3 \xi + a_4.
$$ 
So we may assume that the $K$-rational point is an affine point,
and by translating $x$ by $x_0$ in the equation above, we may assume
that $(0,\beta)$ is a rational point.

Set $x = u/v$ and $y = (Av^2 + Buv + Cu^2 + Du^3)/v^2$ in the
equation.  We get
$$
(Av^2 + Buv + Cu^2 + Du^3)^2 = a_4 u^4 + a_3 u^3 v + a_2 u^2 v^2 + a_1 u v^3 + \beta^2 v^4.
$$

We successively solve for the coefficients, so that the difference of
the two sides is divisible by $u^3$.

In other words, we set
$$
A = \beta,\quad B = \frac{a_1}{2 \beta}, \quad C = \frac{4 a_2 \beta^2 - a_1^2}{8 \beta^3}.
$$ 
Finally, we divide the resulting equation by $u^3$, and choose $D =
-2\beta$ so that the coefficients of $v^2$ and $u^3$ add to zero. This
leads us to a Weierstrass equation of the form
$$
v^2 + b_1 uv + b_3 v = u^3 + b_2 u^2 + b_4 u
$$
which we may then convert to Weierstrass form.

\section{Appendix B}\label{appb}

In this section we explicitly describe the convex polytope
$\mathcal{P}$ which is used in Section \ref{polytope}. The point
$(\alpha_0, \alpha_1, \dots, \alpha_{45}) \in \R^{16}$ of
$\mathcal{P}$ corresponds to the divisor $D = H - \sum \alpha_\mu
N_\mu$ of $NS(X)$. These divisors must satisfy the linear constraints
$D \cdot r' \geq 0$, for $316$ different $r' \in S \otimes \Q$. We may
scale the $r'$ so it lies in $S$, which is integral, and therefore we
may assume the inequalities are of the form $\sum c_\mu \alpha_\mu + d
\geq 0$, with $c_\mu$ and $d$ integers with no common divisor greater
than $1$.

The polytope $\mathcal{P}$ has $30124888$ vertices, and it is beyond
our current limits of computation to enumerate these directly, within
a reasonable amount of time. Hence, we must make use of the symmetry
group of the polytope. The software \verb,SymPol, \cite{RS} outputs a
list of $2961$ vertices $\mathcal{V}_0$ of $\mathcal{P}$, which are
exhaustive modulo the action of the (affine) symmetry group of
$\mathcal{P}$ in $\R^{16}$. However, the affine symmetry group of
$\mathcal{P}$ is an index $2$ subgroup of the full symmetry group of
$D'$, because the switch $\sigma \in Aut(D')$ is not an affine
symmetry of $\mathcal{P}$. Taking the full symmetry group into
account, we obtain a list $\mathcal{V}$ of $1492$ vertices. This
computation using \verb,Sympol, requires only about $33$ minutes of
computation time on a $3$ gigahertz processor, using $5.4$ gigabytes
of memory. We now describe an alternate method to just verify the
correctness of this list of vertices, which is simpler and less memory
intensive (note that, however, this method relies on having the
putative list of vertices). We thank Henry Cohn for suggesting this
approach.

First, we may dualize the problem, to obtain a polytope $\mathcal{Q}$
which has only $316$ vertices, and a symmetry group $G$ of size $6!
\cdot 2^5 = 23040$. We have a list of $1492$ faces of this polytope,
and we would like to assert that it is complete up to symmetry. Since
the polytope is convex and hence connected, it is enough to verify
that for each of these $1492$ faces $F$, and each codimension $1$
facet $F'$ of $F$, that the face lying on the opposite side of $F'$ is
in the orbit of one of the $1492$ faces under the group $G$. There are
$13394$ such opposite faces, and we check whether each of these lies
in such an orbit.

The computer files for checking these calculations are available from
the \texttt{arXiv.org} e-print archive. To access the auxiliary files,
download the source file for the paper. This will produce not only the
\LaTeX{} file for this paper, but also the computer code.  The file
\texttt{README.txt} gives an overview of the various computer files
involved in the verification.

\section{Acknowledgements}

We thank Henry Cohn, Igor Dolgachev, Noam Elkies, Torsten Ekedahl,
Steve Kleiman, Shigeyuki Kondo, Barry Mazur, Matthias Sch\"utt and
Tetsuji Shioda for many helpful suggestions, and Thomas Rehn and
Achill Sch\"urmann for help with their program SymPol. We are
especially thankful to Masato Kuwata and the anonymous referees for a
careful reading of the first draft of this paper, and for many
valuable corrections to the content and suggestions to improve the
presentation. The computer algebra systems PARI/gp, Maxima, and Magma
were used in the calculations for this paper.

\end{document}